\documentclass{article}
\usepackage[a4paper,bindingoffset=0.2in,%
            left=0.5in,right=0.5in,top=0.5in,bottom=0.75in,%
            footskip=.25in,
            ]{geometry}  
\usepackage{amsmath,amssymb}
\usepackage{amsthm}%
\usepackage{tikz}
\usepackage{subcaption}
\usepackage{wrapfig}
\usepackage{mathtools}
\usepackage[inline]{enumitem}
\usepackage[vlined]{algorithm2e}
\usepackage{bm}
\usepackage{tabto}
\usepackage{microtype}
\usepackage{float} 
\usepackage[section]{placeins} 
\usepackage{makecell}
\usepackage{diagbox}
\usepackage{needspace} 
\usepackage[colorlinks,citecolor=blue]{hyperref}
\usepackage{doi} 

\providecommand{\keywords}[1]
{
  \small	
  \textbf{\textit{Keywords---}} #1
}

\usetikzlibrary{shapes,positioning,fit,matrix,arrows,arrows.meta}
\usetikzlibrary{decorations.pathmorphing,decorations.pathreplacing}
\usetikzlibrary{calc}

\tikzset{
dot/.style={draw,fill,circle,inner sep = 0pt,minimum size = 3pt},
dotBig/.style args={#1,#2}{draw,fill=white,circle,inner sep=0pt,node contents=\textcolor{#2}{#1},minimum size=4.1mm},
vcolour/.style={draw,inner sep=1.5pt,font=\scriptsize,label distance=2pt},
subgraph/.style={draw,ellipse,minimum width=1.25cm,minimum height=1.25cm},
subgraphHoriz/.style={draw,ellipse,minimum width=1.5cm,minimum height=1.25cm},
}

\SetKwProg{Fn}{Subroutine}{}{}           
\newtheorem{theorem}{Theorem}%
\newtheorem{corollary}{Corollary}%
\newtheorem{lemma}{Lemma}%
\newtheorem{observation}{Observation}

\def\etal.{et\penalty50\ al.}

\title{Complexity of Restricted Star Colouring\thanks{First author is supported by SERB (DST), MATRICS scheme MTR/2018/000086\\
A preliminary version of the paper appeared in CALDAM 2020 \cite{shalu_cyriac}.\\
Please cite this article as: Shalu M.A. and C. Antony, The complexity of restricted star colouring, Discrete Applied Mathematics (2021),
\url{https://doi.org/10.1016/j.dam.2021.05.015}.\\
0166-218X/ © 2021 Elsevier B.V. All rights reserved.}}
\author{Shalu M.\ A., and Cyriac Antony\\
Indian Institute of Information Technology, Design \& Manufacturing\\
(IIITDM) Kancheepuram, Chennai, India\\
\texttt{\{shalu,mat17d001\}{\fontfamily{ptm}\selectfont @}iiitdm.ac.in}\\
}
\date{}

\begin{document}
\maketitle              
\begin{abstract}
Restricted star colouring is a variant of star colouring introduced to design heuristic algorithms to estimate sparse Hessian matrices. For \( k\in\mathbb{N} \), a \( k \)-restricted star colouring (\( k \)-rs colouring) of a graph \( G \) is a function \mbox{\( f:V(G)\to\{0,1,\dots,k-1\} \)} such that (i)~\( f(x)\neq f(y) \) for every edge \( xy \) of \( G \), and (ii)~there is no bicoloured 3-vertex path (\( P_3 \)) in \( G \) with the higher colour on its middle vertex. 
We show that for \( k\geq 3 \), it is NP-complete to test whether a given planar bipartite graph of maximum degree \( k \) and arbitrarily large girth admits a \( k \)-rs colouring, and thereby answer a problem posed by Shalu and Sandhya (Graphs and Combinatorics, 2016). In addition, it is NP-complete to test whether a 3-star colourable graph admits a 3-rs colouring. 
We also prove that for all \( \epsilon>0 \), the optimization problem of restricted star colouring a 2-degenerate bipartite graph with the minimum number of colours is NP-hard to approximate within \( n^{\frac{1}{3}-\epsilon} \). 
On the positive side, we design (i)~a linear-time algorithm to test 3-rs colourability of trees, and (ii)~an \( O(n^3) \)-time algorithm to test 3-rs colourability of chordal graphs.
\end{abstract}

%
%
%
%
\section{Introduction}\label{sec:introduction}
Many large scale optimization problems involve a multi-variable function \( f \). 
The second-order approximation of \( f \) using Taylor series expansion requires an estimation of the Hessian matrix of \( f \). Vertex colouring of graphs and its variants have been found immensely useful as models for estimation of sparse Hessian and Jacobian matrices (see \cite{gebremedhin1} for a survey). 
To compute a compressed form of a given sparse matrix, Curtis \etal.~\cite{curtis} partitioned the set of columns of the matrix in such a way that columns that do not share non-zero entries along the same row are grouped together. By exploiting symmetry, Powell and Toint~\cite{powell_toint} designed a heuristic algorithm for partitioning columns of a sparse Hessian matrix implicitly using restricted star colouring. Restricted star colouring was first studied as a variant of star colouring (see next para.\ for definitions). It was also studied independently in the guise of unique superior colouring \cite{karpas,almeter}, a generalization of ordered colouring (see Section~\ref{sec:related} for the definition of ordered colouring). 
Therefore, restricted star colouring (abbreviated \emph{rs colouring}) is intermediate in strength between star colouring and ordered colouring. Ordered colouring is also known as vertex ranking, and has several theoretical as well as practical applications (see \cite{iyer,leiserson,liu} and \cite[Chapter~6]{kubale}). It is worth mentioning that every distance-two colouring is an rs colouring. 
Note that restricted star colouring appears unnamed in \cite{gebremedhin1} and under the name independent set star partition in \cite{shalu_sandhya}. 


This paper is an extension of the work in \cite{shalu_cyriac}. 
In this paper, only vertex colourings of finite simple undirected graphs are considered. 
A \emph{\( k \)-colouring} of a graph \( G \) is a function \( f:V(G)\to \{0,1,\dots,k-1\} \) such that \( f(x)\neq f(y) \) whenever \( xy \) is an edge in \( G \). 
Let us denote the colour class \( f^{-1}(i) \) by \( V_i \). 
A \( k \)-colouring \( f \) of \( G \) is a \emph{\( k \)-star colouring} of \( G \) if there is no \( P_4 \) in \( G \) (not necessarily induced) bicoloured by \( f \). 
In other words, for \( i\neq j \), every component of \( G[V_i\cup V_j] \) is a star. A \( k \)-colouring of \( G \) is a \emph{\( k \)-restricted star colouring} if \( G \) contains no bicoloured \( P_3 \) with the higher colour on its middle vertex (i.e., no path \( x,y,z \) with \( f(y)>f(x)=f(z) \); see Figure~\ref{fig:dart} for an example). 
That is, for \( i<j \), each vertex in \( V_j \) has at most one neighbour in \( V_i \) (see Figure~\ref{fig:dart redrawn}). In other words, every non-trivial component of \( G[V_i\cup V_j] \) is a star with its centre in \( V_i \) for \( i<j \). Hence, every \( k \)-restricted star colouring is a \( k \)-star colouring; but the converse is not true \cite{shalu_sandhya}.

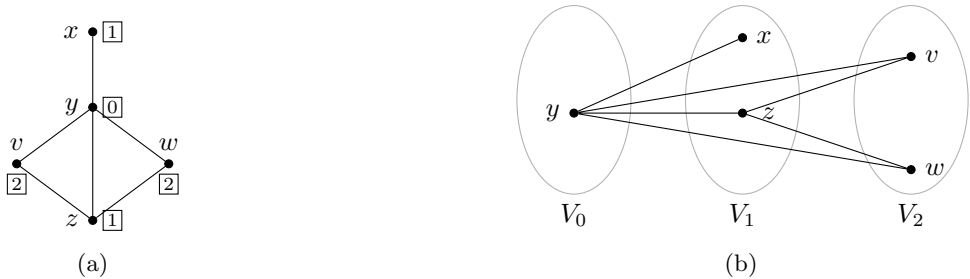
\begin{figure}[hbt]
\centering
\begin{subfigure}[b]{0.3\textwidth}
\centering
\begin{tikzpicture}
\node(x)[dot][label=left:\( x \)][label={[vcolour]right:1}]{};
\draw (x)--++(0,-1) node(y)[dot][label=left:\( y \)][label={[vcolour]right:0}]{}--++(0,-1.5) node(z)[dot][label=left:\( z \)][label={[vcolour]right:1}]{};
\draw (y)--+(-1,-0.75) node(v)[dot][label=\( v \)][label={[vcolour]below:2}]{}--(z);
\draw (y)--+(1,-0.75) node(w)[dot][label=\( w \)][label={[vcolour]below:2}]{}--(z);
\end{tikzpicture}
\caption{}
\label{fig:dart}
\end{subfigure}%
\begin{subfigure}[b]{0.65\textwidth}
\centering
\begin{tikzpicture}
\tikzset{
Vset/.style={draw=black!30,ellipse,minimum width=1.5cm,minimum height=2.5cm},
}
\node(V0) [Vset][label=below:\( V_0 \)]{};
\node(V1) [Vset][right=20pt of V0][label=below:\( V_1 \)]{};
\node(V2) [Vset][right=20pt of V1][label=below:\( V_2 \)]{};
\node(y) at (V0) [dot][label=left:\( y \)][yshift=-5pt]{};
\node(z) at (V1) [dot][label={[xshift=2pt]right:\( z \)}][yshift=-5pt]{};
\path (V1)+(0,1) node(x)[dot][label=right:\( x \)][yshift=-5pt]{};
\path
(V2)+(0,0.75) node(v) [dot][label=right:\( v \)][yshift=-5pt]{}
(V2)+(0,-0.75) node(w) [dot][label=right:\( w \)][yshift=-5pt]{};
\draw (x)--(y)--(z)--(v)--(y)--(w)--(z);
\end{tikzpicture}
\caption{}
\label{fig:dart redrawn}
\end{subfigure}%
\caption{(a)~a 3-rs colouring \( f \) of the dart graph, (b)~colour classes under \( f \)}
\end{figure}

Restricted star colouring is studied in the literature mainly for minor-excluded graph families \cite{karpas}, trees \cite{karpas}, \( q \)-degenerate graphs \cite{karpas}, and degree-bounded graph families \cite{almeter}. The rs chromatic number of a graph \( G \), denoted by \( \chi_{rs}(G) \), is the least integer \( k \) such that \( G \) is \( k \)-rs colourable. 
For some classes such as planar graphs and trees, only asymptotic bound for the rs chromatic number is known \cite{karpas}. 
On the other hand, there are classes for which we can compute the rs chromatic number in polynomial time using a simple formula; for instance, \( \chi_{rs}(Q_d)=d+1 \) for the hypercube \( Q_d \) \cite{almeter}. 
It is also known that \( \chi_{rs}(G)\leq 4\alpha(G) \) for every graph \( G \) of girth at least five \cite{shalu_sandhya}. 
Heuristic algorithms for restricted star colouring are given in ColPack software suite \cite{gebremedhin2} as well as \cite{gebremedhin1,bozdag}.


The natural problem of interest is of rs colouring an input graph with the minimum number of colours. Let us call this problem as \textsc{Min RS Colouring}. The following are decision versions of this problem (throughout this paper, \( n \) denotes the number of vertices unless otherwise specified).\\

\begin{tabular}{l@{\hskip 0.75cm}l}
\textsc{RS Colourability} & \textsc{\( k \)-RS Colourability}\\
Instance: A graph \( G \), an integer \( k\leq n \). & Instance: A graph \( G \).\\
Question: Is \( \chi_{rs}(G)\leq k \)? & Question: Is \( \chi_{rs}(G)\leq k \)?\\[0.35cm]
\end{tabular}

\noindent Note that a graph \( G \) is 2-rs colourable if and only if every component of \( G \) is a star. Hence, \textsc{\( 2 \)-RS Colourability} is in P.\\

\noindent Question \cite{shalu_sandhya}: Is \textsc{\( k \)-RS Colourability} NP-complete for \( k\geq 3 \)\,?\\

\noindent We answer the above question in the affirmative. We prove that for \( k\geq 3 \), \textsc{\( k \)-RS Colourability} is NP-complete for planar bipartite graphs of maximum degree \( k \). 
We also prove that for all \( \epsilon>0 \), \textsc{Min RS Colouring} is NP-hard to approximate within \( n^{\frac{1}{3}-\epsilon} \) for 2-degenerate bipartite graphs. On the positive side, a linear-time algorithm to test 3-rs colourability of trees is designed. Further, we employ this algorithm to design an \( O(n^3) \)-time algorithm to test 3-rs colourability of chordal graphs. Besides, a few results on rs colouring are obtained using simple observations and known results on star/ordered colouring. Most notable of these is the following: for co-bipartite graphs, \textsc{\( k \)-RS Colourability} is in P for all \( k\), whereas \textsc{RS Colourability} is NP-complete. Interestingly, this result holds for star colouring as well.

The rest of the paper is organized as follows. 
Section~\ref{sec:preliminaries} presents the necessary preliminaries. Section~\ref{sec:bipartite} discusses hardness results on bipartite graphs with more focus on planar bipartite graphs. 
This is followed by Section~\ref{sec:properties} which mainly acts as preparatory for Section~\ref{sec:trees and chordal}. Section~\ref{sec:trees and chordal} presents polynomial-time algorithms to test 3-rs colourability of trees and chordal graphs. Section~\ref{sec:related} discusses related colouring variants. Using known results on related colouring variants, complexity results on rs colouring of co-bipartite graphs are obtained.
\section{Preliminaries}\label{sec:preliminaries}
We follow West~\cite{west} for graph theory terminology and notation. 
We write \( v_1,v_2,\dots,v_n \) for an \( n \)-vertex path (\( P_n \)), and call \( v_1 \) and \( v_n \) as the endpoints of the path. We write \( (v_1,v_2,\dots,v_n) \) for an \( n \)-vertex cycle; in particular, a triangle on vertices \( v_1,v_2,v_3 \) is \( (v_1,v_2,v_3) \).
A star is a graph isomorphic to \( K_{1,p} \) for some \( p\geq 0 \). 
The independence number of a graph \( G \) is denoted by \( \alpha(G) \). 
If \( G \) is a graph and \( S\subseteq V(G) \), the subgraph of \( G \) induced by \( S \), denoted by \( G[S] \), is \( G-(V(G)\setminus S) \). 
The star chromatic number \( \chi_s(G) \) is the least integer \( k \) such that \( G \) admits a \( k \)-star colouring. 
The length of a shortest cycle in a graph is called its \emph{girth}. 
A graph with maximum degree at most three is said to be \emph{subcubic}. 
A graph \( G \) is \emph{2-degenerate} if there exists an ordering of vertices in \( G \) such that every vertex has at most two neighbours to its left. 
In this paper, a vertex of degree three or more is called a \emph{3-plus vertex}. 


The next observation follows from the definition.
\begin{observation}
Let \( G \) be a graph, and \( f:V(G)\to\{0,1,\dots,k-1\} \) be a \( k \)-rs colouring of \( G \). If \( v \) is a vertex of \( G \) with \( f(v)=k-1 \), then \( \deg(v)<k \).
\qed
\label{obs:degree restriction}
\end{observation}

We give special attention to 3-restricted star colouring in this paper. 
The following observations are pivotal to our results on 3-rs colouring.

\begin{observation}
Let \( G \) be a graph. If \( f:V(G)\to\{0,1,2\} \) is a 3-rs colouring of \( G \), then the following properties hold. 
\begin{enumerate}[topsep=2pt,itemsep=3pt,label={P\arabic*)},leftmargin=\widthof{P1)}+\labelsep]
\item[P1)] For every 3-plus vertex \( u \) of \( G \), \( f(u)= \) 0 or 1 (that is, a binary colour).
\item[P2)] If \( uv \) is an edge joining 3-plus vertices \( u \) and \( v \) in \( G \), then \( f(v)=1-f(u) \).
\item[P3)] Both endpoints of a \( P_3 \) in \( G \) cannot be coloured~0 by \( f \). So, if \( u,v,w \) is a \( P_3 \) in \( G \) such that \( f(u)=0 \) and \( w \) is a 3-plus vertex in \( G \), then \( f(w)=1 \).
\item[P4)] A \( P_4 \) in \( G \) cannot have one endpoint coloured~0 and the other endpoint coloured 1 by \( f \). So, if \( u,v,w,x \) is a path in \( G \) where \( u \) and \( x \) are 3-plus vertices in \( G \), then \( f(x)=f(u) \).
\item[P6)] Both endpoints of a \( P_6 \) in \( G \) cannot be coloured~0 by \( f \).
\end{enumerate}
\label{obs:properties}
\end{observation}
\noindent Note: Properties above are numbered as a mnemonic. That is, Property~P2 is about path \( P_2 \), Property~P3 is about path \( P_3 \), and so on.
\begin{proof}
One can easily prove Properties P1 to P4 by listing all possible 3-colourings of the path (see supplementary material for detailed proof). 
To prove Property~P6, assume that \( f \) is a 3-rs colouring of \( G \) and \( u,v,w,x,y,z \) is a path in \( G \) with \( f(u)=f(z)=0 \). Applying Property~P3 on the path \( u,v,w \) reveals that \( f(w)\neq 0 \). Similarly, \( f(x)\neq 0 \). Applying Property~P4 on the path \( u,v,w,x \) reveals that \( f(x)\neq 1 \), and thus \( f(x)=2 \). Similarly, \( f(w)\neq 1 \), and thus \( f(w)=2 \); a contradiction.
\end{proof}

\begin{observation}
Let \( u,v,w,x \) be a path in a graph \( G \), and let \( f \) be a 3-rs colouring of \( G \) such that \( f(u)=0 \) and \( f(v)=1 \). Then, \( f(w)=2 \) and \mbox{\( f(x)=0 \)}.
\qed
\label{obs:colour propagation}
\end{observation}
%
\section{Bipartite Graphs}\label{sec:bipartite}
In this section, NP-completeness results on planar bipartite graphs, and an inapproximation result on 2-degenerate bipartite graphs are presented. In fact, our NP-completeness results hold for a much smaller subclass of planar bipartite graphs. To emphasize more interesting parts, the results for the smaller subclass are deferred to the end of the section. Karpas \etal.~\cite{karpas} proved that \( \chi_{rs}(G)=O(\log n) \) for every planar graph \( G \). They also proved the following results on 2-degenerate graphs: (i)~\( \chi_{rs}(G)=O(\sqrt{n}) \) for every 2-degenerate graph \( G \), and (ii)~for every integer \( n \), there exists a 2-degenerate graph \( G \) on \( n \) vertices such that \( \chi_{rs}(G)>n^{\frac{1}{3}} \).

First, we show that \textsc{3-RS Colourability} is NP-complete for subcubic planar bipartite graphs of girth six using a reduction from \textsc{Cubic Planar Positive 1-in-3 Sat}. To describe the latter problem, we introduce necessary terminology assuming that the reader is familiar with satisfiability problems. 

A CNF formula \( \mathcal{B}=(X,C) \), where \( X \) is the set of variables and \( C \) is the set of clauses, is called a positive CNF formula if no clause contains a negated literal; in other words, the clauses are subsets of \( X \). Let \( \mathcal{B}=(X,C) \) be a positive CNF formula with \( X=\{x_1,x_2,\dots,x_n\} \) and \( C=\{C_1,C_2,\dots,C_m \} \). The \emph{graph of formula \( \mathcal{B} \)}, denoted by \( G_\mathcal{B} \), is the graph with vertex set \( X\cup C \) and edges \( x_iC_j \) for every variable \( x_i \) in clause \( C_j \) (\( i=1,2,\dots,n; j=1,2,\dots,m \)).
Figure~\ref{fig:eg graph G(B)} shows the graph \( G_\mathcal{B} \) for the formula \( \mathcal{B}=(X,C) \) where \( X=\{x_1,x_2,x_3,x_4\} \), \( C=\{C_1,C_2,C_3,C_4\} \), \( C_1=\{x_1,x_2,x_3\} \), \( C_2=\{x_1,x_2,x_4\} \), \( C_3=\{x_1,x_3,x_4\} \) and \( C_4=\{x_2,x_3,x_4\} \).\\  

\noindent \textsc{Cubic Planar Positive 1-in-3 Sat (CPP 1-in-3 Sat)}\\
Instance: \tabto*{1.5cm}A positive 3-CNF formula \( \mathcal{B}=(X,C) \) such that\\
          \tabto*{1.5cm}\( G_\mathcal{B} \) is a cubic planar graph.\\
Question: \tabto*{1.6cm}Is there a truth assignment for \( X \) such that\\
	  \tabto*{1.6cm}every clause in \( C \) has exactly one true variable?\\  

\noindent This problem is proved NP-complete by Moore and Robson~\cite{moore} (Note: in \cite{moore}, the problem is called Cubic Planar Monotone 1-in-3 Sat. We use `positive' rather than `monotone' to be unambiguous; see \cite{tippenhauer}). Observe that the graph \( G_\mathcal{B} \) is cubic if and only if each clause contains three variables and each variable occurs in exactly three clauses. As a result, in a \textsc{CPP 1-in-3 Sat} instance, the number of variables equals the number of clauses, that is \( m=n \).

\begin{theorem}
\textsc{3-RS Colourability} is NP-complete for subcubic planar bipartite graphs of girth at least six.
\label{thm:3-rsc planar bipartite girth 6}
\end{theorem}
\begin{proof}
\textsc{3-RS Colourability} is in NP because given a 3-colouring \( f \) (certificate) of the input graph, we can verify in polynomial time that all bicoloured paths \( x,y,z \) satisfy \( f(y)<f(x) \). 

\begin{figure}[hbt]
\centering
  \begin{subfigure}[b]{0.3\textwidth}
  \centering
\scalebox{0.5}{
  \begin{tikzpicture}[node distance=2cm,line width=1pt]
\tikzset{
dot/.style={draw,fill,circle,inner sep = 0pt,minimum size = 3pt},
vcolour/.style={draw,inner sep=1.5pt,font=\scriptsize,label distance=2pt},
subgraph/.style={draw,ellipse,minimum width=1.75cm,minimum height=2cm},
subgraphHoriz/.style={draw,ellipse,minimum width=1.5cm,minimum height=1.25cm},
}
  \tikzstyle every label=[font=\Large]
  \tikzstyle bigDot=[dot,minimum size=4pt]

  \node [bigDot] (x1)[label=right:\( x_1 \)]{};
  \node [bigDot] (x2)[right of=x1,node distance=2cm,label=right:\( x_2 \)]{};
  \node [bigDot] (x3)[right of=x2,node distance=2cm,label=right:\( x_3 \)]{};
  \node [bigDot] (x4)[right of=x3,node distance=2cm,label=right:\( x_4 \)]{};

  \node [bigDot] (C1) [above = 2.5 of x2][label={above:\( C_1 \)}]{};
  \node [bigDot] (C2) [below = 4.5 of x2,xshift=-2pt][label={below:\( C_2 \)}]{};
  \node [bigDot] (C3) [above = 4.5 of x3,xshift=2pt][label={above:\( C_3 \)}]{};
  \node [bigDot] (C4) [below = 2.5 of x3][label={below:\( C_4 \)}]{};
  \draw [rounded corners=0.5cm]
  (x1)--(C1)
  (x2)--(C1)
  (x3)--(C1)
  (x3)--(C3)
  (x4)|-(C3)
  (x1)|-(C2)
  (x2)--(C2)
  (x2)--(C4)
  (x3)--(C4)
  (x4)--(C4);
  \draw
  (x1) .. controls +(0,4.75).. (C3)
  (x4) .. controls +(0,-4.75).. (C2);
  \end{tikzpicture}
}
  \caption{Graph \( G_\mathcal{B} \) for the formula \( (x_1\vee x_2\vee x_3) \) \( \wedge \) \( (x_1\vee x_2\vee x_4) \) \( \wedge \) \( (x_1\!\vee\!x_3\!\vee\!x_4) \) \( \wedge \) \( (x_2\!\vee\!x_3\!\vee\!x_4) \)}
  \label{fig:eg graph G(B)}
  \end{subfigure}%
\hfill
  \begin{subfigure}[b]{0.3\textwidth}
  \centering
\scalebox{0.5}{
  \begin{tikzpicture}[node distance=1.5cm,label distance=-0.5mm,line width=1pt]
\tikzset{
dot/.style={draw,fill,circle,inner sep = 0pt,minimum size = 3pt},
vcolour/.style={draw,inner sep=1.5pt,font=\scriptsize,label distance=2pt},
subgraph/.style={draw,ellipse,minimum width=1.75cm,minimum height=2cm},
subgraphHoriz/.style={draw,ellipse,minimum width=1.5cm,minimum height=1.25cm},
}
  \tikzstyle every label=[font=\Large]
  \tikzstyle bigDot=[dot,minimum size=4pt]
  \node [bigDot] (x1)[label=right:\( x_1 \)]{};
  \node [bigDot] (x2)[right of=x1,node distance=2cm,label=right:\( x_2 \)]{};
  \node [bigDot] (x3)[right of=x2,node distance=2cm,label=right:\( x_3 \)]{};
  \node [bigDot] (x4)[right of=x3,node distance=2cm,label=right:\( x_4 \)]{};
  
  \node [bigDot] (c12)[above of=x2,label=right:\( c_{12} \),node distance=2cm]{};
  \node [bigDot] (c11)[above left of=c12,label=above:\( c_{11} \)]{};
  \node [bigDot] (c13)[above right of=c12,label=above:\( c_{13} \)]{};
  
  \node [bigDot] (c32)[above of=x3,xshift=2pt,label=right:\( c_{32} \),node distance=4cm]{};
  \node [bigDot] (c31)[above left of=c32,label=above:\( c_{31} \)]{};
  \node [bigDot] (c33)[above right of=c32,label=above:\( c_{33} \)]{};
  
  \node [bigDot] (c42)[below of=x3,label=right:\( c_{42} \),node distance=2cm]{};
  \node [bigDot] (c41)[below left of=c42,label=below:\( c_{41} \)]{};
  \node [bigDot] (c43)[below right of=c42,label=below:\( c_{43} \)]{};
  
  \node [bigDot] (c22)[below of=x2,xshift=-2pt,label=right:\( c_{22} \),node distance=4cm]{};
  \node [bigDot] (c21)[below left of=c22,label=below:\( c_{21} \)]{};
  \node [bigDot] (c23)[below right of=c22,label=below:\( c_{23} \)]{};
  
  \draw [rounded corners=0.5cm]
  (x1)--(c11)
  (x2)--(c12)
  (x3)--(c13)
  (x3)--(c32)
  (x4)|-(c33)
  (x1)|-(c21)
  (x2)--(c22)
  (x2)--(c41)
  (x3)--(c42)
  (x4)--(c43)
  (c11)--(c12)--(c13)--(c11)
  (c21)--(c22)--(c23)--(c21)
  (c31)--(c32)--(c33)--(c31)
  (c41)--(c42)--(c43)--(c41);
  \draw 
  (x1) .. controls +(0,5).. (c31)
  (x4) .. controls +(0,-5).. (c23);
  
  \end{tikzpicture}
}
  \captionsetup{width=0.9\linewidth}
  \caption{The intermediate graph obtained from graph \( G_\mathcal{B} \) in~(a)}
  \label{fig:eg intermediate graph}
  \end{subfigure}
\hfill
  \begin{subfigure}[b]{0.3\textwidth}
  \centering
\scalebox{0.5}{
  \begin{tikzpicture}[node distance=1.5cm,label distance=-0.5mm,line width=1pt]
\tikzset{
dot/.style={draw,fill,circle,inner sep = 0pt,minimum size = 3pt},
vcolour/.style={draw,inner sep=1.5pt,font=\scriptsize,label distance=2pt},
subgraph/.style={draw,ellipse,minimum width=1.75cm,minimum height=2cm},
subgraphHoriz/.style={draw,ellipse,minimum width=1.5cm,minimum height=1.25cm},
}
  \tikzstyle every label=[font=\Large]
  \tikzstyle bigDot=[dot,minimum size=4pt]
  
  \node [bigDot] (x1)[label=right:\( x_1 \)]{};
  \node [bigDot] (x2)[right of=x1,node distance=2cm,label=right:\( x_2 \)]{};
  \node [bigDot] (x3)[right of=x2,node distance=2cm,label=right:\( x_3 \)]{};
  \node [bigDot] (x4)[right of=x3,node distance=2cm,label=right:\( x_4 \)]{};
  
  \node [bigDot] (c12)[above of=x2,label=right:\( c_{12} \),node distance=2cm]{};
  \node [bigDot] (c11)[above left of=c12,label=above:\( c_{11} \)]{};
  \node [bigDot] (c13)[above right of=c12,label=above:\( c_{13} \)]{};
  
  \node [bigDot] (c32)[above of=x3,xshift=2pt,label=right:\( c_{32} \),node distance=4cm]{};
  \node [bigDot] (c31)[above left of=c32,label=above:\( c_{31} \)]{};
  \node [bigDot] (c33)[above right of=c32,label=above:\( c_{33} \)]{};
  
  \node [bigDot] (c42)[below of=x3,label=right:\( c_{42} \),node distance=2cm]{};
  \node [bigDot] (c41)[below left of=c42,label=below:\( c_{41} \)]{};
  \node [bigDot] (c43)[below right of=c42,label=below:\( c_{43} \)]{};
  
  \node [bigDot] (c22)[below of=x2,xshift=-2pt,label=right:\( c_{22} \),node distance=4cm]{};
  \node [bigDot] (c21)[below left of=c22,label=below:\( c_{21} \)]{};
  \node [bigDot] (c23)[below right of=c22,label=below:\( c_{23} \)]{};
  
  \draw [rounded corners=0.5cm]
  (x1) -- node[bigDot](y11)[label={right:\( y_{11} \)}]{} (c11)
  (x1) .. controls +(0,5).. node[bigDot](y13)[label=left:\( y_{13} \)]{} (c31)
  (x2)--node[bigDot](y21)[label=right:\( y_{21} \)]{} (c12)
  (x3.105)--node[bigDot](y31)[label={left:\( y_{31} \)}]{} (c13)
  (x3.75)--node[bigDot](y33)[label=right:\( y_{33} \)]{} (c32)
  (x4) .. controls +(0,5).. node[bigDot](y43)[label=right:\( y_{43} \)]{} (c33)
  (x1) .. controls +(0,-5).. node[bigDot](y12)[label=left:\( y_{12} \)]{} (c21)
  (x2.-105)--node[bigDot](y22)[label=left:\( y_{22} \)]{} (c22)
  (x2.-75)--node[bigDot](y24)[label={right:\( y_{24} \)}]{} (c41)
  (x3)--node[bigDot](y34)[label=right:\( y_{34} \)]{} (c42)
  (x4.-105)--node[bigDot](y44)[label={left:\( y_{44} \)}]{} (c43)
  (x4.-75) .. controls +(0,-5).. node[bigDot](y42)[label=right:\( y_{42} \)]{} (c23)
  
  (c11)--node[bigDot](b11)[midway,label={[xshift=2pt,yshift=-3pt]left:\( b_{11} \)}]{} (c12)--node[bigDot](b12)[midway,label={[xshift=-1pt,yshift=-2pt]right:\( b_{12} \)}]{} (c13)--node[bigDot](b13)[midway,label=above:\( b_{13} \)]{} (c11)
  (c21)--node[bigDot](b21)[midway,label=left:\( b_{21} \)]{} (c22)--node[bigDot](b22)[midway,label=right:\( b_{22} \)]{} (c23)--node[bigDot](b23)[midway,label={[yshift=2pt]below:\( b_{23} \)}]{} (c21)
  (c31)--node[bigDot](b31)[midway,label=left:\( b_{31} \)]{} (c32)--node[bigDot](b32)[midway,label=right:\( b_{32} \)]{} (c33)--node[bigDot](b33)[midway,label=above:\( b_{33} \)]{} (c31)
  (c41)--node[bigDot](b41)[midway,label={[xshift=2pt,yshift=6pt]left:\( b_{41} \)}]{} (c42)--node[bigDot](b42)[midway,label={[xshift=-2pt,yshift=4pt]right:\( b_{42} \)}]{} (c43)--node[bigDot](b43)[midway,label=below:\( b_{43} \)]{} (c41);
  
  \end{tikzpicture}
}
  \captionsetup{width=0.95\linewidth}
  \caption{The graph \( G \) \mbox{obtained} from the intermediate graph in~(b)}
  \label{fig:eg graph G}
  \end{subfigure}
\caption{Construction of graph \( G \) from \( G_\mathcal{B} \) in Theorem~\ref{thm:3-rsc planar bipartite girth 6}}
\label{fig:G from G(B)}
\end{figure}
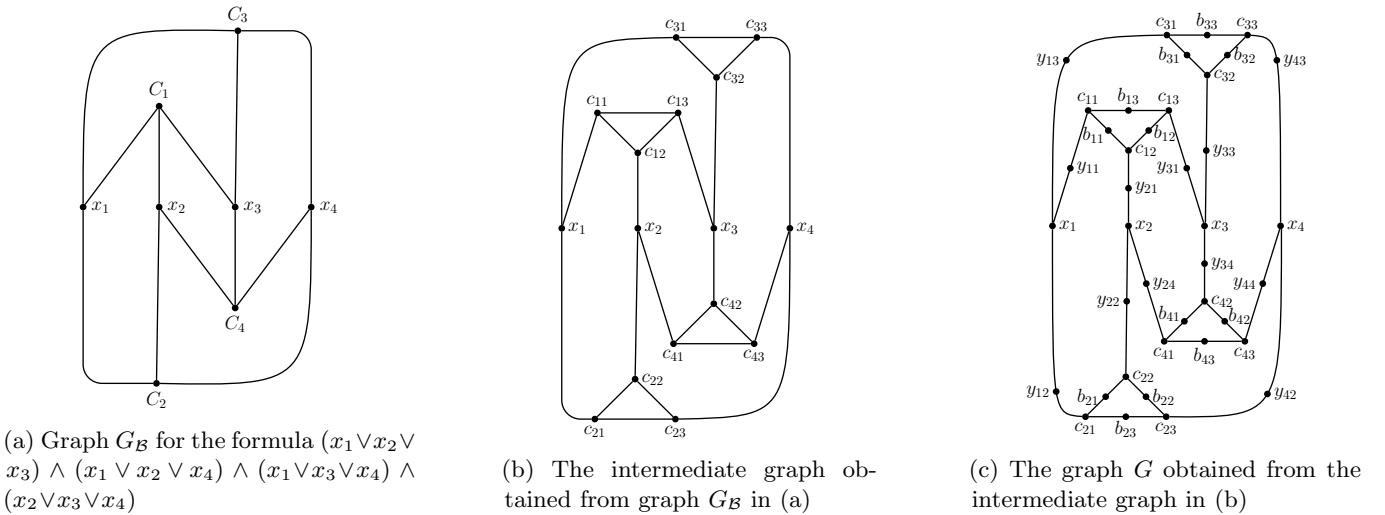
To prove NP-hardness, we transform \textsc{CPP 1-in-3 Sat} problem to \textsc{3-RS Colourability} problem. Let \( \mathcal{B}=(X,C) \) be an instance of \textsc{CPP 1-in-3 Sat} where \( X=\{x_1,x_2,\dots,x_m\} \) and \( C=\{C_1,C_2,\dots,C_m\} \). Recall that \( \mathcal{B} \) is a positive CNF formula and \( G_\mathcal{B} \) is a cubic planar graph. We construct a graph \( G \) from \( G_\mathcal{B} \) as follows. First, an intermediate graph is constructed. For each clause \( C_j=\{x_{j_1},x_{j_2},x_{j_3}\} \), replace vertex \( C_j \) in \( G_\mathcal{B} \) by a triangle \( (c_{j1},c_{j2},c_{j3}) \) and replace edges \( x_{j_1}C_j \), \( x_{j_2}C_j \), \( x_{j_3}C_j \) in \( G_\mathcal{B} \) by edges \( x_{j_1}c_{j1} \), \( x_{j_2}c_{j2} \), \( x_{j_3}c_{j3} \) (see Figure~\ref{fig:G from G(B)}).

The graph \( G \) is obtained by subdividing each edge of this intermediate graph exactly once. Let us call the new vertex introduced upon subdividing the edge \( x_ic_{jk} \) as \( y_{ij} \), and the new vertex introduced upon subdividing the edge \( c_{jk}c_{j\,k+1} \) as \( b_{jk} \) (where index \( k+1 \) is modulo 3). Each 6-vertex cycle \( (c_{j1},b_{j1},c_{j2},b_{j2},c_{j3},b_{j3}) \) serves as the gadget for clause \( C_j \).
Since the intermediate graph is a cubic planar graph of girth three (see Figure~\ref{fig:eg intermediate graph}), \( G \) is a subcubic planar bipartite graph of girth six (see Figure~\ref{fig:eg graph G}).

The graph \( G_\mathcal{B} \) can be constructed in \( O(m) \) time, and it has \( 2m \) vertices and \( 3m \) edges. Also, \( G \) can be constructed from \( G_\mathcal{B} \) in \( O(m) \) time since there are only \( 10m \) vertices and \( 12m \) edges in \( G \). 
All that remains is to prove that \( G \) is 3-rs colourable if and only if \( \mathcal{B} \) is a yes instance of \textsc{CPP 1-in-3 Sat}. 
The following claims help to establish this.\\ 

\begin{itemize}[topsep=0pt,label={CL\arabic*)},leftmargin=\widthof{[CL2]}+\labelsep,itemsep=3pt]
\item[CL1)] If \( f \) is a 3-rs colouring of \( G \),
then for each \( j \), exactly one vertex among \( c_{j1},c_{j2},c_{j3} \) is coloured~0 by \( f \).
\item[CL2)] If \( f \) is a 3-rs colouring of \( G \),
then \( f(c_{jk})=1-f(x_i) \) whenever \( x_i,y_{ij},c_{jk} \) is a path in \( G \).
\item[CL3)] The clause gadget admits a 3-rs colouring scheme that assigns colour 0 on one of the vertices \( c_{j1},c_{j2},c_{j3} \) and colour 1 on the other two vertices.\\
\end{itemize}

Since \( x_i \)'s and \( c_{jk} \)'s are 3-plus vertices in \( G \), they must receive binary colours by Property~P1. 
Observe that every pair of vertices from \( c_{j1},c_{j2},c_{j3} \) is at a distance two from each other. Since Property~P3 forbids assigning colour 0 on both endpoints of a \( P_3 \), at most one vertex among \( c_{j1},c_{j2},c_{j3} \) is coloured~0 by \( f \). Thus, to prove claim~CL1, it suffices to show that at least one of them is coloured~0 by \( f \). On the contrary, assume that \( f(c_{j1})=f(c_{j2})=f(c_{j3})=1 \). Since \( c_{j1},b_{j1},c_{j2} \) is a bicoloured \( P_3 \) with colour 1 at the endpoints, its middle vertex must be coloured~0. That is, \( f(b_{j1})=0 \). Similarly, \( f(b_{j2})=0 \). This means that \( b_{j1},c_{j2},b_{j2} \) is a \( P_3 \) with \( 1=f(c_{j2})>f(b_{j1})=f(b_{j2})=0 \)\,; a contradiction. This proves claim~CL1.

Next, we prove claim~CL2 for \( k=1 \). The proof is similar for other values of \( k \). Let \( x_i,y_{ij},c_{j1} \) be a path in \( G \) where \( i,j\in\{1,2,\dots,m\} \). Recall that \( f(x_i),f(c_{j1})\in\{0,1\} \). If \( f(x_i)=0 \), then \( f(c_{j1})=1 \) due to Property~P3. So, it suffices to prove that \( f(x_i)=1 \) implies \( f(c_{j1})=0 \). On the contrary, assume that \( f(x_i)=f(c_{j1})=1 \). Since \( x_i,y_{ij},c_{j1} \) is a bicoloured \( P_3 \) with colour 1 at its endpoints, its middle vertex \( y_{ij} \) must be coloured 0 (by \( f \)). Thus, we have \( f(y_{ij})=0 \) and \( f(c_{j1})=1 \). Therefore, applying Observation~\ref{obs:colour propagation} on paths \( y_{ij},c_{j1},b_{j1},c_{j2} \) and \( y_{ij},c_{j1},b_{j3},c_{j3} \) gives \( f(b_{j1})=f(b_{j3})=2 \) and \( f(c_{j2})=f(c_{j3})=0 \) (see Figure~\ref{fig:proof of CL2}). The equation \( f(c_{j2})=f(c_{j3})=0 \) contradicts Property~P3. This completes the proof of claim~CL2.

Figure~\ref{fig:3-rsc extension old clause gadget} exhibits the colouring scheme guaranteed by claim~CL3.
\begin{figure}[hbt]
\centering
\begin{minipage}[b]{0.48\textwidth}
\centering
\begin{tikzpicture}
\tikzset{
dot/.style={draw,fill,circle,inner sep = 0pt,minimum size = 3pt},
vcolour/.style={draw,inner sep=1.5pt,font=\scriptsize,label distance=2pt},
subgraph/.style={draw,ellipse,minimum width=1.75cm,minimum height=2cm},
subgraphHoriz/.style={draw,ellipse,minimum width=1.5cm,minimum height=1.25cm},
}
\draw (0,0) node[dot](xi)[label=below:\( x_i \)][label={[vcolour]above:1}]{} --++(1,0) node[dot](yij)[label=below:\( y_{ij} \)][label={[vcolour]above:0}]{} --++(1,0) node[dot](cj1)[label=below:\( c_{j1} \)][label={[vcolour,xshift=-5pt]above:1}]{};

\draw (cj1) --++(1,1.5) node[dot](cj2)[label=above:\( c_{j2} \)][label={[vcolour,yshift=-3pt]below:0}]{} --++(1,-1.5) node[dot](cj3)[label=below:\( c_{j3} \)][label={[vcolour,xshift=-5pt]above left:0}]{} --(cj1);
\path (cj1) --node[dot](bj1)[label=left:\( b_{j1} \)][label={[vcolour]right:2}]{} (cj2)
      (cj2) --node[dot](bj2)[label=right:\( b_{j2} \)]{} (cj3)
      (cj3) --node[dot](bj3)[label=below:\( b_{j3} \)][label={[vcolour]above:2}]{} (cj1);

\draw[ultra thick] (cj2)--(bj2)--(cj3);
\end{tikzpicture}
\caption{\( f(x_i)=f(c_{j1})=1 \) leads to a contradiction}
\label{fig:proof of CL2}
\end{minipage}
\hfill
\begin{minipage}[b]{0.48\textwidth}
\centering
\begin{tikzpicture}
\tikzset{
dot/.style={draw,fill,circle,inner sep = 0pt,minimum size = 3pt},
vcolour/.style={draw,inner sep=1.5pt,font=\scriptsize,label distance=2pt},
subgraph/.style={draw,ellipse,minimum width=1.75cm,minimum height=2cm},
subgraphHoriz/.style={draw,ellipse,minimum width=1.5cm,minimum height=1.25cm},
}
\draw (0,0) node[dot](cj1)[label=below:\( c_{j1} \)][label={[vcolour,xshift=5pt]above right:0}]{} --++(1,1.5) node[dot](cj2)[label=above:\( c_{j2} \)][label={[vcolour,yshift=-3pt]below:1}]{} --++(1,-1.5) node[dot](cj3)[label=below:\( c_{j3} \)][label={[vcolour,xshift=-5pt]above left:1}]{} --(cj1);
\path (cj1) --node[dot](bj1)[label=left:\( b_{j1} \)][label={[vcolour]right:2}]{} (cj2)
      (cj2) --node[dot](bj2)[label=right:\( b_{j2} \)][label={[vcolour]left:0}]{} (cj3)
      (cj3) --node[dot](bj3)[label=below:\( b_{j3} \)][label={[vcolour]above:2}]{} (cj1);
\end{tikzpicture}
\caption{3-rs colouring scheme for the clause gadget guaranteed by claim~CL3 (`rotate' colours if \( c_{j2} \) or \( c_{j3} \) gets colour 0)}
\label{fig:3-rsc extension old clause gadget}
\end{minipage}
\end{figure}

Now, we are ready to prove that \( G \) is 3-rs colourable if and only if \( \mathcal{B} \) is a yes instance of \textsc{CPP 1-in-3 Sat}. Suppose that \( G \) has a 3-rs colouring \( f \). Since \( x_i \)'s are 3-plus vertices in \( G \), \( f(x_i)\in\{0,1\} \) for \( 1\leq i\leq m \). We define a truth assignment \( \mathcal{A} \) for \( X \) by setting variable \( x_i\gets \text{true} \) if \( f(x_i)=1 \), and \( x_i\gets \text{false} \) if \( f(x_i)=0 \). We claim that each clause \( C_j \) has exactly one true variable under \( \mathcal{A} \) (\( 1\leq j\leq m \)). Let \( C_j=\{x_p,x_q,x_r\} \). By claim~CL1, exactly one vertex among \( c_{j1},c_{j2},c_{j3} \) is coloured 0 under \( f \). Without loss of generality, assume that \( f(c_{j1})=0 \) and \( f(c_{j2})=f(c_{j3})=1 \). As \( x_p,y_{pj},c_{j1} \) is a path in \( G \), \( f(x_p)=1-f(c_{j1})=1 \) by claim~CL2. Similarly, \( f(x_q)=1-f(c_{j2})=0 \) and \( f(x_r)=1-f(c_{j3})=0 \) (consider the path \( x_q,y_{qj},c_{j2} \) and the path \( x_r,y_{rj},c_{j3} \)\,). Hence, by definition of \( \mathcal{A} \), \( x_p \) is true whereas \( x_q \) and \( x_r \) are false. Therefore, \( \mathcal{A} \) is a truth assignment such that each clause \( C_j \) has exactly one true variable (\( 1\leq j\leq m \)). So, \( \mathcal{B} \) is a yes instance of \textsc{CPP 1-in-3 Sat}.

Conversely, suppose that \( X \) has a truth assignment \( \mathcal{A} \) such that each clause has exactly one true variable. We produce a 3-rs colouring \( f \) of \( G \) as follows. 
First, colour vertices \( x_i \) by the rule: \( f(x_i)=1 \) if \( x_i \) is true; otherwise, \( f(x_i)=0 \). 
Next, vertices \( y_{ij} \) and \( c_{jk} \) are coloured. Assign colour~2 to all \( y_{ij} \)'s. Whenever \( x_{i},y_{ij},c_{jk} \) is a path in \( G \), assign \( f(c_{jk})=1-f(x_i) \). This ensures that the path \( x_i,y_{ij},c_{jk} \) is not bicoloured. 
Finally, extend the partial colouring to clause gadgets using the scheme guaranteed by claim~CL3. 
We claim that \( f \) is a 3-rs colouring of \( G \). Obviously, \( f \) is a 3-colouring. The bicoloured \( P_3 \)'s in \( G \) are either entirely within a clause gadget or has a vertex \( y_{ij} \) as an endpoint. Since clause gadgets are coloured by a 3-rs colouring scheme, every bicoloured \( P_3 \) of the first kind has a lower colour on its middle vertex. Besides, every bicoloured \( P_3 \) of the second kind has colour 2 at its endpoints because \( y_{ij} \)'s are coloured 2, and thus its middle vertex has a lower colour. Therefore, \( f \) is a 3-rs colouring of \( G \).


So, \( \mathcal{B} \) is a yes instance of \textsc{CPP 1-in-3 Sat} if and only if \( G \) is 3-rs colourable. 
\end{proof}

The reduction in Theorem~\ref{thm:3-rsc planar bipartite girth 6} can be modified to give \( G\bm{'} \) arbitrarily large girth. The modification required is to replace the clause gadget by the gadget displayed in Figure~\ref{fig:new clause gadget}.


\begin{figure}[hbtp]
\centering
\begin{subfigure}[b]{0.52\textwidth}
\centering
\begin{tikzpicture}[scale=0.45]
\draw (0,0) node[dot](cj1)[label={[label distance=-2pt]below left:\( c_{j1} \)}]{} --++(60:1) node[dot](pj1 1){} --++(60:1) node[dot](qj1 1){} --++(60:1) node[dot](rj1 1)[label=right:\( a_{j1}^{(1)} \)]{} --++(60:1) node[dot](pj1 2){} --++(60:1) node[dot](qj1 2){} --++(60:1) node[dot](rj1 2)[label=right:\( a_{j1}^{(2)} \)]{} --++(60:0.25) node{}
++(60:1.25) node{} --++(60:0.25) node[dot](pj1 s){} --++(60:1) node[dot](qj1 s){} --++(60:1) node[dot](rj1 s)[label={[label distance=-2pt]right:\( a_{j1}^{(s)} \)}]{} --++(60:1) node[dot](bj1)[label={[label distance=-2pt]above left:\( b_{j1} \)}]{} --++(60:1) node[dot](cj2)[label=above:\( c_{j2} \)]{};
\path (rj1 2)--node[sloped]{\( \cdots \)} (pj1 s);
\draw (rj1 1)--+(150:1) node[dot]{}
      (rj1 2)--+(150:1) node[dot]{}
      (rj1 s)--+(150:1) node[dot]{};

\draw (cj2) --++(-60:1) node[dot](pj2 1){} --++(-60:1) node[dot](qj2 1){} --++(-60:1) node[dot](rj2 1)[label={[label distance=-2pt]left:\( a_{j2}^{(1)} \)}]{} --++(-60:1) node[dot](pj2 2){} --++(-60:1) node[dot](qj2 2){} --++(-60:1) node[dot](rj2 2)[label=left:\( a_{j2}^{(2)} \)]{} --++(-60:0.25) node{}
++(-60:1.25) node{} --++(-60:0.25) node[dot](pj2 s){} --++(-60:1) node[dot](qj2 s){} --++(-60:1) node[dot](rj2 s)[label=left:\( a_{j2}^{(s)} \)]{} --++(-60:1) node[dot](bj2)[label=right:\( b_{j2} \)]{} --++(-60:1) node[dot](cj3)[label={[label distance=-2pt]below:\( c_{j3} \)}]{};
\path (rj2 2)--node[sloped]{\( \cdots \)} (pj2 s);
\draw (rj2 1)--+(30:1) node[dot]{}
      (rj2 2)--+(30:1) node[dot]{}
      (rj2 s)--+(30:1) node[dot]{};

\draw (cj3) --++(-1,0) node[dot](pj3 1){} --++(-1,0) node[dot](qj3 1){} --++(-1,0) node[dot](rj3 1)[label=above:\( a_{j3}^{(1)} \)]{} --++(-1,0) node[dot](pj3 2){} --++(-1,0) node[dot](qj3 2){} --++(-1,0) node[dot](rj3 2)[label=above:\( a_{j3}^{(2)} \)]{} --++(-0.25,0) node{}
++(-1.25,0) node{} --++(-0.25,0) node[dot](pj3 s){} --++(-1,0) node[dot](qj3 s){} --++(-1,0) node[dot](rj3 s)[label=above:\( a_{j3}^{(s)} \)]{} --++(-1,0) node[dot](bj3)[label={[label distance=-3pt]below:\( b_{j3} \)}]{} -- (cj1);
\path (rj3 2)--node[sloped]{\( \cdots \)} (pj3 s);
\draw (rj3 1)--+(-90:1) node[dot]{}
      (rj3 2)--+(-90:1) node[dot]{}
      (rj3 s)--+(-90:1) node[dot]{};
\end{tikzpicture}
\caption{The new clause gadget. This replaces the old clause gadget \( (c_{j1},b_{j1},c_{j2},b_{j2},c_{j3},b_{j3}) \)}
\label{fig:new clause gadget}
\end{subfigure}%
\hfill
\begin{subfigure}[b]{0.44\textwidth}
\centering
\begin{tikzpicture}[scale=0.57]
\draw (0,0) node[dot](cj1)[label=below:\( c_{j1} \)][label={[vcolour]left:0}]{} --++(60:1) node[dot](pj1 1)[label={[vcolour]left:2}]{} --++(60:1) node[dot](qj1 1)[label={[vcolour]left:1}]{} --++(60:1) node[dot](rj1 1)[label={[vcolour,yshift=-2pt]left:0}][label={[label distance=-8pt]below right:\( a_{j1}^{(1)} \)}]{} --++(60:1) node[dot](pj1 2)[label={[vcolour]left:2}]{} --++(60:1) node[dot](qj1 2)[label={[vcolour]left:1}]{} --++(60:1) node[dot](rj1 2)[label={[vcolour,yshift=-2pt]left:0}][label={[label distance=-8pt]below right:\( a_{j1}^{(2)} \)}]{} --++(60:1) node[dot](bj1)[label={[label distance=-5pt]below right:\( b_{j1} \)}][label={[vcolour]left:2}]{} --++(60:1) node[dot](cj2)[label=above:\( c_{j2} \)][label={[vcolour]right:1}]{};
\draw (rj1 1)--+(150:1) node[dot](lj1 1)[label={[vcolour]left:2}]{}
      (rj1 2)--+(150:1) node[dot](lj1 2)[label={[vcolour]left:2}]{};

\draw (cj2) --++(-60:1) node[dot](pj2 1)[label={[vcolour]right:0}]{} --++(-60:1) node[dot](qj2 1)[label={[vcolour]right:2}]{} --++(-60:1) node[dot](rj2 1)[label={[vcolour,yshift=-2pt]right:1}][label={[font=\small]left:\( a_{j2}^{(1)} \)}]{} --++(-60:1) node[dot](pj2 2)[label={[vcolour]right:0}]{} --++(-60:1) node[dot](qj2 2)[label={[vcolour]right:2}]{} --++(-60:1) node[dot](rj2 2)[label={[vcolour,yshift=-2pt]right:1}][label={[font=\small]left:\( a_{j2}^{(2)} \)}]{} --++(-60:1) node[dot](bj2)[label={[yshift=3pt]left:\( b_{j2} \)}][label={[vcolour]right:0}]{} --++(-60:1) node[dot](cj3)[label=below:\( c_{j3} \)][label={[vcolour]right:1}]{};
\draw (rj2 1)--+(30:1) node[dot](lj2 1)[label={[vcolour]right:2}]{}
      (rj2 2)--+(30:1) node[dot](lj2 2)[label={[vcolour]right:2}]{};

\draw (cj3) --++(-1,0) node[dot](pj3 1)[label={[vcolour]above:2}]{} --++(-1,0) node[dot](qj3 1)[label={[vcolour]above:0}]{} --++(-1,0) node[dot](rj3 1)[label={[vcolour]above:1}][label={[font=\small,label distance=-6pt]below right:\( a_{j3}^{(1)} \)}]{} --++(-1,0) node[dot](pj3 2)[label={[vcolour]above:2}]{} --++(-1,0) node[dot](qj3 2)[label={[vcolour]above:0}]{} --++(-1,0) node[dot](rj3 2)[label={[vcolour]above:1}][label={[font=\small,label distance=-6pt]below right:\( a_{j3}^{(2)} \)}]{} --++(-1,0) node[dot](bj3)[label=below:\( b_{j3} \)][label={[vcolour]above:2}]{} -- (cj1);
\draw (rj3 1)--+(-90:1) node[dot](lj3 1)[label={[vcolour]below:2}]{}
      (rj3 2)--+(-90:1) node[dot](lj3 2)[label={[vcolour]below:2}]{};
\end{tikzpicture}
\caption{a 3-rs colouring scheme for the new clause gadget if \( s=2 \) (`rotate' colours if \( c_{j2} \) or \( c_{j3} \) gets colour 0)}
\label{fig:3-rs colouring scheme for new clause gadget}
\end{subfigure}%
\caption{(a)The new clause gadget, (b) the new clause gadget for \( s=2 \) with a 3-rs colouring}
\end{figure}
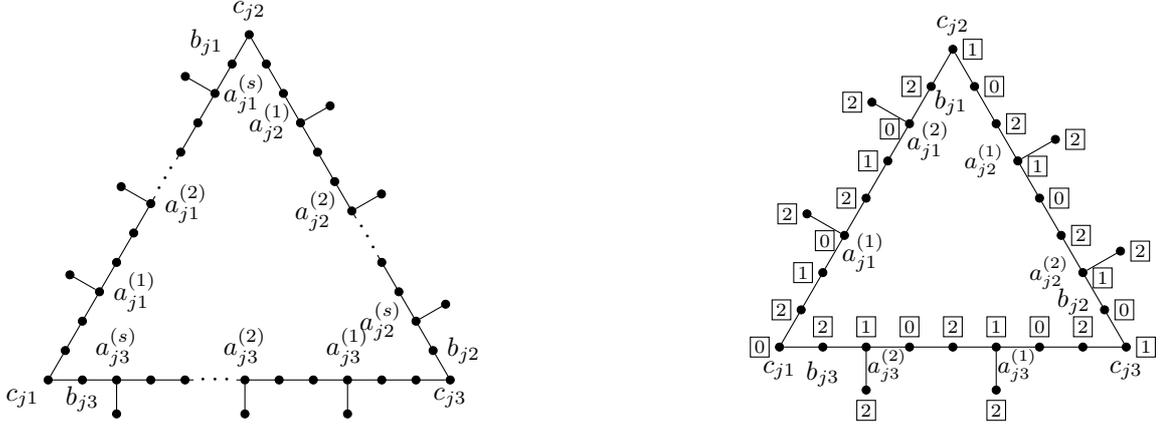

\begin{theorem}
For \( g\geq 6 \), \textsc{3-RS Colourability} is NP-complete for subcubic planar bipartite graphs of girth at least \( g \).
\label{thm:3-rsc planar bipartite girth g}
\end{theorem}
\begin{proof}
We employ a modified form of the reduction in Theorem~\ref{thm:3-rsc planar bipartite girth 6}. Let \( s \) be the smallest even number satisfying \( s\geq \lceil\frac{g}{6}\rceil \) (we need \( s \) to be even to ensure that the graph to be constructed is bipartite). Replace the old clause gadget by the new one displayed in Figure~\ref{fig:new clause gadget}. For convenience, let us call the graph constructed in Theorem~\ref{thm:3-rsc planar bipartite girth 6} as \( G \) and the graph produced by the modified construction as \( G_{new} \). For the reduction to work, it suffices to show that claims CL1, CL2 and CL3 still hold in \( G_{new} \). 
\begin{itemize}[topsep=2pt,itemsep=2pt]
\item[CL1)] If \( f \) is a 3-rs colouring of \( G_{new} \), 
then for each \( j \), exactly one vertex among \( c_{j1},c_{j2},c_{j3} \) is coloured~0 by \( f \).
\item[CL2)] If \( f \) is a 3-rs colouring of \( G_{new} \), 
then \( f(c_{jk})=1-f(x_i) \) whenever \( x_i,y_{ij},c_{jk} \) is a path in \( G_{new} \).
\item[CL3)] The new clause gadget admits a 3-rs colouring scheme that assigns colour~0 on one of the vertices \( c_{j1},c_{j2},c_{j3} \) and colour~1 on the other two vertices.\\
\end{itemize}

Thanks to Claim~1 below, we can prove claims~CL1 and CL2 without difficulty (see supplementary material for complete proof).\\

\noindent Claim 1: \emph{For every 3-rs colouring \( f \) of the new clause gadget, \( f(a_{jk}^{(t)})=f(c_{jk}) \) for \( 1\leq j\leq m \), \( 1\leq k\leq 3 \), and \( 1\leq t\leq s \).}\\

\noindent Claim~1 follows from applying Property~P4 to \( c_{jk},a_{jk}^{(1)} \)-path and \( a_{jk}^{(t)},a_{jk}^{(t+1)} \)-paths for \( 1\leq t<s \). 

Figure~\ref{fig:3-rs colouring scheme for new clause gadget} exhibits the colouring scheme guaranteed by claim~CL3 for \( s=2 \). For higher values of \( s \), a colouring scheme can be produced by following the same pattern.

All that remains is to show that \( G_{new} \) is a subcubic planar bipartite graph of girth at least \( g \). Obviously, it is a subcubic planar graph. Observe that the modified construction can be viewed as replacing each path \( c_{jk},b_{jk},c_{j\,k+1} \) of \( G \) by an even length path and then attaching some pendant vertices (the length of \( c_{jk},c_{j\,k+1} \)-path in \( G_{new} \) is \( 3s+2 \) which is even because \( s \) is even). Therefore, \( G_{new} \) is bipartite. Observe that every cycle in \( G \) contains at least two paths of the form \( c_{jk},b_{jk},c_{j\,k+1} \). Similarly, every cycle in \( G_{new} \) contains at least two \( c_{jk},c_{j\,k+1} \)-paths. Since such paths have length \( 3s+2 \), the girth of \( G_{new} \) is at least \( 2(3s+2)>6s\geq g \). This completes the proof of the theorem.
\end{proof}

Theorem~\ref{thm:3-rsc planar bipartite girth g} can be generalized using a simple operation. For a graph \( G \) with \( \Delta(G)=k \), the graph \( G^+ \) is obtained from \( G \) by adding enough pendant vertices at every vertex \( v \) of \( G \) so that \( \deg_{G^+}(v)=k+1 \). Hence, each vertex in \( G^+ \) has degree 1 or \( k+1 \). As we are only adding pendant vertices, \( G^+ \) preserves the planarity, bipartiteness and girth of \( G \). Moreover, we have \( \Delta(G^+)=k+1 \). Further, this operation is useful to construct graphs of desired rs chromatic number.
\begin{observation}
Let \( G \) be a graph with \( \Delta(G)=k \) where \( k\in\mathbb{N} \). Then, \( G \) is \( k \)-rs colourable if and only if \( G^+ \) is \( (k+1) \)-rs colourable.
\label{obs:rsc increment}
\end{observation}
\begin{proof}
If \( G \) is \( k \)-rs colourable, then we can colour the new pendant vertices added to \( G \) with a new colour \( k \) so that \( G^+ \) is \( (k+1) \)-rs colourable. Conversely, suppose that \( G^+ \) admits a \( (k+1) \)-rs colouring \( f \). Recall that for every vertex \( v \) of \( G^+ \), \( \deg_{G^+}(v)= \) 1 or \( k+1 \). By Observation~\ref{obs:degree restriction}, a vertex of degree \( k+1 \) cannot receive colour \( k \) under a \( (k+1) \)-rs colouring. 
Hence, no non-pendant vertex in \( G^+ \) is coloured \( k \) by \( f \). Observe that the set of non-pendant vertices in \( G^+ \) is precisely \( V(G) \). Since only colours 0 to \( k-1 \) are used on non-pendant vertices of \( G^+ \) (under \( f \)), the restriction of \( f \) to \( V(G) \) is a \( k \)-rs colouring of \( G \).
\end{proof}
By mathematical induction, Theorem~\ref{thm:3-rsc planar bipartite girth g} can be generalized as follows using Observation~\ref{obs:rsc increment}.
\begin{theorem}
For \( k\geq 3 \) and \( g\geq 6 \), \( k \)-\textsc{RS Colourability} is NP-complete for planar bipartite graphs of maximum degree \( k \) and girth at least \( g \).
\qed
\label{thm:k-rsc planar bipartite girth g}
\end{theorem}

Next, we show that the NP-completeness result presented in Theorem~\ref{thm:3-rsc planar bipartite girth 6} holds for a much smaller subclass. 
\begin{theorem}
\textsc{3-RS Colourability} is NP-complete for subcubic planar bipartite graphs \( G \) even when \( G \) is 2-degenerate, girth\( (G)\geq 6 \), \( \chi_s(G)=3 \), and \( \chi_{rs}(G)\leq 4 \).
\label{thm:3-rsc planar bipartite etc}
\end{theorem}
\begin{proof}
To prove the theorem, it is enough to show that the graph \( G \) constructed in the proof of Theorem~\ref{thm:3-rsc planar bipartite girth 6} is 2-degenerate and satisfies \( \text{girth}(G)\geq 6 \), \( \chi_s(G)=3 \), and \( \chi_{rs}(G)\leq 4 \). It is already shown that \( G \) has girth at least six. To obtain a 2-degenerate ordering of \( V(G) \), list vertices of degree three first, followed by vertices of degree two (note that vertices of degree three form an independent set). 
Next, we show that \( G \) is 3-star colourable and 4-rs colourable.\\

\noindent Claim~1: \( G \) admits a 3-star colouring \( \phi \).\\[5pt]
To produce \( \phi \), an auxiliary graph \( H_\mathcal{B} \) is constructed first from the formula \( \mathcal{B}=(X,C) \) of Theorem~\ref{thm:3-rsc planar bipartite girth 6}. The vertex set of \( H_\mathcal{B} \) is the set of variables \( X \), and two vertices \( x_i \) and \( x_j \) are joined by an edge in \( H_\mathcal{B} \) if and only if there is a clause in \( \mathcal{B} \) containing both of them (see Figure~\ref{fig:H_B} for an example). 

\begin{figure}[hbt]
\centering
\begin{subfigure}[b]{0.33\textwidth}
\centering
\scalebox{0.55}{
\begin{tikzpicture}[every label/.style={font=\Large},largeVcolour/.style={vcolour,font=\small}]
\clip (-0.75,-5.5) rectangle (5.75,5.5);
  \node [dot] (x1)[label=left:\( x_1 \)][label={[largeVcolour]right:0}]{};
  \node [dot] (x2)[right of=x1,node distance=1.75cm][label=left:\( x_2 \)][label={[largeVcolour]right:1}]{};
  \node [dot] (x3)[right of=x2,node distance=1.75cm][label=left:\( x_3 \)][label={[largeVcolour]right:2}]{};
  \node [dot] (x4)[right of=x3,node distance=1.75cm][label=left:\( x_4 \)][label={[largeVcolour]right:3}]{};
  
  \node[coordinate](c12)[above of=x2,node distance=2cm]{};
  \node[coordinate](c11)[above left of=c12]{};
  \node[coordinate](c13)[above right of=c12]{};
  
  \node[coordinate](c32)[above of=x3,node distance=4cm]{};
  \node[coordinate](c31)[above left of=c32]{};
  \node[coordinate](c33)[above right of=c32]{};
  
  \node[coordinate](c42)[below of=x3,node distance=1cm]{};
  \node[coordinate](c41)[below left of=c42]{};
  \node[coordinate](c43)[below right of=c42]{};
  
  \node[coordinate](c22)[below of=x2,node distance=2.5cm]{};
  \node[coordinate](c21)[below left of=c22]{};
  \node[coordinate](c23)[below right of=c22]{};

  \path (x1)|- coordinate(12) (c21);
  \path (x1)|- coordinate(13) (c31);
  \path (x4)|- coordinate(42) (c23);
  \path (x4)|- coordinate(43) (c33);
  
\path
  (c13)--node[midway](C1){} (c11)
  (c23)--node[midway](C2){} (c21)
  (c33)--node[midway](C3){} (c31)
  (c43)--node[midway](C4){} (c41);

\path
(x1)--node(x1-2){} (x2) 
(x1-2)--node(x1tox2)[pos=0.75]{} (C1)
(x2)--node(x2-3){} (x3) 
(x2-3)--node(x2tox3)[pos=0.75]{} (C1);
\draw
(x1) ..controls (x1tox2).. (x2)
(x2) ..controls (x2tox3).. (x3)
(x1) ..controls (13) and (C3).. (x3)
(x3) ..controls (C3) and (43).. (x4)
(x2) ..controls (C2) and (42).. (x4)
(x1) to[out=-105,in=-75,looseness=3.5] (x4);
\end{tikzpicture}
}
\caption{}
\label{fig:H_B}
\end{subfigure}%
\begin{subfigure}[b]{0.33\textwidth}
\centering
\scalebox{0.55}{
\begin{tikzpicture}[every label/.style={font=\Large}]
\clip (-0.75,-5.5) rectangle (5.75,5.5);
  \node [dot] (x1)[label=left:\( x_1 \)]{};
  \node [dot] (x2)[right of=x1,node distance=1.75cm][label=left:\( x_2 \)]{};
  \node [dot] (x3)[right of=x2,node distance=1.75cm][label=left:\( x_3 \)]{};
  \node [dot] (x4)[right of=x3,node distance=1.75cm][label=left:\( x_4 \)]{};
  
  \node[coordinate](c12)[above of=x2,node distance=2cm]{};
  \node[coordinate](c11)[above left of=c12]{};
  \node[coordinate](c13)[above right of=c12]{};
  
  \node[coordinate](c32)[above of=x3,node distance=4cm]{};
  \node[coordinate](c31)[above left of=c32]{};
  \node[coordinate](c33)[above right of=c32]{};
  
  \node[coordinate](c42)[below of=x3,node distance=1cm]{};
  \node[coordinate](c41)[below left of=c42]{};
  \node[coordinate](c43)[below right of=c42]{};
  
  \node[coordinate](c22)[below of=x2,node distance=2.5cm]{};
  \node[coordinate](c21)[below left of=c22]{};
  \node[coordinate](c23)[below right of=c22]{};

  \path (x1)|- coordinate(12) (c21);
  \path (x1)|- coordinate(13) (c31);
  \path (x4)|- coordinate(42) (c23);
  \path (x4)|- coordinate(43) (c33);
  
\path
  (c13)--node[dot][midway][label={[yshift=-2pt]above:\( C_1 \)}](C1){} (c11)
  (c23)--node[dot][midway][label=below:\( C_2 \)](C2){} (c21)
  (c33)--node[dot][midway][label=above:\( C_3 \)](C3){} (c31)
  (c43)--node[dot][midway][label=below:\( C_4 \)](C4){} (c41);

\draw
 (x1)--(C1)   (x2)--(C1)  (x3)--(C1)
 (x1) .. controls (12).. (C2)
 (x2)--(C2)
 (x4) .. controls (42).. (C2)
 (x1) .. controls (13).. (C3)
 (x3)--(C3)
 (x4) .. controls (43).. (C3)
 (x2)--(C4)   (x3)--(C4)  (x4)--(C4);

\path
(x1)--node(x1-2){} (x2) 
(x1-2)--node(x1tox2)[pos=0.75]{} (C1)
(x2)--node(x2-3){} (x3) 
(x2-3)--node(x2tox3)[pos=0.75]{} (C1);
\draw
(x1) ..controls (x1tox2).. (x2)
(x2) ..controls (x2tox3).. (x3)
(x1) ..controls (13) and (C3).. (x3)
(x3) ..controls (C3) and (43).. (x4)
(x2) ..controls (C2) and (42).. (x4)
(x1) to[out=-105,in=-75,looseness=3.5] (x4);
\end{tikzpicture}
}
\caption{}
\label{fig:H_B on G_B}
\end{subfigure}%
\begin{subfigure}[b]{0.33\textwidth}
\centering
\scalebox{0.6}{
\begin{tikzpicture}[every label/.style={font=\Large},largeVcolour/.style={vcolour,font=\small}]
  \node [dot] (x1)[label=left:\( x_1 \)][label={[largeVcolour]right:0}]{};
  \node [dot] (x2)[right of=x1,node distance=1.75cm][label=left:\( x_2 \)][label={[largeVcolour]right:0}]{};
  \node [dot] (x3)[right of=x2,node distance=1.75cm][label=left:\( x_3 \)][label={[largeVcolour]right:1}]{};
  \node [dot] (x4)[right of=x3,node distance=1.75cm][label=left:\( x_4 \)][label={[largeVcolour]right:1}]{};
  
  \node [dot] (c12)[above of=x2,label={[largeVcolour]right:1},node distance=2cm][label={[label distance=-5pt]below left:\( c_{12} \)}]{};
  \node [dot] (c11)[above left of=c12,label={[largeVcolour]above:1}][label=left:\( c_{11} \)]{};
  \node [dot] (c13)[above right of=c12,label={[largeVcolour]above:0}][label={right:\( c_{13} \)}]{};
  
  \node [dot] (c32)[above of=x3,label={[largeVcolour]right:0},node distance=4cm][label={[label distance=-5pt]below left:\( c_{32} \)}]{};
  \node [dot] (c31)[above left of=c32,label={[largeVcolour]above:1}][label=above left:\( c_{31} \)]{};
  \node [dot] (c33)[above right of=c32,label={[largeVcolour]above:0}][label={above right:\( c_{33} \)}]{};
  
  \node [dot] (c42)[below of=x3,label={[largeVcolour]right:0},node distance=2cm][label={[label distance=-5pt]above left:\( c_{42} \)}]{};
  \node [dot] (c41)[below left of=c42,label={[largeVcolour]below:1}][label=left:\( c_{41} \)]{};
  \node [dot] (c43)[below right of=c42,label={[largeVcolour]below:0}][label={right:\( c_{43} \)}]{};
  
  \node [dot] (c22)[below of=x2,label={[largeVcolour]right:1},node distance=4cm][label={[label distance=-5pt]above left:\( c_{22} \)}]{};
  \node [dot] (c21)[below left of=c22,label={[largeVcolour]below:1}][label=below left:\( c_{21} \)]{};
  \node [dot] (c23)[below right of=c22,label={[largeVcolour]below:0}][label={below right:\( c_{23} \)}]{};

  \path (x1)|- coordinate(12) (c21);
  \path (x1)|- coordinate(13) (c31);
  \path (x4)|- coordinate(42) (c23);
  \path (x4)|- coordinate(43) (c33);
  
  \draw 
  (x1) -- node[dot](y11)[label={[largeVcolour]right:2}][label={[label distance=-9pt,yshift=2pt]above left:\( y_{11} \)}]{} (c11)
  (x1) .. controls (13).. node[dot](y13)[label={[largeVcolour]right:2}][label=left:\( y_{13} \)]{} (c31)
  (x2)--node[dot](y21)[label={[largeVcolour]right:2}][label=left:\( y_{21} \)]{} (c12)
  (x3)--node[dot](y31)[label={[largeVcolour]left:2}][label={[label distance=-5pt,yshift=-4pt]below left:\( y_{31} \)}]{} (c13)
  (x3)--node[dot](y33)[label={[largeVcolour]right:2}][label={[label distance=-3pt]left:\( y_{33} \)}]{} (c32)
  (x4) .. controls (43).. node[dot](y43)[label={[largeVcolour]right:2}][label={[label distance=-2pt]left:\( y_{43} \)}]{} (c33)
  (x1) .. controls (12).. node[dot](y12)[label={[largeVcolour]right:2}][label=left:\( y_{12} \)]{} (c21)
  (x2)--node[dot](y22)[label={[largeVcolour]right:2}][label=left:\( y_{22} \)]{} (c22)
  (x2)--node[dot](y24)[label={[largeVcolour]right:2}][label={[label distance=-10pt,yshift=-4pt]below left:\( y_{24} \)}]{} (c41)
  (x3)--node[dot](y34)[label={[largeVcolour]right:2}][label=left:\( y_{34} \)]{} (c42)
  (x4)--node[dot](y44)[label={[largeVcolour]left:2}][label={[label distance=-10pt,yshift=-5pt]below right:\( y_{44} \)}]{} (c43)
  (x4) .. controls (42).. node[dot](y42)[label={[largeVcolour]right:2}][label=left:\( y_{42} \)]{} (c23);
  
\draw
  (c11)--node[dot](b11)[midway,label={[largeVcolour]left:0}][label={[font=\scriptsize,xshift=2pt,yshift=2pt]below:\( b_{\tiny 1\!1} \)}]{} (c12)--node[dot](b12)[midway,label={[largeVcolour]right:2}][label={[font=\scriptsize,xshift=2pt]left:\( b_{\tiny 12} \)}]{} (c13)--node[dot](b13)[midway,label={[largeVcolour]above:2}][label={[font=\scriptsize,yshift=4pt]below:\( b_{\tiny 13} \)}]{} (c11)
  (c21)--node[dot](b21)[midway,label={[largeVcolour]left:0}][label={[font=\scriptsize,xshift=0pt,yshift=-2pt]above:\( b_{\tiny 21} \)}]{} (c22)--node[dot](b22)[midway,label={[largeVcolour]right:2}][label={[font=\scriptsize,xshift=2pt]left:\( b_{\tiny 22} \)}]{} (c23)--node[dot](b23)[midway,label={[largeVcolour]below:2}][label={[font=\scriptsize,yshift=-3pt]above:\( b_{\tiny 23} \)}]{} (c21)
  (c31)--node[dot](b31)[midway,label={[largeVcolour]left:2}][label={[font=\scriptsize,xshift=2pt,yshift=2pt]below:\( b_{\tiny 31} \)}]{} (c32)--node[dot](b32)[midway,label={[largeVcolour]right:1}][label={[font=\scriptsize,xshift=2pt]left:\( b_{\tiny 32} \)}]{} (c33)--node[dot](b33)[midway,label={[largeVcolour]above:2}][label={[font=\scriptsize,yshift=4pt]below:\( b_{\tiny 33} \)}]{} (c31)
  (c41)--node[dot](b41)[midway,label={[largeVcolour]left:2}][label={[font=\scriptsize,xshift=0pt,yshift=-2pt]above:\( b_{\tiny 41} \)}]{} (c42)--node[dot](b42)[midway,label={[largeVcolour]right:1}][label={[font=\scriptsize,xshift=2pt]left:\( b_{\tiny 42} \)}]{} (c43)--node[dot](b43)[midway,label={[largeVcolour]below:2}][label={[font=\scriptsize,yshift=-3pt]above:\( b_{\tiny 43} \)}]{} (c41);

\draw [ultra thick]
(c32)--(c33)
(c42)--(c43);
\end{tikzpicture}
}
\caption{}
\label{fig:3-star colouring of G}
\end{subfigure}
\caption{(a)~graph \( H_\mathcal{B} \) for the formula \( (x_1\vee x_2\vee x_3) \) \( \wedge \) \( (x_1\vee x_2\vee x_4) \) \( \wedge \) \( (x_1\!\vee\!x_3\!\vee\!x_4) \) \( \wedge \) \( (x_2\!\vee\!x_3\!\vee\!x_4) \) with a 4-colouring \( h \), (b)~graph obtained by adding edges in \( H_\mathcal{B} \) into \( G_\mathcal{B} \), (c)~graph \( G \) with a 3-star colouring \( \phi \)}
\label{fig:3-star colouring from 4-colouring}
\end{figure}
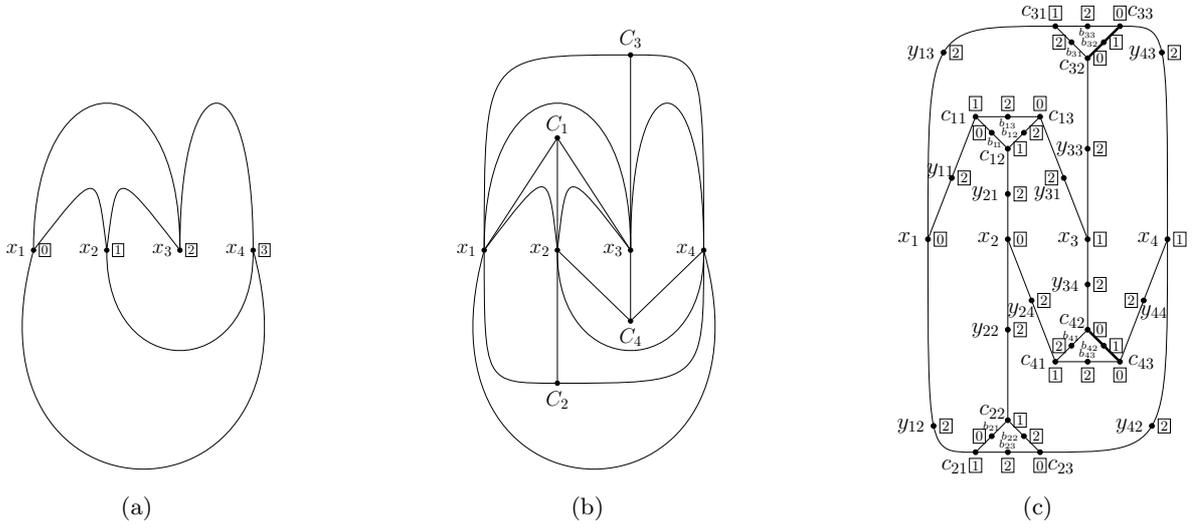

Recall that the graph \( G_\mathcal{B} \) of Theorem~\ref{thm:3-rsc planar bipartite girth 6} is planar. The graph \( H_\mathcal{B} \) is planar because adding edges in \( H_\mathcal{B} \) into \( G_\mathcal{B} \) preserves planarity of \( G_\mathcal{B} \) (see Figure~\ref{fig:H_B on G_B}; for instance, we can add edge \( x_1x_2 \) to \( G_\mathcal{B} \) without affecting planarity by drawing it close to the path \( x_1,C_1,x_2 \)). 
Hence, by 4-colour theorem, \( H_\mathcal{B} \) admits a 4-colouring \( h \) that uses colours 0,1,2,3 (not necessarily an optimal colouring). Let us partition the set of colours \( \{0,1,2,3\} \) into two sets \( L=\{0,1\} \) and \( M=\{2,3\} \). Observe that for each clause \( C_j=\{x_p,x_q,x_r\} \) in the formula \( \mathcal{B} \), \( (x_p,x_q,x_r) \) is a triangle in \( H_\mathcal{B} \). Hence, for each clause \( C_j=\{x_p,x_q,x_r\} \), either \( |M\cap\{h(x_p),h(x_q),h(x_r)\}|=1 \) or \( |L\cap\{h(x_p),h(x_q),h(x_r)\}|=1 \). In the former case, we call \( C_j \) as a type-I clause, and in the latter case, a type-II clause.

Construction of \( \phi \) is as follows. First, colour vertices \( x_i \) by assigning \( \phi(x_i)=0 \) if \( h(x_i)\in L\), and \( \phi(x_i)=1 \) otherwise. Then, all vertices \( y_{ij} \) are coloured 2. Next, for every path \( x_i,y_{ij},c_{jk} \) in \( G \), vertex \( c_{jk} \) is assigned the opposite binary colour of \( x_i \) \big(i.e., \( f(c_{jk})=1-f(x_i) \)\big). Note that for each type-I clause \( C_j \) in \( \mathcal{B} \), exactly one of \( c_{j1},c_{j2},c_{j3} \) is coloured~0 by \( \phi \) (and the other two are coloured~1); similarly, for each type-II clause \( C_j \), exactly one of \( c_{j1},c_{j2},c_{j3} \) is coloured~1 by \( \phi \). For every type-I clause \( C_j \), extend the colouring to the gadget for \( C_j \) using the scheme in Figure~\ref{fig:type-I phi}. For every type-II clause \( C_j \), extend the colouring to the gadget for \( C_j \) using the scheme in Figure~\ref{fig:type-II phi}. An example is exhibited in Figure~\ref{fig:3-star colouring of G}. Obviously, \( \phi \) is a 3-colouring (\( y_{ij} \)'s get colour~2 and their neighbours get binary colours, and clause gadgets are coloured by 3-colouring schemes).\\

\begin{figure}[hbt]
\centering
\begin{subfigure}[b]{0.33\textwidth}
\centering
\begin{tikzpicture}
\tikzset{
dot/.style={draw,fill,circle,inner sep = 0pt,minimum size = 3pt},
vcolour/.style={draw,inner sep=1.5pt,font=\scriptsize,label distance=2pt},
subgraph/.style={draw,ellipse,minimum width=1.75cm,minimum height=2cm},
subgraphHoriz/.style={draw,ellipse,minimum width=1.5cm,minimum height=1.25cm},
}
\draw (0,0) node[dot](cj1)[label=below:\( c_{j1} \)][label={[vcolour,xshift=5pt]above right:0}]{} --++(1,1.5) node[dot](cj2)[label=above:\( c_{j2} \)][label={[vcolour,yshift=-3pt]below:1}]{} --++(1,-1.5) node[dot](cj3)[label=below:\( c_{j3} \)][label={[vcolour,xshift=-5pt]above left:1}]{} --(cj1);
\path (cj1) --node[dot](bj1)[label=left:\( b_{j1} \)][label={[vcolour]right:2}]{} (cj2)
      (cj2) --node[dot](bj2)[label=right:\( b_{j2} \)][label={[vcolour]left:0}]{} (cj3)
      (cj3) --node[dot](bj3)[label=below:\( b_{j3} \)][label={[vcolour]above:2}]{} (cj1);
\end{tikzpicture}
\caption{}
\label{fig:type-I phi}
\end{subfigure}%
\hfill
\begin{subfigure}[b]{0.33\textwidth}
\centering
\begin{tikzpicture}
\tikzset{
dot/.style={draw,fill,circle,inner sep = 0pt,minimum size = 3pt},
vcolour/.style={draw,inner sep=1.5pt,font=\scriptsize,label distance=2pt},
subgraph/.style={draw,ellipse,minimum width=1.75cm,minimum height=2cm},
subgraphHoriz/.style={draw,ellipse,minimum width=1.5cm,minimum height=1.25cm},
}
\draw (0,0) node[dot](cj1)[label=below:\( c_{j1} \)][label={[vcolour,xshift=5pt]above right:1}]{} --++(1,1.5) node[dot](cj2)[label=above:\( c_{j2} \)][label={[vcolour,yshift=-3pt]below:0}]{} --++(1,-1.5) node[dot](cj3)[label=below:\( c_{j3} \)][label={[vcolour,xshift=-5pt]above left:0}]{} --(cj1);
\path (cj1) --node[dot](bj1)[label=left:\( b_{j1} \)][label={[vcolour]right:2}]{} (cj2)
      (cj2) --node[dot](bj2)[label=right:\( b_{j2} \)][label={[vcolour]left:1}]{} (cj3)
      (cj3) --node[dot](bj3)[label=below:\( b_{j3} \)][label={[vcolour]above:2}]{} (cj1);
\end{tikzpicture}
\caption{}
\label{fig:type-II phi}
\end{subfigure}%
\hfill
\begin{subfigure}[b]{0.33\textwidth}
\centering
\begin{tikzpicture}
\tikzset{
dot/.style={draw,fill,circle,inner sep = 0pt,minimum size = 3pt},
vcolour/.style={draw,inner sep=1.5pt,font=\scriptsize,label distance=2pt},
subgraph/.style={draw,ellipse,minimum width=1.75cm,minimum height=2cm},
subgraphHoriz/.style={draw,ellipse,minimum width=1.5cm,minimum height=1.25cm},
}
\draw (0,0) node[dot](cj1)[label=below:\( c_{j1} \)][label={[vcolour,xshift=5pt]above right:2}]{} --++(1,1.5) node[dot](cj2)[label=above:\( c_{j2} \)][label={[vcolour,yshift=-3pt]below:1}]{} --++(1,-1.5) node[dot](cj3)[label=below:\( c_{j3} \)][label={[vcolour,xshift=-5pt]above left:1}]{} --(cj1);
\path (cj1) --node[dot](bj1)[label=left:\( b_{j1} \)][label={[vcolour]right:3}]{} (cj2)
      (cj2) --node[dot](bj2)[label=right:\( b_{j2} \)][label={[vcolour]left:0}]{} (cj3)
      (cj3) --node[dot](bj3)[label=below:\( b_{j3} \)][label={[vcolour]above:3}]{} (cj1);
\end{tikzpicture}
\caption{}
\label{fig:type-II f after recolouring}
\end{subfigure}%
\caption{(a)~the colouring scheme used by \( \phi \) for the gadget of a type-I clause \( C_j \) (`rotate' the scheme if needed), (b)~the colouring scheme used by \( \phi \) for the gadget of a type-II clause \( C_j \) (`rotate' the scheme if needed), (c)~colouring of the gadget for a type-II clause \( C_j \) under \( f \) (after recolouring \( b_{j2} \))}
\end{figure}
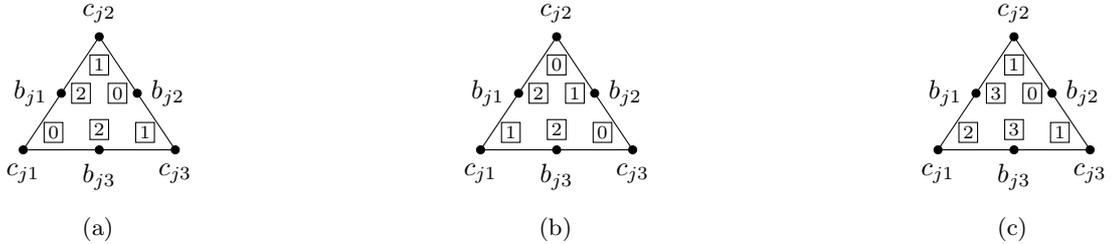

\noindent Claim~1.1: Under \( \phi \), if \( Q \) is a bicoloured \( P_3 \) in \( G \) with a higher colour on its middle vertex, then the following hold.\\
(1)~\( Q \) must be a \( c_{jk},b_{jk},c_{j\,k+1} \) path coloured 0,1,0 by \( \phi \); and\\
(2)~neighbours of \( c_{jk} \) (resp.\ \( c_{j\,k+1} \)) other than \( b_{jk} \) are coloured~2 by \( \phi \)\\
where \( C_j \) is a type-II clause (and index \( k+1 \) is modulo 3).\\[5pt]
\noindent If \( Q \) is a bicoloured \( P_3 \) in \( G \), then either (i)~\( Q \) must be within the gadget for a clause \( C_j \), or (ii)~\( Q \) has a vertex \( y_{ij} \) as an endpoint. In Case (i), either \( Q \) has colour~0 on its middle vertex (when \( C_j \) is a type-I clause), or \( Q \) has colour~1 on its middle vertex (when \( C_j \) is a type-II clause). When \( C_j \) is a type-II clause, either \( Q \) is a \( b_{jk},c_{j\,k+1},b_{j\,k+1} \) path coloured 2,1,2 (eg.: path \( b_{j3},c_{j1},b_{j1} \) in Figure~\ref{fig:type-II phi}) or it is a \( c_{jk},b_{jk},c_{j\,k+1} \) path coloured 0,1,0 (eg.: path \( c_{j2},b_{j2},c_{j3} \) in Figure~\ref{fig:type-II phi}); the middle vertex has a higher colour only in the latter case. In Case~(ii), \( Q \) has a lower colour on its middle vertex because \( y_{ij} \)'s are coloured~2. This proves Point~(1) of Claim~1.1. Point~(2) of Claim~1.1 is true due to the construction of~\( \phi \) (see highlighted \( P_3 \)'s in Figure~\ref{fig:3-star colouring of G}).


We are ready to prove that \( \phi \) is a 3-star colouring. To produce a contradiction, assume that \( w,x,y,z \) is a \( P_4 \) in \( G \) bicoloured by \( \phi \), say \( \phi(w)=\phi(y)=q \) and \( \phi(x)=\phi(z)=p \) where \( 0\leq p<q\leq 2 \). Note that \( x,y,z \) is a bicoloured \( P_3 \) with a higher colour on its middle vertex. Hence, by Point~(1) of Claim~1.1, the path \( x,y,z \) is a \( c_{jk},b_{jk},c_{j\,k+1} \) path coloured 0,1,0 by \( \phi \) where \( 1\leq j\leq m \) and \( 1\leq k\leq 3 \). In particular, \( p=0 \) and \( q=1 \). Note that \( x \) has a neighbour \( w \) coloured \( q=1 \). That is, \( c_{jk} \) or \( c_{j\,k+1} \) has a neighbour \( w\neq b_{jk} \) coloured~1. This is a contradiction to Point~(2) of Claim~1.1. This proves Claim~1. Hence, \( G \) is 3-star colourable. Since \( G \) contains cycles, \( G \) is not 2-star colourable. Therefore, \( \chi_s(G)=3 \).\\

\noindent Claim~2: \( G \) admits a 4-rs colouring \( f \).\\[5pt]
We produce \( f \) by a simple modification to \( \phi \). First, define \( f \) as \( f(v)=\phi(v)+1 \) for each vertex \( v \) of \( G \). 
%
Next, for each path \( Q=c_{jk},b_{jk},c_{j\,k+1} \) coloured 1,2,1\( \bm{,} \) recolour the middle vertex \( b_{jk} \) with colour~0 (see Figure~\ref{fig:type-II f after recolouring}). We claim that this recolouring makes \( f \) a 4-rs colouring of \( G \). Note that if \( Q \) is a  bicoloured \( P_3 \) in \( G \), then either (i)~\( Q \) is within a clause gadget, or (ii)~\( Q \) has a vertex \( y_{ij} \) as an endpoint. In Case~(i), \( Q \) has a lower colour on its middle vertex because \( f \) uses an rs colouring scheme on the gadget for each type-I clause (namely the scheme in Figure~\ref{fig:type-I phi} with colours incremented by one) and each type-II clause (see Figure~\ref{fig:type-II f after recolouring}). In Case~(ii), \( Q \) has a lower colour on its middle vertex because \( y_{ij} \)'s are coloured~3 by \( f \). Therefore, \( f \) is indeed a 4-rs colouring of \( G \).
%
This proves Claim~2. That is, \( \chi_{rs}(G)\leq 4 \). This completes the proof of the theorem.
\end{proof}


\begin{corollary}
\textsc{Min RS Colouring} cannot be approximated by a factor less than 4/3 for subcubic planar bipartite graphs \( G \) even when \( \chi_s(G)=3 \) and \( \text{girth}(G)\geq 6 \).
\end{corollary}
\begin{proof}
By Theorem~\ref{thm:3-rsc planar bipartite etc}, given a subcubic planar bipartite graph \( G \) with \( \chi_s(G)=3 \) and \( \text{girth}(G)\geq 6 \), it is NP-hard to distinguish between the cases \( \chi_{rs}(G)=3 \) and \( \chi_{rs}(G)=4 \). A \( (\frac{4}{3}-\epsilon) \)-approximation algorithm for \textsc{Min RS Colouring} for some \( \epsilon>0 \) can be used to distinguish between the cases \( \chi_{rs}(G)=3 \) and \( \chi_{rs}(G)=4 \). Therefore, \textsc{Min RS Colouring} cannot be approximated within a factor less than \( 4/3 \).
\end{proof}

Similarly, we can strengthen Theorem~\ref{thm:k-rsc planar bipartite girth g} as follows (proof is omitted; see supplementary material).
\begin{theorem}
For \( k\geq 3 \) and \( g\geq 6 \), \textsc{\( k\)-RS Colourability} is NP-complete for planar bipartite graphs \( G \) even when \( G \) is 2-degenerate, \( \Delta(G)=k \), girth\( (G) \geq g \), \( \chi_s(G)\leq k \) and \( \chi_{rs}(G)\leq k+1 \). \qed
\label{thm:k-rsc planar bip etc}
\end{theorem}


\noindent \textbf{Inapproximability of \textsc{Min RS Colouring}}\\
By an approximation factor preserving reduction from \textsc{Min Colouring} to \textsc{Min RS Colouring}, we show that for all \( \epsilon>0 \), \textsc{Min RS Colouring} is inapproximable within \( n^{\frac{1}{3}-\epsilon} \) for 2-degenerate bipartite graphs. The construction employed is a slightly modified form of that used by Gebremedhin \etal.~\cite{gebremedhin3} to prove that \textsc{Min Star Colouring} is inapproximable.
\begin{theorem}
For all \( \epsilon>0 \), it is NP-hard to approximate \textsc{Min RS Colouring} within \( n^{\frac{1}{3}-\epsilon} \) for 2-degenerate bipartite graphs.
\label{thm:rs col inapproximable}
\end{theorem}
\begin{proof}


We show that an approximation algorithm for \textsc{Min RS Colouring} leads to an approximation algorithm for \textsc{Min Colouring}. We assume that for a given \( \epsilon>0 \), there is an algorithm that takes a graph \( G\bm{'} \) on \( N \) vertices as input and approximates \( \chi_{rs}(G\bm{'}) \) within a factor of \( N^{\frac{1}{3}-\epsilon} \). Then, we show that for a given \( \epsilon>0 \), there is an algorithm that takes a graph \( G \) on \( n \) vertices as input and approximates \( \chi(G) \) within a factor of \( n^{1-\epsilon} \). Given a graph \( G \) with \( n \) vertices, \( m \) edges and maximum degree \( \Delta \), construct a graph \( G\bm{'} \) from \( G \) by replacing each edge \( e=uv \) of \( G \) by a copy of the complete bipartite graph \( K_{2,\Delta+1} \) with parts \( \{u,v\} \) and \( \{e_1,e_2,\dots,e_{\Delta+1}\} \) (see Figure~\ref{fig:making G'}).

\begin{figure}[hbt]
\centering
\begin{tikzpicture}
\path (0,0) coordinate(lhs);
\path (lhs) node[dot](v)[label=left:\( u \)]{} --++(1,0) node[dot](u)[label=right:\( v \)]{} --++(1,0) coordinate(from) --++(0.35,0) coordinate (to) --++(1,0) coordinate(rhs);
\draw (v)-- node[below]{e}  (u);

\path [-stealth,draw=black,ultra thick] (from)--(to);

\draw (rhs) node[dot](v)[label=left:\( u \)]{};
\path (v)--+(3,0) node[dot](u)[label=right:\( v \)]{};
\draw (v)--+(1.5,1) node[dot](e1)[label=above:\( e_1 \)]{}
      (v)--+(1.5,0.25) node[dot](e2)[label=above:\( e_2 \)]{}
      (v)--+(1.5,-1) node[dot](eD+1)[label=below:\( e_{\Delta+1} \)]{};
\draw (u)--(e1)
      (u)--(e2)
      (u)--(eD+1);
\path (e2)-- node[sloped]{\( \dots \)} (eD+1);
\end{tikzpicture}
\caption{Edge replacement operation to construct \( G\bm{'} \) from \( G \)}
\label{fig:making G'}
\end{figure}
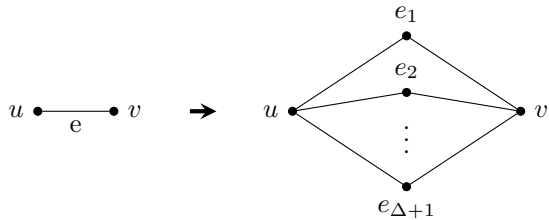

Note that \( G\bm{'} \) has \( N=n+(\Delta+1)m<n+(n+1)\binom{n}{2}\leq n^3 \) vertices. Clearly, \( G\bm{'} \) is bipartite, and \( G\bm{'} \) can be constructed in time polynomial in \( m+n \). 
To obtain a 2-degenerate ordering of \( V(G\bm{'}) \), list members of \( V(G) \) first, followed by the newly introduced vertices (i.e., \( e_i \)'s).

First, we prove that \( G\bm{'} \) is \( (k+1) \)-rs colourable whenever \( G \) is \( k \)-colourable. Clearly, each \( k \)-colouring \( f \) of \( G \) can be extended into a \( (k+1) \)-colouring \( f\bm{'} \) of \( G\bm{'} \) by assigning the new colour \( k \) to all new vertices. A \( P_3 \) in \( G\bm{'} \) is either of the form \( u,e_i,v \) where \( uv \) is an edge in \( G \), or of the form \( e_i,v,e'_j \). In the former case, the \( P_3 \) is tricoloured by \( f\bm{'} \) because \( f(u)\neq f(v) \). In the latter case, the \( P_3 \) is bicoloured by \( f\bm{'} \) (since \( f\bm{'}(e_i)=f\bm{'}(e'_j)=k \)), but the middle vertex has a lower colour. Hence, \( f\bm{'} \) is a \( (k+1) \)-rs colouring of \( G\bm{'} \). This proves that \( \chi_{rs}(G\bm{'})\leq \chi(G)+1 \). 

Next, we show that given a \( (k+1) \)-rs colouring \( f\bm{'} \) of \( G\bm{'} \), a \( k \)-colouring \( f \) of \( G \) can be obtained in polynomial time. For \( k\geq \Delta+1 \), a greedy colouring of \( G \) can be used as \( f \). Suppose \( k\leq \Delta \). For each edge \( e=uv \) of \( G \), at least two from associated vertices \( e_1,e_2,\dots,e_{\Delta+1} \) in \( G\bm{'} \) must get the same colour under \( f\bm{'} \) (if not, \( \Delta+1 \) colours are needed for \( e_i \)'s and another colour is needed for \( u \), a contradiction because \( f\bm{'} \) is a \( (k+1) \)-colouring and \( k\leq \Delta \)). W.l.o.g., assume that \( f\bm{'}(e_1)=f\bm{'}(e_2) \). Then, \( f\bm{'} \) must use distinct colours on vertices \( u \) and \( v \) (if not, \( u,e_1,v,e_2 \) is a bicoloured \( P_4 \), and hence \( f\bm{'} \) is not even a star colouring let alone an rs colouring). Moreover, for every vertex \( v\in V(G) \), \( deg_{G\bm{'}}(v)\geq k+1 \) and thus \( f\bm{'}(v)<k \) (if \( v \) is coloured \( k \), then \( \deg_{G\bm{'}}(v)\leq k \) because \( v \) has at most one neighbour coloured \( i \) for \( 0\leq i\leq k-1 \)). Hence, the restriction of \( f\bm{'} \) to \( V(G) \) is indeed a \( k \)-colouring of \( G \). 

Finally, for a given \( \epsilon>0 \), we produce an algorithm that approximates \( \chi(G) \) within \( n^{1-\epsilon} \). Let \( n_0 \) be a fixed positive integer such that \( n_0^{2\epsilon}\geq 2 \). If \( n<n_0 \), then \( \chi(G) \) can be computed exactly using an exact algorithm. Suppose \( n\geq n_0 \). Now, the assumed approximation algorithm (see the beginning of the proof) gives a \( (k+1) \)-rs colouring of \( G\bm{'} \) where \( k+1\leq N^{\frac{1}{3}-\epsilon}\chi_{rs}(G\bm{'}) \). Using the method described above, a \( k \)-colouring of \( G \) can be produced. Since \( \chi_{rs}(G\bm{'})\leq \chi(G)+1 \), \( N<n^3 \) and \( 2\leq n_0^{2\epsilon} \), we have \( k<n^{1-\epsilon}\chi(G) \) by the following equation.
\[
\begin{multlined}
k<k+1\leq N^{\frac{1}{3}-\epsilon}\chi_{rs}(G\bm{'})< (n^3)^{\frac{1}{3}-\epsilon}(\chi(G)+1)\\[3pt]
\leq n^{1-3\epsilon}\,2\chi(G) \leq n^{1-3\epsilon}n_0^{2\epsilon}\chi(G)\leq n^{1-3\epsilon}n^{2\epsilon}\chi(G)=n^{1-\epsilon}\chi(G)
\end{multlined}
\]
This gives an \( (n^{1-\epsilon}) \)-approximation algorithm for \textsc{Min Colouring}. So, for a given \( \epsilon>0 \), if there is an \( (n^{\frac{1}{3}-\epsilon}) \)-approximation algorithm for \textsc{Min RS Colouring}, then there is an \( (n^{1-\epsilon}) \)-approximation algorithm for \textsc{Min Colouring}. 
Since it is NP-hard to approximate \textsc{Min Colouring} within \( n^{1-\epsilon} \) for all \( \epsilon>0 \) \cite{zuckerman}, it is NP-hard to approximate \textsc{Min RS Colouring} within \( n^{\frac{1}{3}-\epsilon} \) for all \( \epsilon>0 \).
\end{proof}

Karpas \etal.\ provided a method to produce an rs colouring of a given 2-degenerate graph \( G \) in polynomial time using  \( O(n^\frac{1}{2}) \) colours (special case \( d=2 \) of Theorem~6.2 in \cite{karpas}). 
Hence, the rs chromatic number is approximable within \( O(n^\frac{1}{2}) \) in 2-degenerate graphs. Along with Theorem~\ref{thm:rs col inapproximable}, this means that in 2-degenerate graphs, \textsc{Min RS Colouring} can be approximated within \( O(n^\frac{1}{2}) \), yet inapproximable within \( n^{\frac{1}{3}-\epsilon} \) for all \( \epsilon>0 \). The following open problem is of interest in this regard.\\

\noindent\textbf{Question:} Is \textsc{Min RS Colouring} inapproximable within \( n^p \) for some \( p \) in the range \( \frac{1}{3}\leq p<\frac{1}{2} \) ?

\section{Properties of RS Colouring}\label{sec:properties}
In this section, some basic observations on rs colouring are presented.
From these observations, an easy formula for the rs chromatic number of split graphs is obtained. Also, the material in this section is used in the forthcoming section on trees and chordal graphs.

\begin{observation}
For every graph \( G \), \( \chi_{rs}(G)\leq n-\alpha(G)+1 \).
\label{obs:n-alpha+1}
\end{observation}
\begin{proof}
Let \( v_0,v_1,\dots,v_{n-1} \) be the vertices in \( G \). Let \( \alpha=\alpha(G) \), and let \( I \) be a maximum independent set in \( G \). W.l.o.g., assume that \( I=\{v_{n-\alpha},\dots,v_{n-1}\} \). Define \( f:V(G)\to\{0,1,\dots,n-\alpha\} \) as \( f(v_i)=i \) for \( 0\leq i\leq n-\alpha-1 \), and \( f(v_i)=n-\alpha \) for \( n-\alpha\leq i\leq n-1 \) (i.e., assign colour \( n-\alpha \) on members of \( I \)). Since \( f^{-1}(i) \) is a singleton for \( 0\leq i\leq n-\alpha-1 \) and \( f^{-1}(n-\alpha)=I \), \( f \) is an \( (n-\alpha+1) \)-rs colouring of \( G \). Therefore, \( \chi_{rs}(G)\leq n-\alpha(G)+1 \).
\end{proof}

The next observation is a direct consequence of Observation~\ref{obs:degree restriction} (proof is available in supplementary material).
\begin{observation}
Let \( G \) be a graph, and \( C=\{v_1,v_2,\dots,v_k\} \) be a clique in \( G \). If \( \deg(v_i)\geq k \) for all \( i \), then \( G \) is not \( k \)-rs colourable.
\qed
\label{obs:clique degree restriction}
\end{observation}

\begin{theorem}
If \( G \) is a split graph, then \( \chi_{rs}(G)=n-\alpha(G)+1 \).
\end{theorem}
\begin{proof}
Let \( G \) be a split graph whose vertex set is partitioned into a clique \( C \) and a maximum independent set \( I \) \big(i.e., \( |I|=\alpha(G) \)\big). Since \( I \) is a maximum independent set, every vertex in \( C \) has a neighbour in \( I \). Hence, each vertex in \( C \) has degree at least \( |C| \). Therefore, by Observation~\ref{obs:clique degree restriction}, \( G \) is not \( |C| \)-rs colourable. That is, \( \chi_{rs}(G)\geq |C|+1=n-\alpha(G)+1 \). On the other hand, by Observation~\ref{obs:n-alpha+1}, \( \chi_{rs}(G)\leq n-\alpha(G)+1 \). Therefore, \( \chi_{rs}(G)=n-\alpha(G)+1 \).
\end{proof}

Next, we present a lemma on 3-rs colouring of paths which is used in the design of 3-rs colourability testing algorithm for trees. The aforementioned algorithm for a tree \( T \) deals with paths in \( T \) whose internal vertices have degree two in \( T \) and endpoints are 3-plus vertices in \( T \). Since every 3-rs colouring of a tree must use binary colours on 3-plus vertices, the following question becomes the centre of attention. For \( i,j\in\{0,1\} \) and \( n\geq 3 \), what are the values of \( n,i,j \) such that the path \( P_n \) admits a 3-rs colouring that uses colours \( i \) and \( j \) at its endpoints? By Properties~P3, P4 and P6, such a 3-rs colouring does not exist in the following cases: (i)~\( n=3 \), \( i=j=0 \), (ii)~\( n=4 \), \( i\neq j \), and (iii)~\( n=6 \), \( i=j=0 \). Soon, we show that such a 3-rs colouring is possible in all other cases. Figure~\ref{fig:3-rsc path small n} shows that this is the case when \( 2<n<7 \). 
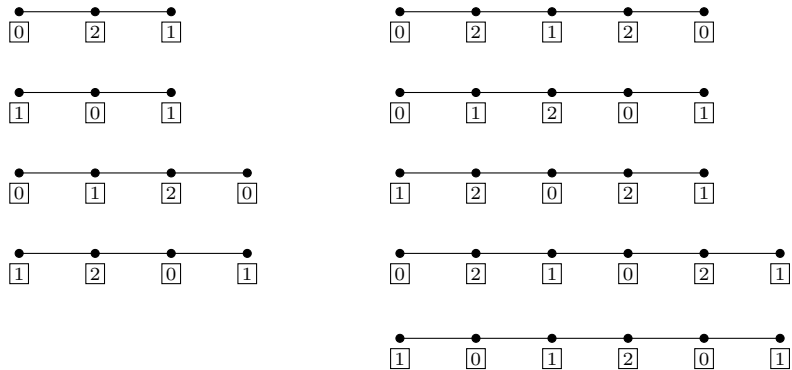
\begin{figure}[hbt]
\centering
\begin{tikzpicture}
\matrix[matrix of nodes,row sep=0.6cm]{
\draw (0,0) node[dot](a1)[label={[vcolour]below:0}]{} --++(1,0) node[dot](a2)[label={[vcolour]below:2}]{} --++(1,0) node[dot](a3)[label={[vcolour]below:1}]{};
         &[1.75cm]
	 \draw (0,0) node[dot](a1)[label={[vcolour]below:0}]{} --++(1,0) node[dot](a2)[label={[vcolour]below:2}]{} --++(1,0) node[dot](a3)[label={[vcolour]below:1}]{} --++(1,0) node[dot](a4)[label={[vcolour]below:2}]{} --++(1,0) node[dot](a5)[label={[vcolour]below:0}]{};\\
\draw (0,0) node[dot](a1)[label={[vcolour]below:1}]{} --++(1,0) node[dot](a2)[label={[vcolour]below:0}]{} --++(1,0) node[dot](a3)[label={[vcolour]below:1}]{};
         &
	 \draw (0,0) node[dot](a1)[label={[vcolour]below:0}]{} --++(1,0) node[dot](a2)[label={[vcolour]below:1}]{} --++(1,0) node[dot](a3)[label={[vcolour]below:2}]{} --++(1,0) node[dot](a4)[label={[vcolour]below:0}]{} --++(1,0) node[dot](a5)[label={[vcolour]below:1}]{};\\
\draw (0,0) node[dot](a1)[label={[vcolour]below:0}]{} --++(1,0) node[dot](a2)[label={[vcolour]below:1}]{} --++(1,0) node[dot](a3)[label={[vcolour]below:2}]{} --++(1,0) node[dot](a4)[label={[vcolour]below:0}]{};
         &
	 \draw (0,0) node[dot](a1)[label={[vcolour]below:1}]{} --++(1,0) node[dot](a2)[label={[vcolour]below:2}]{} --++(1,0) node[dot](a3)[label={[vcolour]below:0}]{} --++(1,0) node[dot](a4)[label={[vcolour]below:2}]{} --++(1,0) node[dot](a5)[label={[vcolour]below:1}]{};\\
\draw (0,0) node[dot](a1)[label={[vcolour]below:1}]{} --++(1,0) node[dot](a2)[label={[vcolour]below:2}]{} --++(1,0) node[dot](a3)[label={[vcolour]below:0}]{} --++(1,0) node[dot](a4)[label={[vcolour]below:1}]{};
         &
	 \draw (0,0) node[dot](a1)[label={[vcolour]below:0}]{} --++(1,0) node[dot](a2)[label={[vcolour]below:2}]{} --++(1,0) node[dot](a3)[label={[vcolour]below:1}]{} --++(1,0) node[dot](a4)[label={[vcolour]below:0}]{} --++(1,0) node[dot](a5)[label={[vcolour]below:2}]{} --++(1,0) node[dot](a6)[label={[vcolour]below:1}]{};\\

         &
	 \draw (0,0) node[dot](a1)[label={[vcolour]below:1}]{} --++(1,0) node[dot](a2)[label={[vcolour]below:0}]{} --++(1,0) node[dot](a3)[label={[vcolour]below:1}]{} --++(1,0) node[dot](a4)[label={[vcolour]below:2}]{} --++(1,0) node[dot](a5)[label={[vcolour]below:0}]{} --++(1,0) node[dot](a6)[label={[vcolour]below:1}]{};\\
};
\end{tikzpicture}
\caption{For \( 2<n<7 \) and \( i,j\in\{0,1\} \), \( P_n \) can be 3-rs coloured with colours \( i \) and \( j \) at its endpoints except for the cases (i)~\( n=3 \), \( i=j=0 \), (ii)~\( n=4 \), \( i\neq j \), and (iii)~\( n=6 \), \( i=j=0 \)}
\label{fig:3-rsc path small n}
\end{figure}

Lemma~\ref{lem:extension to path} below guarantees that such a 3-rs colouring exists whenever \( n\geq 7 \).

\begin{lemma}
If \( n\geq 7 \), for every \( i,j\in\{0,1\} \), the path \( P_n \) has a 3-rs colouring with one endpoint coloured \( i \) and the other, coloured \( j \).
\label{lem:extension to path}
\end{lemma}
\begin{proof}
(by mathematical induction on \( n \))\\
Base cases: \( n\in\{7,8\} \)\\
For \( n\in\{7,8\} \) and \( i,j\in\{0,1\} \), there exists a 3-rs colouring of \( P_n \) with endpoints coloured \( i,j \) as shown in Figure~\ref{fig:3-rsc medium n}.
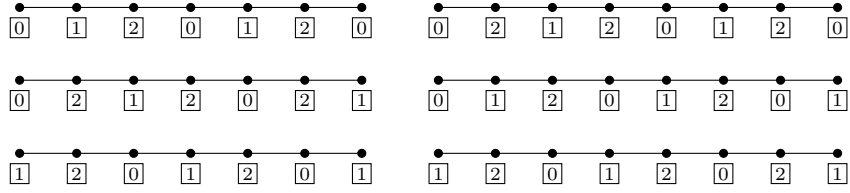
\begin{figure}[hbt]
\centering
\begin{tikzpicture}
\matrix[matrix of nodes,row sep=0.5cm]{
\draw (0,0) node[dot](a1)[label={[vcolour]below:0}]{} --++(0.75,0) node[dot](a2)[label={[vcolour]below:1}]{} --++(0.75,0) node[dot](a3)[label={[vcolour]below:2}]{} --++(0.75,0) node[dot](a4)[label={[vcolour]below:0}]{} --++(0.75,0) node[dot](a5)[label={[vcolour]below:1}]{} --++(0.75,0) node[dot](a6)[label={[vcolour]below:2}]{} --++(0.75,0) node[dot](a7)[label={[vcolour]below:0}]{};
         &[0.75cm]
	 \draw (0,0) node[dot](a1)[label={[vcolour]below:0}]{} --++(0.75,0) node[dot](a2)[label={[vcolour]below:2}]{} --++(0.75,0) node[dot](a3)[label={[vcolour]below:1}]{} --++(0.75,0) node[dot](a4)[label={[vcolour]below:2}]{} --++(0.75,0) node[dot](a5)[label={[vcolour]below:0}]{} --++(0.75,0) node[dot](a6)[label={[vcolour]below:1}]{} --++(0.75,0) node[dot](a7)[label={[vcolour]below:2}]{} --++(0.75,0) node[dot](a8)[label={[vcolour]below:0}]{};\\
\draw (0,0) node[dot](a1)[label={[vcolour]below:0}]{} --++(0.75,0) node[dot](a2)[label={[vcolour]below:2}]{} --++(0.75,0) node[dot](a3)[label={[vcolour]below:1}]{} --++(0.75,0) node[dot](a4)[label={[vcolour]below:2}]{} --++(0.75,0) node[dot](a5)[label={[vcolour]below:0}]{} --++(0.75,0) node[dot](a6)[label={[vcolour]below:2}]{} --++(0.75,0) node[dot](a7)[label={[vcolour]below:1}]{};
         &
	 \draw (0,0) node[dot](a1)[label={[vcolour]below:0}]{} --++(0.75,0) node[dot](a2)[label={[vcolour]below:1}]{} --++(0.75,0) node[dot](a3)[label={[vcolour]below:2}]{} --++(0.75,0) node[dot](a4)[label={[vcolour]below:0}]{} --++(0.75,0) node[dot](a5)[label={[vcolour]below:1}]{} --++(0.75,0) node[dot](a6)[label={[vcolour]below:2}]{} --++(0.75,0) node[dot](a7)[label={[vcolour]below:0}]{} --++(0.75,0) node[dot](a8)[label={[vcolour]below:1}]{};\\
\draw (0,0) node[dot](a1)[label={[vcolour]below:1}]{} --++(0.75,0) node[dot](a2)[label={[vcolour]below:2}]{} --++(0.75,0) node[dot](a3)[label={[vcolour]below:0}]{} --++(0.75,0) node[dot](a4)[label={[vcolour]below:1}]{} --++(0.75,0) node[dot](a5)[label={[vcolour]below:2}]{} --++(0.75,0) node[dot](a6)[label={[vcolour]below:0}]{} --++(0.75,0) node[dot](a7)[label={[vcolour]below:1}]{};
         &
	 \draw (0,0) node[dot](a1)[label={[vcolour]below:1}]{} --++(0.75,0) node[dot](a2)[label={[vcolour]below:2}]{} --++(0.75,0) node[dot](a3)[label={[vcolour]below:0}]{} --++(0.75,0) node[dot](a4)[label={[vcolour]below:1}]{} --++(0.75,0) node[dot](a5)[label={[vcolour]below:2}]{} --++(0.75,0) node[dot](a6)[label={[vcolour]below:0}]{} --++(0.75,0) node[dot](a7)[label={[vcolour]below:2}]{} --++(0.75,0) node[dot](a8)[label={[vcolour]below:1}]{};\\
};
\end{tikzpicture}
\caption{For \( n\in\{7,8\} \), \( P_n \) can be 3-rs coloured with colours \( i,j\in \{0,1\} \) at the endpoints}
\label{fig:3-rsc medium n}
\end{figure}

\noindent Induction step: \( (n\geq 9) \)\\
Let the path be \( v_1,v_2,\dots,v_n \). We need to show that the path has a 3-rs colouring \( f \) such that \( f(v_1)=i \) and \( f(v_n)=j \). By the induction hypothesis, the path \( v_3,v_4,\dots,v_n \) has a 3-rs colouring \( f\bm{'} \) such that \( f\bm{'}(v_3)=1-i \) and \( f\bm{'}(v_n)=j \). By assigning \( f\bm{'}(v_2)=2 \) and \( f\bm{'}(v_1)=i \), \( f\bm{'} \) can be extended into a 3-rs colouring of the whole path that uses colours \( i \) and \( j \) at its endpoints.
\end{proof}

The following theorem sums up the picture.
\begin{theorem}
Let \( v_1,v_2,\dots,v_n \) be a path where \( n\in\mathbb{N} \), and let \( i,j\in\{0,1\} \). The path does not admit a 3-rs colouring \( f \) with \( f(v_1)=i \) and \( f(v_n)=j \) in the following cases: (i)~\( n=2 \), \( i=j \); (ii)~\( n=3 \), \( i=j=0 \); (iii)~\( n=4 \), \( i\neq j \); (iv)~\( n=6 \), \( i=j=0 \). In all other cases, the path admits a 3-rs colouring \( f \) such that \( f(v_1)=i \) and \( f(v_n)=j \).
\label{thm:3-rsc extension paths}
\end{theorem}

\section{Trees and Chordal Graphs}\label{sec:trees and chordal}
For every tree \( T \), \( \chi_{rs}(T)=O(\log n/\log \log n) \) and the bound is tight \cite{karpas}. But, the claim \cite{karpas} that \( \chi_{rs}(T)=\Omega(\log n/\log \log n) \) is false; for instance, all paths are 3-rs colourable no matter how long. Since every tree \( T \) admits a distance-two colouring with \( \Delta(T)+1 \) colours \cite{lih}, \( \chi_{rs}(T)\leq \Delta(T)+1 \) (recall that every distance-two colouring is an rs colouring). This section presents a linear-time algorithm to test 3-rs colourability of trees as well as an \( O(n^3) \)-time algorithm to test 3-rs colourability of chordal graphs. 

Note that the decision problem \textsc{\( k \)-RS Colourability} can be expressed in MSO (see supplementary material), and hence can be solved in linear time in graphs of bounded treewidth by Courcelle's theorem \cite{borie,courcelle} (in fact, the problem can be expressed in MSO\( _1 \), MSO without edge set quantification). Unfortunately, algorithms obtained from the MSO expression of problems suffer from extremely large constants hidden in the big-O notation \cite{langer}. In contrast, the algorithm for trees provided in this section has a hidden constant of reasonable size (less than 35) making it practically useful.

Note that our algorithm for trees only tests 3-rs colourability of an input tree \( T \) with a 3-plus vertex; it does not produce a 3-rs colouring of \( T \) in case \( T \) is 3-rs colourable (it is possible to extend our algorithm into a 3-rs colouring algorithm for trees; but the extension requires a considerable amount of work).\\

\noindent\textbf{Overview of the algorithm}\\
Consider a rooted tree representation of the input tree \( T \) so that we can process it in a bottom-up fashion. The processing part of the algorithm hinges on the following question: \emph{What is the `role' of each connected subgraph of \( T \) in deciding 3-rs colourability of \( T \)?} We categorize relevant connected subgraphs of trees into classes so that members of the same class play the same `role' (in deciding 3-rs colourability of the tree). As a result, the 3-rs colourability status of a tree is preserved when a connected subgraph is locally replaced by another member of the same class. This notion of local replacement allows us to formalize the classification of connected subgraphs. To this end, it suffices to consider two kinds of connected subgraphs called \emph{branches} and \emph{rooted subtrees} (defined below). For convenience, we deal with 3-rs colouring extension (motivation is explained below in detail). So, branches and rooted subtrees may contain coloured vertices, and we are concerned of 3-rs colouring extension status of trees containing them (i.e., whether we can assign colours on the rest of the vertices to produce a 3-rs colouring). We introduce an equivalence relation among branches (resp.\ rooted subtrees) as \( B_1\sim B_2 \) if replacing \( B_1 \) by \( B_2 \) in every tree preserves 3-rs colouring extension status. 
Our algorithm recursively computes the equivalence classes of branches and rooted subtrees of the input tree. The input tree \( T \) itself is technically a rooted subtree, and we can determine 3-rs colourability of \( T \) based on the equivalence class of that rooted subtree.\\

\noindent\textbf{Ingredients of the algorithm and motivation}\\
There are two basic ideas behind the algorithm. The first idea is that some subgraphs force colours. For example, suppose \( u,v,w \) is a path in a tree \( T \) where \( u,v \) and \( w \) are 3-plus vertices in \( T \). If \( f \) is a 3-rs colouring of \( T \), then \( f \) must use binary colours on \( u,v \) and \( w \), and due to Property~P3, the only possibility is \( f(v)=0 \) and \( f(u)=f(w)=1 \). In short, the subgraph \( H_1 \) displayed in Figure~\ref{fig:H1} forces colours.

\begin{figure}[h]
\centering
\begin{tikzpicture}
\draw (0,0) node[dot]{}--++(1,0) node[dot](v)[label=above:\( u \)][label={[vcolour]below left:1}]{}--++(1,0) node[dot](u)[label=above:\( v \)][label={[vcolour]below left:0}]{}--++(1,0) node[dot](w)[label=above:\( w \)][label={[vcolour]below left:1}]{}--++(1,0) node[dot]{};
\draw (v)--+(0,-1) node[dot]{}
      (u)--+(0,-1) node[dot]{}
      (w)--+(0,-1) node[dot]{};
\end{tikzpicture}
\caption{Tree \( H_1 \) (with colours forced)}
\label{fig:H1}
\end{figure}
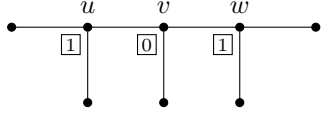

The second idea is more ubiquitous in colouring: we can combine 3-rs colourings of subgraphs to get a 3-rs colouring of the whole graph provided some simple conditions are met. 
In this context, it takes the following form: we can decompose a tree into two subgraphs with an edge (or a vertex) in common, and combine 3-rs colourings of the subgraphs to produce a 3-rs colouring of the whole tree provided the subgraph colourings agree on common vertices (or the common vertex has colour~0). 

\begin{lemma}
Let \( T \) be a tree, and \( e=u_1u_2 \) be an edge in \( T \). Let \( U_1 \) and \( U_2 \) be the vertex sets of components in \( T-e \) such that \( u_1\in U_1 \) and \( u_2\in U_2 \). Suppose \( G_1=T[U_1\cup \{u_2\}] \) admits a 3-rs colouring \( f_1 \), and \( G_2=T[U_2\cup \{u_1\}] \) admits  a 3-rs colouring \( f_2 \). If \( f_1 \) and \( f_2 \) agree on the common vertices, then they can be combined to give a 3-rs colouring \( f \) of \( T \). That is, if \( f_1(u_1)=f_2(u_1) \) and \( f_1(u_2)=f_2(u_2) \), then the function \( f \) defined on \( V(T) \) as \( f(w)=f_i(w) \) for \( w\in U_i \), \( i=1,2 \) is a 3-rs colouring of \( T \).
\label{lem:joining at edge}
\end{lemma}
\begin{proof}
Note that every 3-vertex path in \( T \) is either entirely in \( G_1 \), or entirely in \( G_2 \). Since \( f \) restricted to \( V(G_i) \) is a 3-rs colouring of \( G_i \) for \( i=1,2 \) (namely \( f_i \)), \( f \) is indeed a 3-rs colouring of \( T \).
\end{proof}

\begin{lemma}
Let \( T \) be a tree, and \( v \) be a vertex of \( T \). Let \( U_1 \) be the vertex set of a component in \( T-v \), and let \( U_2=V(T-v)\setminus U_1 \). Suppose that \( G_i=T[U_i\cup \{v\}] \) admits a 3-rs colouring \( f_i \) such that \( f_i(v)=0 \) for \( i=1,2 \). Then, \( f_1 \) and \( f_2 \) can be combined to produce a 3-rs colouring \( f \) of \( T \) \( ( \)i.e., \( f(w)=f_i(w) \) \( \forall w\in V(G_i) \), \( i=1,2 \)\( ) \).
\label{lem:joining at vertex}
\end{lemma}
\begin{proof}
Clearly, for each 3-vertex path \( Q \) in \( T \), either (i)~\( Q \) is entirely in \( G_i \) for some \( i\in\{1,2\} \), or (ii)~\( Q \) has \( v \) as its middle vertex. Since \( f_1 \) and \( f_2 \) are 3-rs colourings, and \( v \) is coloured~0, \( Q \) has a lower colour on its middle vertex. Since \( Q \) is arbitrary, under \( f \), there is no bicoloured \( P_3 \) in \( T \) with a higher colour on its middle vertex.
\end{proof}

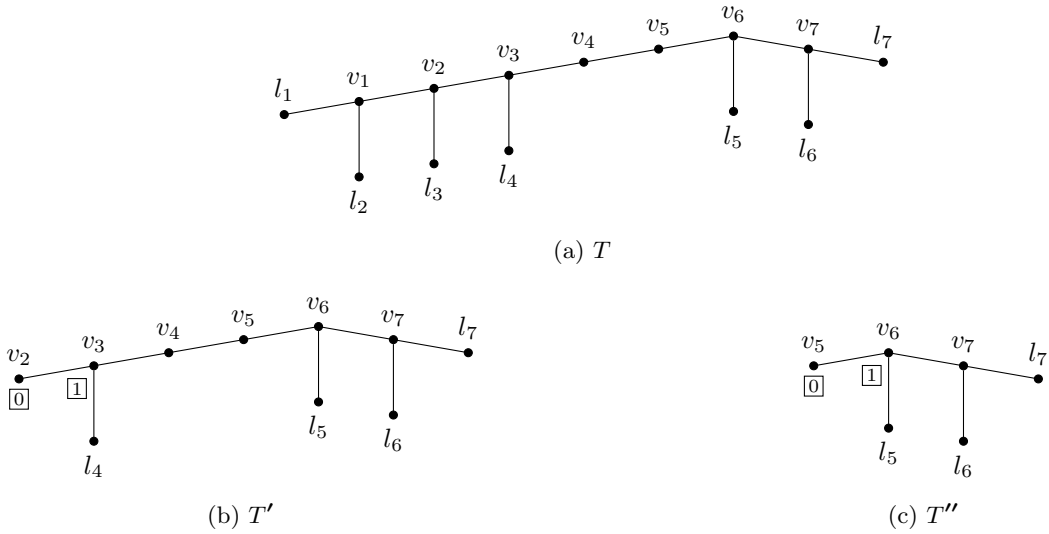
\begin{figure}[h]
\centering
\begin{subfigure}[t]{1\textwidth}
\centering
\begin{tikzpicture}
\draw (0,0) node[dot][label=above:\( l_1 \)]{}--++(10:1) node[dot](a1)[label=above:\( v_1 \)]{}--++(10:1) node[dot](a2)[label=above:\( v_2 \)]{}--++(10:1) node[dot](a3)[label=above:\( v_3 \)]{}--++(10:1) node[dot](a4)[label=above:\( v_4 \)]{}--++(10:1) node[dot](a5)[label=above:\( v_5 \)]{}--++(10:1) node[dot](a6)[label=above:\( v_6 \)]{}--++(-10:1) node[dot](a7)[label=above:\( v_7 \)]{}--++(-10:1) node[dot][label=above:\( l_7 \)]{};
\draw (a1)--+(0,-1) node[dot][label=below:\( l_2 \)]{}
      (a2)--+(0,-1) node[dot][label=below:\( l_3 \)]{}
      (a3)--+(0,-1) node[dot][label=below:\( l_4 \)]{}
      (a6)--+(0,-1) node[dot][label=below:\( l_5 \)]{}
      (a7)--+(0,-1) node[dot][label=below:\( l_6 \)]{};
\end{tikzpicture}
\caption{\( T \)}
\label{fig:example T}
\vspace*{0.25cm}
\end{subfigure}%

\begin{subfigure}[t]{0.5\textwidth}
\centering
\begin{tikzpicture}
\draw (0,0) node[dot](a2)[label=above:\( v_2 \)][label={[vcolour]below:0}]{}--++(10:1) node[dot](a3)[label=above:\( v_3 \)][label={[vcolour,yshift=-2pt]below left:1}]{}--++(10:1) node[dot](a4)[label=above:\( v_4 \)]{}--++(10:1) node[dot](a5)[label=above:\( v_5 \)]{}--++(10:1) node[dot](a6)[label=above:\( v_6 \)]{}--++(-10:1) node[dot](a7)[label=above:\( v_7 \)]{}--++(-10:1) node[dot][label=above:\( l_7 \)]{};
\draw (a3)--+(0,-1) node[dot][label=below:\( l_4 \)]{}
      (a6)--+(0,-1) node[dot][label=below:\( l_5 \)]{}
      (a7)--+(0,-1) node[dot][label=below:\( l_6 \)]{};
\end{tikzpicture}
\caption{\( T\bm{'} \)}
\label{fig:example T'}
\end{subfigure}%
\begin{subfigure}[t]{0.5\textwidth}
\centering
\begin{tikzpicture}
\draw (0,0) node[dot](a5)[label=above:\( v_5 \)][label={[vcolour]below:0}]{}--++(10:1) node[dot](a6)[label=above:\( v_6 \)][label={[vcolour,yshift=-2pt]below left:1}]{}--++(-10:1) node[dot](a7)[label=above:\( v_7 \)]{}--++(-10:1) node[dot][label=above:\( l_7 \)]{};
\draw (a6)--+(0,-1) node[dot][label=below:\( l_5 \)]{}
      (a7)--+(0,-1) node[dot][label=below:\( l_6 \)]{};
\end{tikzpicture}
\caption{\( T\bm{''} \)}
\label{fig:example T''}
\end{subfigure}
\caption{An example of testing 3-rs colourability of a tree}
\label{fig:3-rsc tree example}
\end{figure}

To give a flavour of the problem, we present an example for testing 3-rs colourability of trees (without using the algorithm). If the rooted tree \( T \) displayed in Figure~\ref{fig:3-rsc tree example}a admits a 3-rs colouring \( f \), then \( f(v_2)=0 \) and \( f(v_1)=f(v_3)=1 \) due to the presence of the subgraph \( H_1 \) (see Figure~\ref{fig:H1}). 
Therefore, if \( T \) admits a 3-rs colouring \( f \), then \( T\bm{'} \) displayed in Figure~\ref{fig:3-rsc tree example}b admits a 3-rs colouring extension \big(by restricting \( f \) to \( V(T\bm{'}) \)\big). Also, any 3-rs colouring extension \( f\bm{'} \) of \( T\bm{'} \) can be extended into a 3-rs colouring of \( T \) by assigning colour~1 at \( v_1 \) and colour~2 at \( l_1, l_2 \) and \( l_3 \). Therefore, \( T \) admits a 3-rs colouring if and only if \( T\bm{'} \) admits a 3-rs colouring extension. Note that a 3-rs colouring extension of \( T\bm{'} \) must assign colour 0 at \( v_5 \) by Observation~\ref{obs:colour propagation} (applied on path \( v_2,v_3,v_4,v_5 \)), and colour 1 at \( v_6 \) (because \( v_6 \) is a 3-plus vertex). Therefore, \( T\bm{'} \) admits a 3-rs colouring extension if and only if \( T\bm{''} \) displayed in Figure~\ref{fig:example T''} admits a 3-rs colouring extension. So, \( T \) admits a 3-rs colouring if and only if \( T\bm{''} \) admits a 3-rs colouring extension. 
But, \( T\bm{''} \) does not admit a 3-rs colouring extension (\( v_7 \) must be coloured 0, and thus \( v_5,v_6,v_7 \) is a bicoloured \( P_3 \) with a higher colour on its middle vertex). Therefore, \( T \) doesn't admit a 3-rs colouring.

\begin{figure}[hbt]
\centering
\begin{subfigure}[b]{0.3\textwidth}
\centering
\begin{tikzpicture}[scale=0.5]
\node at (0,0) [dot](root){};
\draw (root) --+(-1,-1) node[dot]{}
      (root) --+( 0,-1) node[dot]{}
      (root) --+( 1,-1) node[dot] (v)[label=above:\( v \),label={[vcolour]right:1}]{};

\draw (v) --++(-1,-1) node[dot] (L11){} --++(-1,-1) node[dot] (u)[label=left:\( u \)]{} --++(-1,-1) node[dot] (L21){} --++(-1,-1) node[dot] (L31){};
\draw (u) --+(1,-1) node[dot] (L22){};
\draw (L21) --+(1,-1) node[dot] (L32) [label={[vcolour]right:0}]{};
\draw (v) --++(1,-1) node[dot] (L12){} --++(1,-1) node[dot] (u')[label=above:\( u' \)]{};
\end{tikzpicture}
\caption{\( T \)}
\end{subfigure}%
\begin{subfigure}[b]{0.3\textwidth}
\centering
\begin{tikzpicture}[scale=0.5]
\node at (0,0) [dot](v)[label=above:\( v \),label={[vcolour]right:1}]{};
\draw (v) --++(-1,-1) node[dot] (L11){} --++(-1,-1) node[dot] (u)[label=left:\( u \)]{} --++(-1,-1) node[dot] (L21){} --++(-1,-1) node[dot] (L31){};
\draw (u) --+(1,-1) node[dot] (L22){};
\draw (L21) --+(1,-1) node[dot] (L32) [label={[vcolour]right:0}]{};
\draw (v) --++(1,-1) node[dot] (L12){} --++(1,-1) node[dot] (u')[label=above:\( u' \)]{};
\end{tikzpicture}
\caption{\( T_v \)}
\label{fig:eg rooted subtree}
\end{subfigure}%
\begin{subfigure}[b]{0.37\textwidth}
\centering
\begin{tikzpicture}[scale=0.5]
\node at (0,0) [dot](v1)[label=above:\( v \),label={[vcolour]right:1}]{};
\draw (v1) --++(-1,-1) node[dot] (L11){} --++(-1,-1) node[dot] (u)[label=left:\( u \)]{} --++(-1,-1) node[dot] (L21){} --++(-1,-1) node[dot] (L31){};
\draw (u) --+(1,-1) node[dot] (L22){};
\draw (L21) --+(1,-1) node[dot] (L32) [label={[vcolour]right:0}]{};

\draw (v1) ++(1.75,0) node[dot](v2)[label=above:\( v \),label={[vcolour]right:1}]{};
\draw (v2) --++(1,-1) node[dot] (L12){} --++(1,-1) node[dot] (u')[label=above:\( u' \)]{};
\end{tikzpicture}
\caption{branches of \( T \) at \( v \)}
\label{fig:eg branches}
\end{subfigure}
\caption{(a)~a partially coloured rooted tree \( T \), (b)~the rooted subtree of \( T \) at \( v \) (denoted by \( T_v \)), (c)~branches of \( T \) at \( v \)}
\label{fig:eg rooted subtree, branch}
\end{figure}
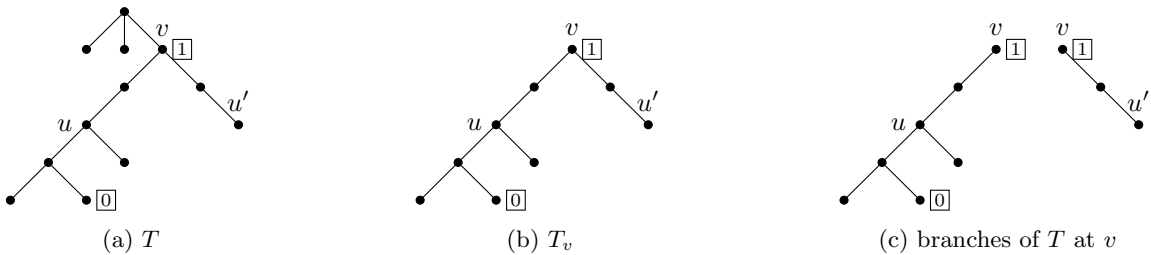

As seen in the above example, dealing with 3-rs colouring extension is helpful. For simplicity, the input tree \( T \) (not a path) is viewed as a rooted tree with a 3-plus vertex as the root. The following definitions help to present the algorithm. It is assumed that the reader is familiar with rooted tree terminology such as parent, child, descendant, and so on. Let \( T \) be a partially coloured rooted tree. If \( v \) is a vertex of \( T \) with \( \deg_T(v)\neq 2 \), the \emph{rooted subtree of \( T \) at \( v \)}, denoted by \( T_v \), is the subgraph of \( T \) induced by \( v \) and its descendants in \( T \) along with parent-child relations and colours inherited from \( T \) (see Figure~\ref{fig:eg rooted subtree}). If we `split' \( T_v \) at \( v \), each resulting piece is called a \emph{branch of \( T \) at \( v \)} (see Figure~\ref{fig:eg rooted subtree, branch}; formally, each branch of \( T \) at \( v \) is \( T_v[U_i\cup \{v\}] \) with inherited parent-child relations and colours where \( U_i \) is the vertex set of some component of \( T_v-v \)). By our definition, the root of a branch in \( T \) is always a 3-plus vertex in \( T \) (if \( \deg_T(v)=1 \), then \( T_v\cong K_1 \) and hence there is no branch at \( v \)). 
\begin{figure}[hbt]
\centering
\begin{subfigure}[c]{0.4\textwidth}
\centering
\begin{tikzpicture}
\draw (0,0) node[dot][label=above:\( l_1 \)]{}--++(10:1) node[dot](a1)[label=above:\( v_1 \)]{}--++(10:1) node[dot](a2)[label=above:\( v_2 \)]{}--++(10:1) node[dot](a3)[label=above:\( v_3 \)]{};
\draw (a1)--+(-80:1) node[dot][label=below:\( l_2 \)]{}
      (a2)--+(-80:1) node[dot][label=below:\( l_3 \)]{};
\end{tikzpicture}
\end{subfigure}%
\begin{subfigure}[c]{0.1\textwidth}
\begin{tikzpicture}
\draw[-stealth,ultra thick] (0,0)--+(0.35,0);
\end{tikzpicture}
\end{subfigure}%
\begin{subfigure}[c]{0.3\textwidth}
\begin{tikzpicture}
\draw (0,0)++(0,1.5) node[dot](a2)[label=above:\( v_2 \)][label={[vcolour]below:0}]{}--++(10:1) node[dot](a3)[label=above:\( v_3 \)][label={[vcolour]below:1}]{};
\end{tikzpicture}
\end{subfigure}%
\caption{A branch replacement operation. Applying this operation on the left branch at \( v_3 \) transforms \( T \) into \( T\bm{'} \) (\( T \) and \( T\bm{'} \) are displayed in Figure~\ref{fig:3-rsc tree example})}
\label{fig:sample branch replacement}
\end{figure}
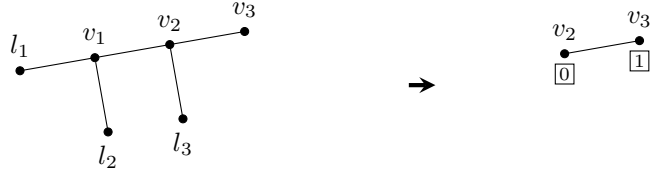


As far as 3-rs colouring extension is concerned, some branches can be replaced by other branches. For instance, the transformation from \( T \) to \( T\bm{'} \) in the first example (see Figure~\ref{fig:3-rsc tree example}) can be viewed as a local replacement operation, namely replacement of the left branch of \( T \) at \( v_3 \) by another branch (the replacement operation is displayed in Figure~\ref{fig:sample branch replacement}). 
A replacement operation that produces a partially coloured tree \( T\bm{'} \) from a partially coloured tree \( T \) is said to \emph{preserve 3-rs colouring extension status} if both \( T \) and \( T\bm{'} \) admit a 3-rs colouring extension, or neither does. The replacement operations considered in this paper are either (i)~replacement of branches by branches, or (ii)~replacement of rooted subtrees by rooted subtrees. 

We define an equivalence relation on the set of all branches (of partially coloured rooted trees) as follows: \( B_1\sim B_2 \) if the operation of replacing branch \( B_1 \) by branch \( B_2 \) in every partially coloured rooted tree preserves 3-rs colouring extension status. 
It is left as an exercise to the reader to verify that this is indeed an equivalence relation (see supplementary material). 

\begin{figure}[h] 
\centering
\begin{subfigure}[t]{0.3\textwidth}
\centering
\begin{tikzpicture}
\draw (0,0) node(y)[dot][label={[vcolour]below:0}]{} --++(10:0.75) node (x)[dot][label={[vcolour]below:1}]{} --++(10:0.75) node(v)[label=above:\( v \)][dot]{}; 
\end{tikzpicture}
\caption*{Class~A}
\end{subfigure}%
\begin{subfigure}[t]{0.3\textwidth}
\centering
\begin{tikzpicture}
\draw (0,0) node (x)[dot]{} --++(10:0.75) node(v)[label=above:\( v \)][dot][label={[vcolour]below:0}]{};
\end{tikzpicture}
\caption*{Class~B}
\end{subfigure}%
\begin{subfigure}[t]{0.3\textwidth}
\centering
\begin{tikzpicture}
\draw (0,0) node (x)[dot][label={[vcolour]below:0}]{} --++(10:0.75) node(v)[label=above:\( v \)][dot][label={[vcolour]below:1}]{};
\end{tikzpicture}
\caption*{Class~C}
\end{subfigure}

\begin{subfigure}[t]{0.3\textwidth}
\centering
\begin{tikzpicture}
\draw (0,0) node (x)[dot]{} --++(10:0.75) node(v)[label=above:\( v \)][dot][label={[vcolour]below:1}]{};
\end{tikzpicture}
\caption*{Class~D}
\end{subfigure}%
\begin{subfigure}[t]{0.3\textwidth}
\centering
\begin{tikzpicture}
\draw (0,0) node(y)[dot][label={[vcolour]below:1}]{} --++(10:0.75) node (x)[dot]{} --++(10:0.75) node(v)[label=above:\( v \)][dot]{};
\end{tikzpicture}
\caption*{Class~E}
\end{subfigure}%
\begin{subfigure}[t]{0.3\textwidth}
\centering
\begin{tikzpicture}
\draw (0,0) node (x)[dot]{} --++(10:0.75) node(v)[label=above:\( v \)][dot]{};
\end{tikzpicture}
\caption*{Class~F}
\end{subfigure}%
\caption{The representatives of Classes A to F}
\label{fig:branch classes}
\end{figure}
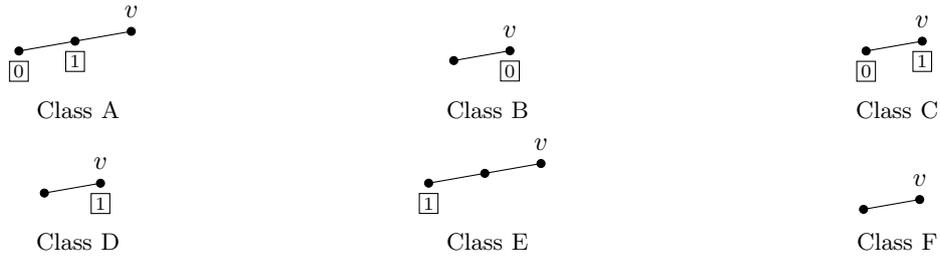

Every uncoloured branch (i.e., branch without any colour) belongs to one of six equivalence classes under this equivalence relation (proved later in Theorem~\ref{thm:uncoloured branch/rooted subtree}). We call these equivalence classes as Class~A, Class~B, \( \dots \), Class~F; and a representative of each class is shown in Figure~\ref{fig:branch classes} (representatives may contain colours). They are referred to as the representatives of respective classes for the rest of the paper. It is left as an exercise to the reader to verify that branches in Figure~\ref{fig:branch classes} belong to different equivalence classes (see supplementary material).


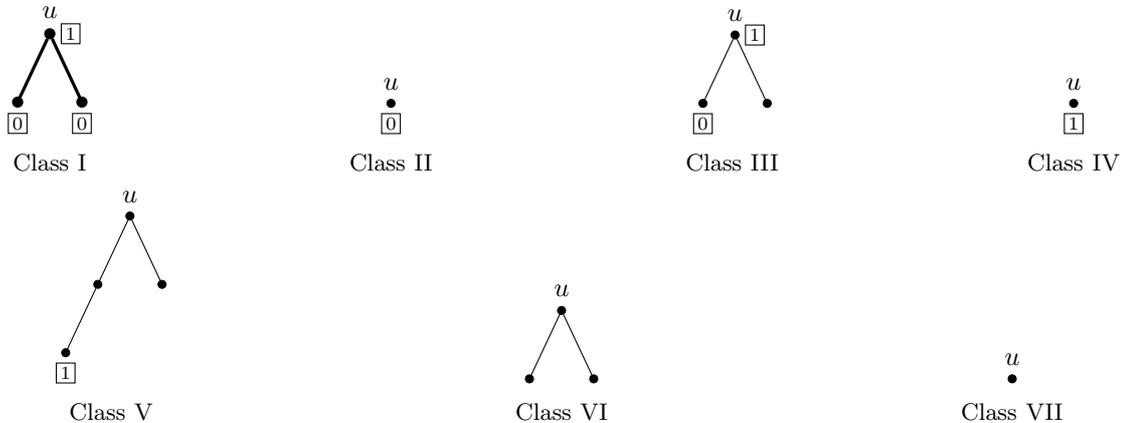
\begin{figure}[hbt] 
  \centering
  \begin{subfigure}[t]{0.25\textwidth}
  \centering
  \begin{tikzpicture}
  \draw[very thick] (0,0) node (u)[label=above:\( u \)][dot][label={[thin,vcolour]right:1}]{};
  \draw[very thick] (u) --+(-65:1)  node(ls)[dot][label={[thin,vcolour]below:0}]{}
        (u) --+(-115:1) node(l1)[dot][label={[thin,vcolour]below:0}]{};
  \end{tikzpicture}
  \caption*{Class~I}
  \end{subfigure}%
\begin{subfigure}[t]{0.25\textwidth}
  \centering
  \begin{tikzpicture}
  \draw (0,0) node (u)[label=above:\( u \)][dot][label={[vcolour]below:0}]{};
  \end{tikzpicture}
  \caption*{Class~II}
  \end{subfigure}%
  \begin{subfigure}[t]{0.25\textwidth}
  \centering
  \begin{tikzpicture}
  \node (u)[label=above:\( u \)][dot][label={[vcolour]right:1}]{};
  \draw (u) --+(-65:1)  node[dot]{}
        (u) --+(-115:1) node[dot][label={[vcolour]below:0}]{};
  \end{tikzpicture}
  \caption*{Class~III}
  \end{subfigure}%
\begin{subfigure}[t]{0.25\textwidth}
  \centering
  \begin{tikzpicture}
  \draw (0,0) node (u)[label=above:\( u \)][dot][label={[vcolour]below:1}]{};
  \end{tikzpicture}
  \caption*{Class~IV}
  \end{subfigure}%

  \begin{subfigure}[t]{0.33\textwidth}
  \centering
  \begin{tikzpicture}
  \draw (0,0) node (u)[label=above:\( u \)][dot]{} --+(-115:1) node (x)[dot]{};
  \draw (u) --+(-65:1) node (l1)[dot]{};
  \draw (x) --+(-115:1) node(y)[dot][label={[vcolour]below:1}]{};
  \end{tikzpicture}
  \caption*{Class~V}
  \end{subfigure}%
\begin{subfigure}[t]{0.33\textwidth}
  \centering
  \begin{tikzpicture}
  \draw (0,0) node (u)[label=above:\( u \)][dot]{} --+(-115:1) node (x)[dot]{};
  \draw (u) --+(-65:1) node (l1)[dot]{};
  \end{tikzpicture}
  \caption*{Class~VI}
  \end{subfigure}%
\begin{subfigure}[t]{0.33\textwidth}
  \centering
  \begin{tikzpicture}
  \draw (0,0) node (u)[label=above:\( u \)][dot]{};
  \end{tikzpicture}
  \caption*{Class~VII}
  \end{subfigure}%
  \caption{The representatives of Classes I to VII}
  \label{fig:rooted subtree classes}
  \end{figure}

Similarly, we define an equivalence relation on the set of all rooted subtrees (of partially coloured rooted trees) as follows: \( R_1\sim R_2 \) if replacing rooted subtree \( R_1 \) by rooted subtree \( R_2 \) in every partially coloured rooted tree preserves 3-rs colouring extension status. An uncoloured rooted subtree belongs to one of seven equivalence classes under this equivalence relation (proved later in Theorem~\ref{thm:uncoloured branch/rooted subtree}). They are named Class~I, Class~II, \( \dots \), Class~VII; and a representative of each class is shown in Figure~\ref{fig:rooted subtree classes}. They are referred to as the representatives of respective classes for the rest of the paper. It is left as an exercise to the reader to verify that rooted subtrees in Figure~\ref{fig:rooted subtree classes} belong to different equivalence classes (see supplementary material).

Observe that this division of branches (resp.\ rooted subtrees) into equivalence classes reflect the conditions branches (resp.\ rooted subtrees) impose on trees containing them to admit a 3-rs colouring extension. Let \( T \) be a partially coloured rooted tree with a 3-plus vertex as its root. If \( T \) contains a branch \( B^* \) isomorphic to the representative of Class~A, then \( T \) does not admit a 3-rs colouring extension because (i)~the root of branch \( B^* \) must be coloured 2 by Observation~\ref{obs:colour propagation}, and (ii)~the root of a branch is always a 3-plus vertex and hence cannot be coloured 2. Thus, by definition of the equivalence relation, \( T \) does not admit a 3-rs colouring extension if \( T \) contains a Class~A branch. 
Similarly, if \( T \) contains a rooted subtree \( R^* \) isomorphic to the representative of Class~I, then \( T \) does not admit a 3-rs colouring extension because \( R^* \) itself is a bicoloured \( P_3 \) with a higher colour on its middle vertex. Hence, \( T \) does not admit a 3-rs colouring extension if \( T \) contains a Class~I rooted subtree. Since \( T \) is arbitrary, we have the following lemma. 
\begin{lemma}
A branch \( B \) is in Class~A if and only if no partially coloured rooted tree (with a 3-plus vertex as root) containing \( B \) admits a 3-rs colouring extension. Similarly, a rooted subtree \( R \) is in Class~I if and only if no partially coloured rooted tree (with a 3-plus vertex as root) containing \( R \) admits a 3-rs colouring extension.\qed
\label{lem:classes 1 and I}
\end{lemma}


The first branch in Figure~\ref{fig:eg branches} is composed of the rooted subtree \( T_u \) and a \( u,v \)-path. 
In general, a branch \( B \) at a vertex \( v \) is composed of a rooted subtree \( T_u \) and a \( u,v \)-path (where \( u \) is the first descendant of \( v \) in \( B \) which is not of degree two). 
The length of the \( u,v \)-path is called the \emph{up-distance} of the branch \( B \). For instance, the first branch in Figure~\ref{fig:eg branches} has up-distance two because the length of the \( u,v \)-path is two. 

\begin{figure}[hbt]
\begin{minipage}[c]{0.43\textwidth}
\setlength\tabcolsep{2 pt}
\renewcommand{\arraystretch}{1.25}
\begin{table}[H]
\centering
\caption{Equivalence class of branch \( B \) in terms of the up-distance of \( B \) and the equivalence class of \( T_u \).}
\begin{tabular}{|c|cccccc|}
\hline
\diagbox[rightsep=3pt,height=0.75cm,width=1.5cm]{\scriptsize up-dist.}{\( T_u \)}
              &II   &III   &IV   &V    &VI  &VII\\
\hline
1             &C  &A  &B  &C  &E  &F  \\
2             &D  &B  &E  &F  &F  &''  \\
3             &B  &C  &D  &E  &''  &''  \\
4             &E  &D  &F  &F  &''  &''  \\
5             &D  &B  &E  &''  &''  &''  \\
6             &F  &E  &F  &''  &''  &''  \\
7             &E  &D  &''  &''  &''  &''  \\
8             &F  &F  &''  &''  &''  &''  \\
9             &''  &E  &''  &''  &''  &''  \\
\( \geq 10 \) &''  &F  &''  &''  &''  &'' \\
\hline
\end{tabular}
\label{tbl:class of branch}
\end{table}
\vspace{0.25cm}
\setlength\tabcolsep{6 pt}
\end{minipage}
\hfill
\begin{minipage}[c]{0.5\textwidth}
\begin{table}[H]
\renewcommand\cellalign{lc}  
\centering
\caption{Equivalence class of rooted subtree \( T_v \) in terms of the number of branches at \( v \) belonging to each class.}
\begin{tabular}{|c|l|}
\hline
\( b>0 \) &\makecell{If \( c=d=0 \),\\
                     \quad\( T_v\in \) Class~II.\\
                     If \( c+d>0 \),\\
                     \quad\( T_v\in \) Class~I.}\\
\hline
\( b=0,\, c+d>0 \) & \makecell{If \( c+e=0 \),\\
                             \quad\( T_v\in \) Class~IV.\\
                             If \( c+e=1 \),\\
                             \quad\( T_v\in \) Class~III.\\
                             If \( c+e\geq 2 \),\\
                             \quad\( T_v\in \) Class~I.}\\
\hline
\( b=c=d=0 \) & \makecell{If \( e=f=0 \),\\
                          \quad\( T_v\in \) Class~VII.\\
                          If \( e=0 \) and \( f>0 \),\\
                          \quad\( T_v\in \) Class~VI.\\
                          If \( e=1 \),\\
                          \quad\( T_v\in \) Class~V.\\
                          If \( e\geq 2 \),\\
                          \quad\( T_v\in \) Class~II.}\\
\hline
\end{tabular}
\label{tbl:class of rooted subtree}
\end{table}
\end{minipage}
\end{figure}

\needspace{10\baselineskip}
\noindent\emph{Computation of equivalence classes}\\
Let \( B \) be a branch composed of a rooted subtree \( T_u \) and a \( u,v \)-path. If \( T_u \) is in Class~I, then \( B \) is in Class~A due to Lemma~\ref{lem:classes 1 and I} (note that a tree containing \( B \), in turn, contains \( T_u \)). Otherwise, we can determine the equivalence class of \( B \) from Table~\ref{tbl:class of branch} based on the equivalence class of \( T_u \) and the up-distance of \( B \). 
Similarly, a rooted subtree \( T_v \) is composed of branches of \( T \) at \( v \). If at least one of those branches is in Class~A, then \( T_v \) is in Class~I by Lemma~\ref{lem:classes 1 and I}. Otherwise, we can determine the equivalence class of \( T_v \) from Table~\ref{tbl:class of rooted subtree} where \( b,c,d,e,f \) denote the number of Class~B branches, Class~C branches, \( \dots \) , Class~F branches of \( T \) at \( v \) respectively (\( b,c,d,e,f\geq 0 \)). 
Subsections~\ref{sec:class of branch} and \ref{sec:class of rooted subtree} provide an overview of proofs of Tables~\ref{tbl:class of branch} and \ref{tbl:class of rooted subtree}, respectively. Complete proofs are available in supplementary material. 

One can read the look-up tables as follows. 
If a branch \( B \) is composed of a Class~III rooted subtree and a path, and up-distance of \( B \) is 5, then \( B \) is in Class~B (see Table~\ref{tbl:class of branch}; up-distance=5 is the sixth row and \( T_u\in \) Class~III is the third column of the table). If a rooted subtree \( T_v \) is made of one Class~C branch and three Class~D branches (i.e., \( b=0 \), \( c=1 \), \( d=3 \), and \( e=f=0 \)), then \( T_v \) is in Class~III (see the case \( c+e=1 \) of the second row of Table~\ref{tbl:class of rooted subtree}).\\

\needspace{5\baselineskip}
\noindent \textbf{Algorithm for testing 3-rs colourability of trees}\\
Assume that the input to the algorithm is a rooted tree \( T \) with a 3-plus vertex \( a \) as its root (recall that such a representation can be obtained in linear time by BFS). 
The algorithm determines the equivalence class of the rooted subtree \( T_a \) recursively. To be explicit, the algorithm visits vertices of \( T \) in bottom-up order (via post-order traversal) and determines the equivalence classes of all branches and rooted subtrees of \( T \) until (i)~it comes across a Class~A branch or a Class~I rooted subtree, or (ii)~all vertices of \( T \) are visited. In the former case, \( T \) is not 3-rs colourable. In the latter case, \( T=T_a \) is in one of the classes Class~II, Class~III, \( \dots \), Class~VII, and thus, \( T \) is 3-rs colourable.\\

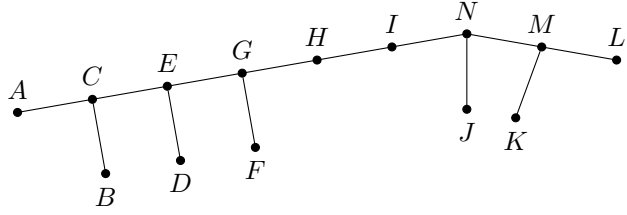
\begin{figure}[hbt]
\centering
\begin{tikzpicture}
\draw (0,0) node[dot][label=above:\( A \)]{}--++(10:1) node[dot](a1)[label=above:\( C \)]{}--++(10:1) node[dot](a2)[label=above:\( E \)]{}--++(10:1) node[dot](a3)[label=above:\( G \)]{}--++(10:1) node[dot](a4)[label=above:\( H \)]{}--++(10:1) node[dot](a5)[label=above:\( I \)]{}--++(10:1) node[dot](a6)[label=above:\( N \)]{}--++(-10:1) node[dot](a7)[label=above:\( M \)]{}--++(-10:1) node[dot][label=above:\( L \)]{};
\draw (a1)--+(-80:1) node[dot][label=below:\( B \)]{}
      (a2)--+(-80:1) node[dot][label=below:\( D \)]{}
      (a3)--+(-80:1) node[dot][label=below:\( F \)]{}
      (a6)--+(-90:1) node[dot][label=below:\( J \)]{}
      (a7)--+(-110:1) node[dot][label=below:\( K \)]{};
\end{tikzpicture}
\caption{A rooted tree \( T \) with root \( N \). A post-order traversal of the tree visits vertices in alphabetical order}
\label{fig:T redrawn}
\end{figure}

To illustrate the algorithm execution, let us revisit the first example. The tree \( T \) of Figure~\ref{fig:3-rsc tree example} is redrawn in Figure~\ref{fig:T redrawn} (vertex labels are changed for convenience).  The equivalence classes of branches and rooted subtrees of \( T \) are computed in the following order using Table~\ref{tbl:class of branch} as the look-up table for branches, and Table~\ref{tbl:class of rooted subtree}, for rooted subtrees.\\
\noindent\( T_A\in \) Class~VII (\( \because \) \( b=c=d=e=f=0 \); see Table~\ref{tbl:class of rooted subtree}),\\
Left branch of \( T \) at \( C \) \( \in \) Class~F (\( \because \) \( T_A\in  \) Class~VII, up-dist.\ \( = 1 \); see Table~\ref{tbl:class of branch}),\\
\( T_B\in \) Class~VII (\( \because \) \( b=c=d=e=f=0 \)),\\
Right branch of \( T \) at \( C \) \( \in \) Class~F (\( \because \) \( T_B\in \) Class~VII, up-dist.\ \( = 1 \)),\\
\( T_C\in \) Class~VI (\( \because \) \( b=c=d=e=0 \) and \( f=2 \)),\\
Left branch of \( T \) at \( E \) \( \in \) Class~E (\( \because \) \( T_C\in \) Class~VI, up-dist.\ \( = 1 \)),\\
\( T_D\in \) Class~VII (\( \because \) \( b=c=d=e=f=0 \)),\\
Right branch of \( T \) at \( E \) \( \in \) Class~F (\( \because \) \( T_D \in \) Class~VII, up-dist.\ \( = 1 \)),\\
\( T_E\in \) Class~V (\( \because \) \( b=c=d=0 \) and \( e=f=1 \)),\\
Left branch of \( T \) at \( G \) \( \in \) Class~C (\( \because \) \( T_E \in \) Class~V, up-dist.\ \( = 1 \)),\\
\( T_F\in \) Class~VII (\( \because \) \( b=c=d=e=f=0 \)),\\
Right branch of \( T \) at \( G \) \( \in \) Class~F (\( \because \) \( T_F \in \) Class~VII, up-dist.\ \( = 1 \)),\\
\( T_G\in \) Class~III (\( \because \) \( b=d=e=0 \) and \( c=f=1 \)),\\
(Next, we consider \( N \), the immediate ancestor of \( G \) which is a 3-plus vertex. Note that vertices \( H \) and \( I \) are of degree two.)\\
Left branch of \( T \) at \( N \) \( \in \) Class~C (\( \because \) \( T_G\in \) Class~III, up-dist.\ \(= 3 \)),\\
\( T_J\in \) Class~VII (\( \because \) \( b=c=d=e=f=0 \)),\\
Middle branch of \( T \) at \( N\in\, \)Class~F\,(\( \because \) \( T_J \in \) Class~VII, up-dist.\ \( = 1 \)),\\
\( T_K\in \) Class~VII (\( \because \) \( b=c=d=e=f=0 \)),\\
Left branch of \( T \) at \( M \) \( \in \) Class~F (\( \because \) \( T_K \in \) Class~VII, up-dist.\ \( = 1 \)),\\
\( T_L\in \) Class~VII (\( \because \) \( b=c=d=e=f=0 \)),\\
Right branch of \( T \) at \( M\in\, \)Class~F\,(\( \because \) \( T_L\in \) Class~VII, up-dist.\ \( = 1 \)),\\
\( T_M\in \) Class~VI (\( \because \) \( b=c=d=e=0 \) and \( f=2 \)),\\
Right branch of \( T \) at \( N \) \( \in \) Class~E (\( \because \) \( T_M \in \) Class~VI, up-dist.\ \( = 1 \)),\\
\( T_N\in \) Class~I (\( \because \) \( b=d=0 \) and \( c=e=f=1 \)).\\
That is, \( T_N \) does not admit a 3-rs colouring extension (by Lemma~\ref{lem:classes 1 and I}). Therefore, \( T=T_N \) is not 3-rs colourable.

As promised earlier, we next prove that every uncoloured branch belongs to one of the classes Class~A, Class~B, \( \dots \), Class~F; and every uncoloured rooted subtree belongs to one of the classes Class~I, Class~II, \( \dots \), Class~VII. Lemmas~\ref{lem:branch replacement corollary} and \ref{lem:subtree replacement corollary} below are corollaries of Table~\ref{tbl:class of branch} and Table~\ref{tbl:class of rooted subtree} respectively. 

\begin{lemma}
Let \( B \) be an uncoloured branch comprised of a rooted subtree \( T_u \) and a path. If \( T_u \) belongs to one of the classes Class~I, Class~II, \( \dots \), Class~VII, then \( B \) belongs to one of the classes Class~A, Class~B, \( \dots \), Class~F.\qed
\label{lem:branch replacement corollary}
\end{lemma}
\begin{lemma}
Let \( T_u \) be a rooted subtree comprised of branches each of which belongs to one of the classes Class~A, Class~B, \( \dots \), Class~F. Then, \( T_u \) belongs to one of the classes Class~I, Class~II, \( \dots \), Class~VII.\qed
\label{lem:subtree replacement corollary}
\end{lemma}

\begin{theorem}
Every uncoloured branch belongs to one of the classes Class~A, Class~B, \( \dots \), Class~F; and every uncoloured rooted subtree belongs to one of the classes Class~I, Class~II, \( \dots \), Class~VII.
\label{thm:uncoloured branch/rooted subtree}
\end{theorem}
\begin{proof}
To produce a contradiction, assume that the theorem is not true. Let \( X  \) be a counterexample (a branch or a rooted subtree) with the least number of vertices. 
If \( X \) is a branch comprised of a rooted subtree \( T_u \) and a path, then \( T_u \) must belong to one of the classes Class~I, Class~II, \( \dots \), Class~VII (if not, \( T_u \) is a counterexample with fewer vertices than \( X \)). So, \( X \) must belong to one of the classes Class~A, Class~B, \( \dots \),  Class~F by Lemma~\ref{lem:branch replacement corollary}; a contradiction. If \( X \) is a rooted subtree, then \( X \) is composed of branches that belong to Classes A to F (if one of the branches is not in Classes~A to F, then that branch is a counterexample with fewer vertices than \( X \)). So, \( X \) must belong to one of the classes Class~I, Class~II, \( \dots \), Class~VII by Lemma~\ref{lem:subtree replacement corollary}; a contradiction.
\end{proof}

The pseudocode and the run-time analysis of the algorithm are presented in Subsection~\ref{sec:pseudocode}.

\subsection{Proof of Table~\ref{tbl:class of branch} (Overview)}\label{sec:class of branch} 
\begin{figure}[hbt]
\centering
\begin{subfigure}[b]{0.25\textwidth}
\centering
\begin{tikzpicture}
\draw (0,0) node (u)[dot][label={[vcolour]below:0}][label=\( u \)]{};
\end{tikzpicture}
\caption{\( T_u \)}
\label{fig:Tu in branch replacement 6}
\end{subfigure}%
\begin{subfigure}[b]{0.25\textwidth}
\centering
\begin{tikzpicture}
\draw (0,0) node (u)[dot][label={[vcolour]below:0}][label=\( u \)]{} --++(10:0.75) node(v)[label=above:\( v \)][dot]{};
\end{tikzpicture}
\caption{\( B \)}
\label{fig:B in branch replacement 6}
\end{subfigure}%
\caption{Rooted subtree \( T_u \) and branch \( B \) for the first entry of Table~\ref{tbl:class of branch}}
\end{figure}
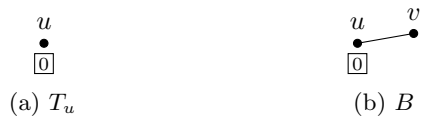
The first entry of the table says that a branch \( B \) composed of a Class~II rooted subtree \( T_u \) and a path of length one (i.e. up-distance\( (B)=1 \)) belongs to Class~C. Since \( T_u\in \) Class~II, we may assume that \( T_u \) is the representative of Class~II by definition of the equivalence relation; i.e., \( T_u \) is the graph in Figure~\ref{fig:Tu in branch replacement 6} and \( B \) is the graph in Figure~\ref{fig:B in branch replacement 6}. Thus, to prove the entry, it suffices to show that the branch replacement operation displayed in Figure~\ref{fig:branch replacement 6} preserves 3-rs colouring extension status (i.e., \( T\bm{'} \) preserves 3-rs colouring extension status of \( T \)). Since \( v \) is a 3-plus vertex, this branch replacement operation indeed preserves 3-rs colouring extension status. This proves the first entry of Table~\ref{tbl:class of branch}. 


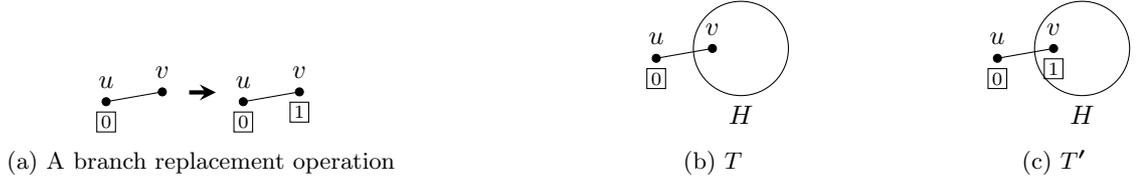
\begin{figure}[hbt]
\centering
\begin{subfigure}[b]{0.5\textwidth}
\centering
\begin{tikzpicture}
\draw (0,0) node (u)[dot][label={[vcolour]below:0}][label=\( u \)]{} --++(10:0.75) node(v)[label=above:\( v \)][dot]{};
\end{tikzpicture}
\tikz \path [-stealth,draw=black,ultra thick] (0,0)++(0,0.5)--+(0.35,0);
\begin{tikzpicture}
\draw (0,0) node (u)[dot][label={[vcolour]below:0}][label=\( u \)]{} --++(10:0.75) node(v)[label=above:\( v \)][dot][label={[vcolour]below:1}]{};
\end{tikzpicture}
\caption{A branch replacement operation}
\end{subfigure}%
\begin{subfigure}[b]{0.25\textwidth}
\centering
\begin{tikzpicture}[scale=0.75]
\draw (0,0) node[dot](u)[label=above:\( u \), label={[vcolour]below:0}]{} --++(10:1) node[dot](v)[label=above:\( v \)]{} ++(0.5,0) node[subgraph](H)[label=below:\( H \)]{};
\end{tikzpicture}
\caption{\( T \)}
\end{subfigure}%
\begin{subfigure}[b]{0.25\textwidth}
\centering
\begin{tikzpicture}[scale=0.75]
\draw (0,0) node[dot](u)[label=above:\( u \)][label={[vcolour]below:0}]{} --++(10:1) node[dot](v)[label=above:\( v \)][label={[vcolour]below:1}]{} ++(0.5,0) node[subgraph](H)[label=below:\( H \)]{};
\end{tikzpicture}
\caption{\( T\bm{'} \)}
\end{subfigure}%
\caption{A branch replacement operation, and partially coloured rooted trees before and after replacement}
\label{fig:branch replacement 6}
\end{figure}

Similarly, to prove entries in Table~\ref{tbl:class of branch}, it suffices to show that branch replacement operations displayed in Figure~\ref{fig:branch replacement 1 to 8} and Figure~\ref{fig:branch replacement 9 to 15} preserve 3-rs colouring extension status; this is easy to prove with the help of Theorem~\ref{thm:3-rsc extension paths} (see supplementary material for details and proof for each operation). We provide a sample lemma below.

\begin{lemma}
Branch replacement operation~(13) preserves 3-rs colouring extension status (see Figure~\ref{fig:branch replacement 13}).
\label{lem:branch replacement 13}
\end{lemma}
\begin{proof}
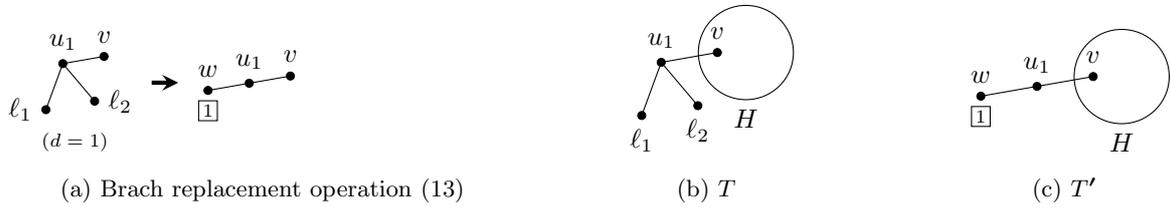
\begin{figure}[h]
\centering
\begin{subfigure}[b]{0.4\textwidth}
\mbox{
\begin{tikzpicture}
\draw (0,0) node[dot](u1)[label=above:\( u_1 \)]{} --++(10:0.55) node[dot](v)[label=above:\( v \)]{};
\draw (u1) --+(-110:0.65) node[dot](l1)[label=left:\( \ell_1 \)]{}
      (u1) --+(-50:0.65) node[dot](l2)[label=right:\( \ell_2 \)]{};
\node [fit={(v) (l1)}][label={[font=\scriptsize]below:(\( d=1 \))}]{};
\end{tikzpicture}
\tikz \path [-stealth,draw=black,ultra thick] (0,0)++(0,1.0)--+(0.35,0);
\begin{tikzpicture}
\draw (0,0) node[dot](w)[label=above:\( w \), label={[vcolour]below:1}]{} --++(10:0.55) node[dot](u1)[label=above:\( u_1 \)]{} --++(10:0.55) node[dot](v)[label=above:\( v \)]{};
\path (v)+(0,-1) node(dummy){};
\end{tikzpicture}
}
\caption{Brach replacement operation (13)}
\end{subfigure}%
\begin{subfigure}[b]{0.25\textwidth}
\centering
\begin{tikzpicture}[scale=0.75]
\draw (0,0) node[dot](u1)[label=above:\( u_1 \)]{} --++(10:1) node[dot](v)[label=above:\( v \)]{} ++(0.5,0) node[subgraph](H)[label=below:\( H \)]{};
\draw (u1) --+(-110:1) node[dot](l1)[label=below:\( \ell_1 \)]{}
      (u1) --+(-50:1) node[dot](l2)[label=below:\( \ell_2 \)]{};
\end{tikzpicture}
\caption{\( T \)}
\end{subfigure}%
\hspace{3pt}
\begin{subfigure}[b]{0.25\textwidth}
\centering
\begin{tikzpicture}[scale=0.75]
\draw (0,0) node[dot](w)[label=above:\( w \), label={[vcolour]below:1}]{} --++(10:1) node[dot](u)[label=above:\( u_1 \)]{} --++(10:1) node[dot](v)[label=above:\( v \)]{} ++(0.5,0) node[subgraph](H)[label=below:\( H \)]{};
\end{tikzpicture}
\caption{\( T\bm{'} \)}
\end{subfigure}
\caption{Branch replacement operation~(13), and partially coloured rooted trees before and after replacement}
\label{fig:branch replacement 13}
\end{figure}

Suppose that \( T \) admits a 3-rs colouring extension \( f \). Let \( H=T-\{u_1,\ell_1,\ell_2\} \). If \( f(v)=0 \), then \( f \) restricted to \( V(H) \) can be extended into a 3-rs colouring extension of \( T\bm{'} \) by assigning colour~2 at \( u_1 \) and colour~1 at \( w \). Otherwise, \( f(v)=1 \) and hence \( f(u_1)=0 \) (by Property~P2). So, the restriction of \( f \) to \( V(H) \) can be extended into a 3-rs colouring extension of \( T\bm{'} \) by assigning colour~0 at \( u_1 \) and colour~1 at \( w \).

Conversely, suppose that \( T\bm{'} \) admits a 3-rs colouring extension \( f\bm{'} \). If \( f\bm{'}(v)=0 \), then \( f\bm{'} \) restricted to \( V(H) \) can be extended into a 3-rs colouring extension of \( T \) by assigning colour~1 at \( u_1 \) and colour~2 at \( \ell_1,\ell_2 \). Otherwise, \( f\bm{'}(v)=1 \). Since \( w,u_1,v \) is a bicoloured \( P_3 \) with endpoints coloured~1, its middle vertex must be coloured~0. That is, \( f\bm{'}(u_1)=0 \). Hence, \( f\bm{'} \) restricted to \( V(H) \) can be extended into a 3-rs colouring extension of \( T \) by assigning colour~0 at \( u_1 \) and colour~2 at \( \ell_1,\ell_2 \).
\end{proof}

\begin{figure}[hbtp] 
\centering
\begin{tikzpicture}
\matrix[matrix of nodes,row sep=0.05cm,nodes={anchor=center}]{
\tikz \node {(1)};&[-0.4cm]
\begin{tikzpicture}
\draw (0,0) node[dot](u1)[label=above:\( u_1 \)][label={[vcolour]below:1}]{} --++(10:0.55) node[dot](v)[label=above:\( v \)]{};
\node [fit={(u1) (v)}](left iv)[label={[font=\scriptsize,label distance=4pt]below:(\( d=1 \))}]{};
\end{tikzpicture}
           & \tikz \path [-stealth,draw=black,ultra thick] (0,0)--(0.35,0);
                & \begin{tikzpicture}
                  \draw (0,0) node[dot][label=above:\( u_1 \)]{} --++(10:0.55) node[dot](v)[label=above:\( v \), label={[vcolour]below:0}]{};
                  \end{tikzpicture} & (Class~B)\\
\tikz \node {(2)};&
\begin{tikzpicture}
\draw (0,0) node[dot](u1)[label=above:\( u_1 \)][label={[vcolour]below:1}]{} --++(10:0.55) node[dot](u2)[label=above:\( u_2 \)]{} --++(10:0.55) node[dot](u3)[label=above:\( u_3 \)]{} --++(10:0.55) node[dot](v)[label=above:\( v \)]{};
\path (u2) --node[sloped]{\tiny \( \dots \)} (u3);
\node [fit={(u1) (v)}](left iv)[label={[font=\scriptsize]below:(\( d=3 \))}]{};
\end{tikzpicture}
           & \tikz \path [-stealth,draw=black,ultra thick] (0,0)--(0.35,0);
                & \begin{tikzpicture}
                  \draw (0,0) node[dot][label=above:\( u_3 \)]{} --++(10:0.55) node[dot](v)[label=above:\( v \), label={[vcolour]below:1}]{}; 
                  \end{tikzpicture} & (Class~D)\\
\tikz \node {(3)};&
\begin{tikzpicture}
\draw (0,0) node[dot](u1)[label=above:\( u_1 \)][label={[vcolour]below:1}]{} --++(10:0.55) node[dot](u2)[label=above:\( u_2 \)]{} --++(10:0.55) node[dot](u3)[label=above:\( u_3 \)]{} --++(10:0.55) node[dot](u4)[label=above:\( u_4 \)]{} --++(10:0.55) node[dot](v)[label=above:\( v \)]{};
\path (u2) --node[sloped]{\tiny \( \dots \)} (u4);
\node [fit={(u1) (v)}](left iv)[label={[font=\scriptsize]below:(\( d=4 \))}]{};
\end{tikzpicture}
           & \tikz \path [-stealth,draw=black,ultra thick] (0,0)--(0.35,0);
                & \begin{tikzpicture}
                  \draw (0,0) node[dot](u4)[label=above:\( u_4 \)]{} --++(10:0.55) node[dot](v)[label=above:\( v \)]{}; 
                  \end{tikzpicture} & (Class~F)\\

\tikz \node {(4)};&
\begin{tikzpicture}
\draw (0,0) node[dot](u1)[label=above:\( u_1 \)][label={[vcolour]below:1}]{} --++(10:0.55) node[dot](u2)[label=above:\( u_2 \)]{} --++(10:0.55) node[dot](u3)[label=above:\( u_3 \)]{} --++(10:0.55) node[dot](u4)[label=above:\( u_4 \)]{} --++(10:0.55) node[dot](u5)[label=above:\( u_5 \)]{} --++(10:0.55) node[dot](v)[label=above:\( v \)]{};
\path (u2) --node[sloped]{\tiny \( \dots \)} (u5);
\node [fit={(u1) (v)}](left iv)[label={[font=\scriptsize]below:(\( d=5 \))}]{};
\end{tikzpicture}
           & \tikz \path [-stealth,draw=black,ultra thick] (0,0)--(0.35,0);
                & \begin{tikzpicture}
                  \draw (0,0) node[dot](u4)[label=above:\( u_4 \), label={[vcolour]below:1}]{} --++(10:0.55) node[dot](u5)[label=above:\( u_5 \)]{} --++(10:0.55) node[dot](v)[label=above:\( v \)]{}; 
                  \end{tikzpicture} & (Class~E)\\
\tikz \node {(5)};&
\begin{tikzpicture}
\draw (0,0) node[dot](u1)[label=above:\( u_1 \), label={[vcolour]below:1}]{} --++(10:0.55) node[dot](u2)[label=above:\( u_2 \)]{} --++(10:0.275) ++(10:0.55) --++(10:0.275) node[dot](ud)[label=above:\( u_d \)]{} --++(10:0.55) node[dot](v)[label=above:\( v \)]{};
\path (u2) --node[sloped]{\tiny \( \dots \)} (ud);
\node [fit={(u1) (v)}](left iv)[label={[font=\scriptsize]below:(\( d\geq 6 \))}]{};
\end{tikzpicture}
           &[-0.2cm] \tikz \path [-stealth,draw=black,ultra thick] (0,0)--(0.35,0);
                &[-0.2cm] \begin{tikzpicture}
                  \draw (0,0) node[dot](ud)[label=above:\( u_d \)]{} --++(10:0.55) node[dot](v)[label=above:\( v \)]{}; 
                  \end{tikzpicture} & (Class~F)\\
\tikz \node {(6)};&
\begin{tikzpicture}
\draw (0,0) node[dot](u1)[label=above:\( u_1 \)][label={[vcolour]below:0}]{} --++(10:0.55) node[dot](v)[label=above:\( v \)]{};
\node [fit={(u1) (v)}](left iv)[label={[font=\scriptsize,label distance=4pt]below:(\( d=1 \))}]{};
\end{tikzpicture}
   &\tikz \path [-stealth,draw=black,ultra thick] (0,0)--(0.35,0);
       &\begin{tikzpicture}
       \draw (0,0) node[dot](u1)[label=above:\( u_1 \)][label={[vcolour]below:0}]{} --++(10:0.55) node[dot](v)[label=above:\( v \), label={[vcolour]below:1}]{};
       \end{tikzpicture} & (Class~C) \\
\tikz \node {(7)};&
\begin{tikzpicture}
\draw (0,0) node[dot](u1)[label=above:\( u_1 \), label={[vcolour,name=u1Colour]below:0}]{} --++(10:0.55) node[dot](u2)[label=above:\( u_2 \)]{} --++(10:0.55) node[dot](v)[label={[xshift=-7pt,yshift=-2pt,anchor=south west]above:\( v{=}u_3 \)}]{};
\node [fit={(u1) (v)}](left iv)[label={[font=\scriptsize,label distance=3pt]below:(\( d=2 \))}]{};
\end{tikzpicture}
           & \tikz \path [-stealth,draw=black,ultra thick] (0,0)--(0.35,0);
                & \begin{tikzpicture}
                  \draw (0,0) node[dot]{} --++(10:0.55) node[dot](u3)[label=above:\( v \), label={[vcolour,name=u3Colour]below:1}]{};
                  \end{tikzpicture} & (Class~D) \\
\tikz \node {(8)};&
\begin{tikzpicture}
\draw (0,0) node[dot](u1)[label=above:\( u_1 \), label={[vcolour,name=u1Colour]below:0}]{} --++(10:0.55) node[dot](u2)[label=above:\( u_2 \)]{} --++(10:0.55) node[dot](u3)[label=above:\( u_3 \)]{} --++(10:0.275) ++(10:0.55) --++(10:0.275) node[dot](ud)[label=above:\( u_d \)]{} --++(10:0.55) node[dot](v)[label={[xshift=-7pt,yshift=-2pt,anchor=south west]above:\( v{=}u_{d+1} \)}]{};
\path (u3) --node[sloped]{\tiny \( \dots \)} (ud);
\node [fit={(u1) (v)}][label={[font=\scriptsize,label distance=4pt]below:(up-distance \( =d\geq 3 \))}]{};
\end{tikzpicture}
           & \tikz \path [-stealth,draw=black,ultra thick] (0,0)--(0.35,0);
                & \begin{tikzpicture}
                  \draw (0,0) node[dot](u3)[label=above:\( u_3 \), label={[vcolour,name=u3Colour]below:1}]{} --++(10:0.55) node[dot](u4)[label=above:\( u_4 \)]{} --++(10:0.275) ++(10:0.55) --++(10:0.275) node[dot](ud)[label=above:\( u_d \)]{} --++(10:0.55) node[dot](v)[label=above:\( v \)]{}; 
                  \path (u4) --node[sloped]{\tiny \( \dots \)} (ud);
                  \node [fit={(u1) (v)}][label={[font=\scriptsize,label distance=4pt]below:(up-distance \( =d-2 \))}]{};
                  \end{tikzpicture} &  \\
};
\end{tikzpicture}
\caption{Branch replacement operations (1) to (8)}
\label{fig:branch replacement 1 to 8}
\end{figure}
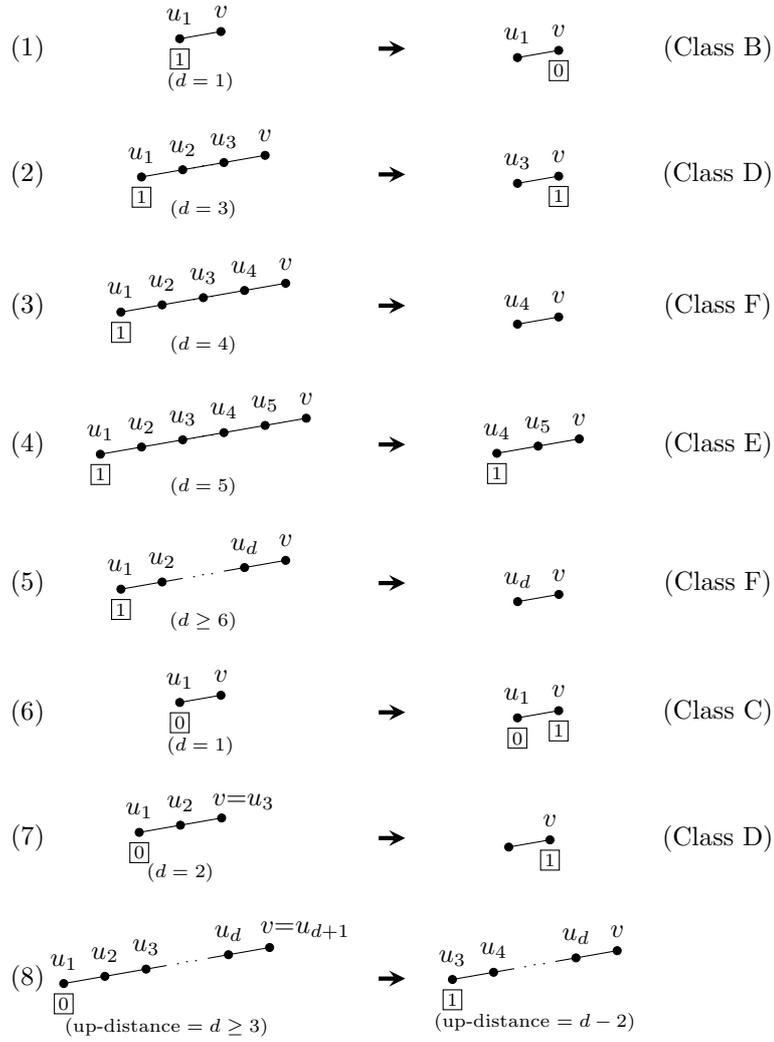

\begin{figure}[hbtp] 
\centering
\begin{tikzpicture}
\matrix[matrix of nodes,row sep=0.05cm,nodes={anchor=center}]{
\tikz \node {(9)};&[-0.4cm]
\begin{tikzpicture}
\draw (0,0) node[dot](w)[label=above:\( w \), label={[vcolour]below:0}]{} --++(10:0.55) node[dot](u1)[label=above:\( u_1 \), label={[vcolour,xshift=1.5pt]below left:1}]{} --++(10:0.55) node[dot](u2)[label=above:\( u_2 \)]{} --++(10:0.55) node[dot](v)[label={[xshift=-7pt,yshift=-2pt,anchor=south west]above:\( v{=}u_3 \)}]{};
\draw (u1) --+(-80:0.65) node[dot](l)[label=right:\( \ell \)]{};
\node [fit={(w) (v) (l)}][label={[font=\scriptsize]below:(\( d=2 \))}]{};
\end{tikzpicture}
           & \tikz \path [-stealth,draw=black,ultra thick] (0,0)--(0.35,0);
                & \begin{tikzpicture}
                  \draw (0,0) node[dot]{} --++(10:0.55) node[dot](u3)[label=above:\( v \), label={[vcolour,name=u3Colour]below:0}]{};
                  \end{tikzpicture} & (Class~B)\\
\tikz \node {(10)};&
\begin{tikzpicture}
\draw (0,0) node[dot](w)[label=above:\( w \), label={[vcolour]below:0}]{} --++(10:0.55) node[dot](u1)[label=above:\( u_1 \), label={[vcolour,xshift=1.5pt]below left:1}]{} --++(10:0.55) node[dot](u2)[label=above:\( u_2 \)]{} --++(10:0.55) node[dot](u3)[label=above:\( u_3 \)]{} --++(10:0.275) ++(10:0.55) --++(10:0.275) node[dot](ud)[label=above:\( u_d \)]{} --++(10:0.55) node[dot](v)[label={[xshift=-7pt,yshift=-2pt,anchor=south west]above:\( v{=}u_{d+1} \)}]{};
\path (u3) --node[sloped]{\tiny \( \dots \)} (ud);
\draw (u1) --+(-80:0.65) node[dot](l)[label=right:\( \ell \)]{};
\node [fit={(w) (v) (l)}](left iv)[label={[font=\scriptsize]below:(up-distance \( =d\geq 3 \))}]{};
\end{tikzpicture}
           & \tikz \path [-stealth,draw=black,ultra thick] (0,0)--(0.35,0);
                & \begin{tikzpicture}
                  \draw (0,0) node[dot](u3)[label=above:\( u_3 \), label={[vcolour,name=u3Colour]below:0}]{} --++(10:0.55) node[dot](u4)[label=above:\( u_4 \)]{} --++(10:0.275) ++(10:0.55) --++(10:0.275) node[dot](ud)[label=above:\( u_d \)]{} --++(10:0.55) node[dot](v)[label=above:\( v \)]{}; 
                  \path (u4) --node[sloped]{\tiny \( \dots \)} (ud);
                  \node [fit={(u3) (u3Colour) (v)}](left iv)[label={[font=\scriptsize]below:(up-distance \( =d-2 \))}]{};
                  \end{tikzpicture} & \\
\tikz \node {(11)};&
\begin{tikzpicture}
\draw (0,0) node[dot](y)[label=above:\( y \), label={[vcolour]below:1}]{} --++(10:0.55) node[dot](x)[label=above:\( x \)]{} --++(10:0.55) node[dot](u1)[label=above:\( u_1 \)]{} --++(10:0.55) node[dot](v)[label=above:\( v \)]{};
\draw (u1) --+(-80:0.65) node[dot](l)[label=right:\( \ell \)]{};
\node [fit={(y) (v) (l)}][label={[font=\scriptsize]below:(\( d=1 \))}]{};
\end{tikzpicture}
        & \tikz \path [-stealth,draw=black,ultra thick] (0,0)--(0.35,0);
              & \begin{tikzpicture}
              \draw (0,0) node[dot](u1)[label=above:\( u_1 \)][label={[vcolour]below:0}]{} --++(10:0.55) node[dot](v)[label=above:\( v \), label={[vcolour]below:1}]{};
              \end{tikzpicture} & (Class~C)\\
\tikz \node {(12)};&
\begin{tikzpicture}
\draw (0,0) node[dot](y)[label=above:\( y \), label={[vcolour]below:1}]{} --++(10:0.55) node[dot](x)[label=above:\( x \)]{} --++(10:0.55) node[dot](u1)[label=above:\( u_1 \)]{} --++(10:0.55) node[dot](u2)[label=above:\( u_2 \)]{} --++(10:0.275) ++(10:0.55) --++(10:0.275) node[dot](ud)[label=above:\( u_d \)]{} --++(10:0.55) node[dot](v)[label=above:\( v \)]{};
\path (u2) --node[sloped]{\tiny \( \dots \)} (ud);
\draw (u1) --+(-80:0.65) node[dot](l)[label=right:\( \ell \)]{};
\node [fit={(y) (v) (l)}](left iv)[label={[font=\scriptsize]below:(up-distance \( =d\geq 2 \))}]{};
\end{tikzpicture}
           & \tikz \path [-stealth,draw=black,ultra thick] (0,0)--(0.35,0);
                & \begin{tikzpicture}
                  \draw (0,0) node[dot](y)[label=above:\( y \), label={[vcolour,name=yColour]below:1}]{} --++(10:0.55) node[dot](x)[label=above:\( x \)]{} --++(10:0.55) node[dot](u1)[label=above:\( u_1 \)]{} --++(10:0.55) node[dot](u2)[label=above:\( u_2 \)]{} --++(10:0.275) ++(10:0.55) --++(10:0.275) node[dot](ud)[label=above:\( u_d \)]{} --++(10:0.55) node[dot](v)[label=above:\( v \)]{}; 
                  \path (u2) --node[sloped]{\tiny \( \dots \)} (ud);
                  \node [fit={(y) (yColour) (v)}](left iv)[label={[font=\scriptsize]below:(up-distance \( =d+2 \))}]{};
                  \end{tikzpicture} & \\
\tikz \node {(13)};&
\begin{tikzpicture}
\draw (0,0) node[dot](u1)[label=above:\( u_1 \)]{} --++(10:0.55) node[dot](v)[label=above:\( v \)]{};
\draw (u1) --+(-110:0.65) node[dot](l1)[label=left:\( \ell_1 \)]{}
      (u1) --+(-50:0.65) node[dot](l2)[label=right:\( \ell_2 \)]{};
\node [fit={(v) (l1)}][label={[font=\scriptsize]below:(\( d=1 \))}]{};
\end{tikzpicture}
             & \tikz \path [-stealth,draw=black,ultra thick] (0,0)--(0.35,0);
                  & \begin{tikzpicture}
                    \draw (0,0) node[dot](w)[label=above:\( w \), label={[vcolour]below:1}]{} --++(10:0.55) node[dot](u1)[label=above:\( u_1 \)]{} --++(10:0.55) node[dot](v)[label=above:\( v \)]{};
                    \end{tikzpicture} & (Class~E) \\
\tikz \node {(14)};&
\begin{tikzpicture}
\draw (0,0) node[dot](u1)[label=above:\( u_1 \)]{} --++(10:0.55) node[dot](u2)[label=above:\( u_2 \)]{} --++(10:0.275) ++(10:0.55) --++(10:0.275) node[dot](ud)[label=above:\( u_d \)]{} --++(10:0.55) node[dot](v)[label=above:\( v \)]{};
\path (u2) --node[sloped]{\tiny \( \dots \)} (ud);
\draw (u1) --+(-110:0.65) node[dot](l1)[label=left:\( \ell_1 \)]{}
      (u1) --+(-50:0.65) node[dot](l2)[label=right:\( \ell_2 \)]{};
\node [fit={(v) (l1)}][label={[font=\scriptsize]below:(\( d\geq 2 \))}]{};
\end{tikzpicture}
           & \tikz \path [-stealth,draw=black,ultra thick] (0,0)--(0.35,0);
                & \begin{tikzpicture}
                  \draw (0,0) node[dot](ud)[label=above:\( u_d \)]{} --++(10:0.55) node[dot](v)[label=above:\( v \)]{}; 
                  \end{tikzpicture} & (Class~F) \\
\tikz \node {(15)};&
\begin{tikzpicture}
\draw (0,0) node[dot](u1)[label=above:\( u_1 \)]{} --++(10:0.55) node[dot](u2)[label=above:\( u_2 \)]{} --++(10:0.275) ++(10:0.55) --++(10:0.275) node[dot](ud)[label=above:\( u_d \)]{} --++(10:0.55) node[dot](v)[label=above:\( v \)]{};
\path (u2) --node[sloped]{\tiny \( \dots \)} (ud);
\node [fit={(u1) (v)}][label={[font=\scriptsize]below:(\( d\geq 2 \))}]{};
\end{tikzpicture}
           & \tikz \path [-stealth,draw=black,ultra thick] (0,0)--(0.35,0);
                & \begin{tikzpicture}
                  \draw (0,0) node[dot](ud)[label=above:\( u_d \)]{} --++(10:0.55) node[dot](v)[label=above:\( v \)]{}; 
                  \end{tikzpicture} & (Class~F)\\
};
\end{tikzpicture}
\caption{Branch replacement operations (9) to (15)}
\label{fig:branch replacement 9 to 15}
\end{figure}
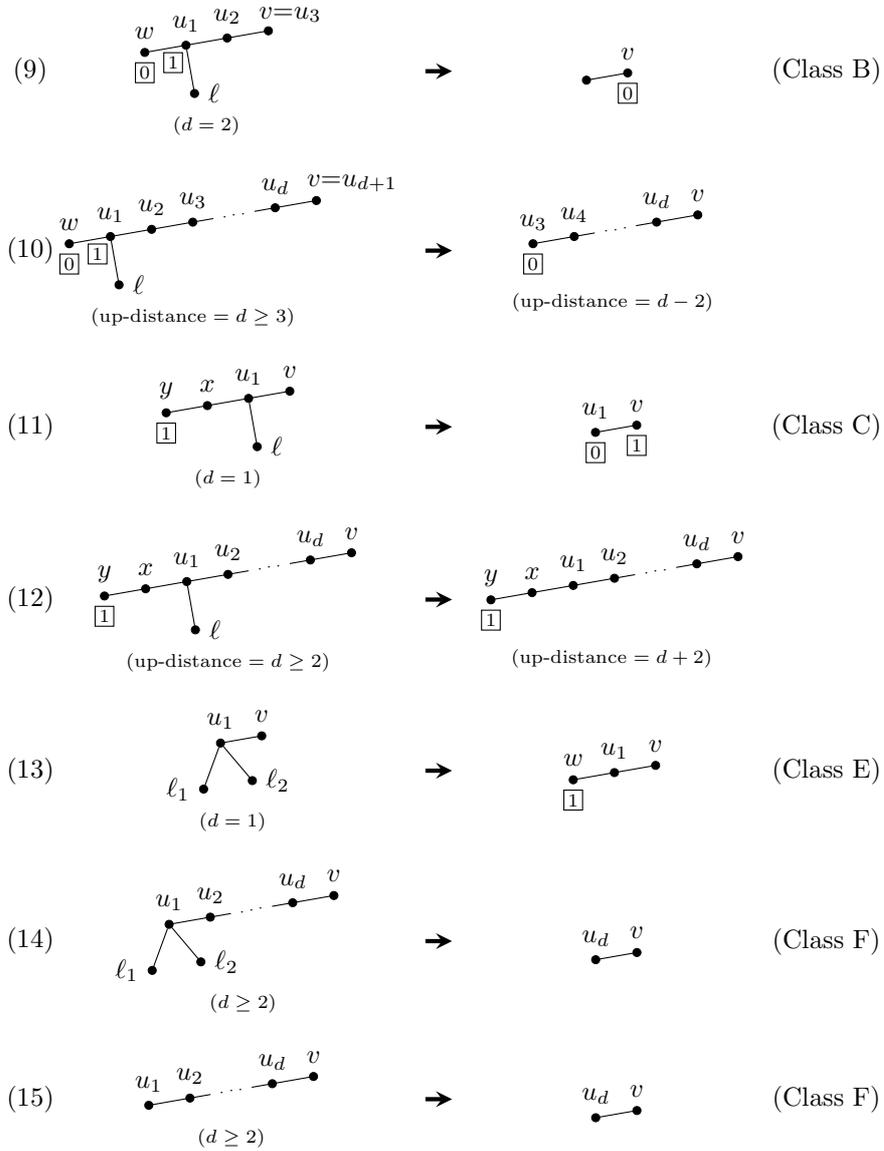

\FloatBarrier

\subsection{Proof of Table~\ref{tbl:class of rooted subtree} (Overview)}\label{sec:class of rooted subtree}
Let \( T \) be a partially coloured rooted tree with a 3-plus vertex as the root. Let \( v \) be a vertex of \( T \) with \( \deg_T(v)\neq 2 \). Suppose that the rooted subtree \( R=T_v \) is composed of branches of \( T \) at \( v \) each of which belong to Classes~\( B \) to \( F \). Let \( b,c,d,e,f \) denote respectively the number of Class~B branches, Class~C branches, \( \dots \), Class~F branches (of \( T \) at \( v \)). Each entry of Table~\ref{tbl:class of rooted subtree} is proved by showing that replacing \( R \) by a suitable rooted subtree preserves 3-rs colouring extension status. 

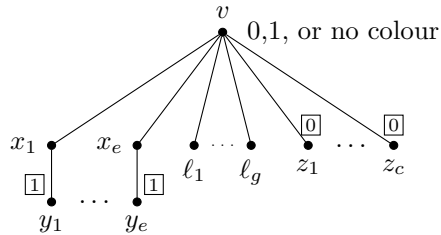
\begin{figure}[hbt]
\centering
\begin{tikzpicture}[scale=1.5]
\node [dot](u) [label=above:\( v \),label={[label distance=4pt]right:0,1, or no colour}]{};
\draw (u) --+(-0.25,-1) node[dot](l1) [label=below:\( \ell_1 \)]{}
      (u) --+(0.25,-1) node[dot](lt) [label=below:\( \ell_g \)]{}
      (u) --+(0.75,-1) node[dot](z1) [label=below:\( z_1 \),label={[vcolour,xshift=1pt]above:0}]{}
      (u) --+(1.50,-1) node[dot](zq) [label=below:\( z_c \),label={[vcolour]above:0}]{}
      (u) --+(-1.50,-1) node[dot](x1) [label=left:\( x_1 \)]{}
      (u) --+(-0.75,-1) node[dot](xs) [label=left:\( x_e \)]{};
\draw (x1) --+(0,-0.5) node[dot](y1) [label=below:\( y_1 \), label={[vcolour]above left:1}]{}
      (xs) --+(0,-0.5) node[dot](ys) [label=below:\( y_e \), label={[vcolour]above right:1}]{};
\path (y1) --node{\( \dots \)} (ys);
\path (l1) --node{\tiny\( \dots \)} (lt);
\path (z1) --node{\( \dots \)} (zq);
\end{tikzpicture}
\caption{General form of \( R \) (\( e\geq 0, g\geq 0, c\geq 0 \))}
\label{fig:general form of R}
\end{figure}

By definition of the equivalence relation on branches, each branch \( B^* \) of \( T \) at \( v \) can be replaced by the representative of the equivalence class of \( B^* \) without affecting 3-rs colouring extension status of \( T \). Unless \( b>0 \) and \( c+d>0 \), all these replacements can be carried out simultaneously, and thus we may assume that \( R \) is the graph in Figure~\ref{fig:general form of R} where \( g=b+d+f \) and \( v \) is coloured~1 if \( c>0 \).

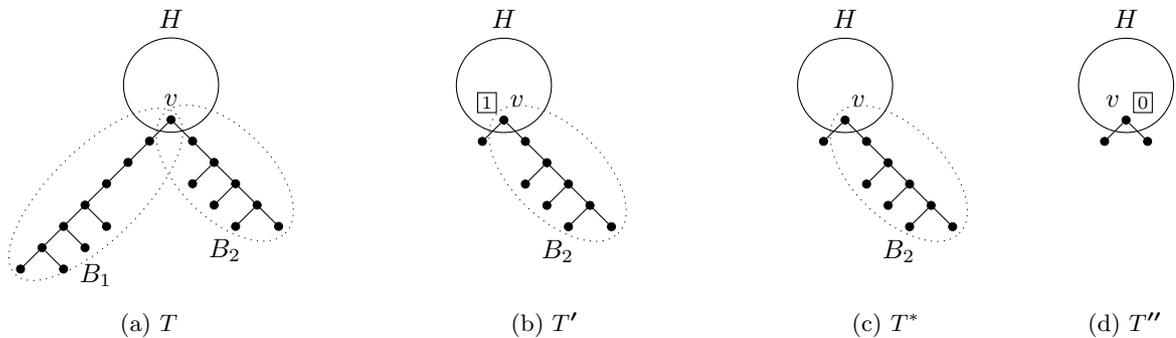
\begin{figure}[hbt]
\centering
\begin{subfigure}[t]{0.33\textwidth}
\centering
\begin{tikzpicture}[scale=0.4]
\centering
\node(v) [dot][label=above:\( v \)]{};
\draw (v)--++(-135:1) node[dot]{} --++(-135:1) node[dot]{}--++(-135:1) node[dot]{}--++(-135:1) node(w)[dot]{}--++(-135:1) node(x)[dot]{}--++(-135:1) node(y)[dot]{}--++(-135:1) node(yleaf)[dot]{};
\draw (w)--+(-45:1) node[dot]{};
\draw (x)--+(-45:1) node[dot]{};
\draw (y)--+(-45:1) node[dot]{};
\draw (v)--++(-45:1) node[dot]{} --++(-45:1) node(w')[dot]{}--++(-45:1) node(x')[dot]{}--++(-45:1) node(y')[dot]{}--++(-45:1) node(y'leaf)[dot]{};
\draw (w')--+(-135:1) node[dot]{};
\draw (x')--+(-135:1) node[dot]{};
\draw (y')--+(-135:1) node[dot]{};

\path (v)--coordinate(midOfB1) (yleaf);
\node at (midOfB1) [draw,ellipse,minimum width=3cm,minimum height=1.2cm,rotate=45,dotted][label=below left:\( B_1 \)]{};
\path (v)--coordinate(midOfB2) (y'leaf);
\node at (midOfB2) [draw,ellipse,minimum width=2.25cm,minimum height=1.2cm,rotate=-45,dotted][label=below right:\( B_2 \)]{};
\draw (v)++(0,1.15) node[subgraph](H)[label=\( H \)]{};
\end{tikzpicture}
\caption{\( T \)}
\end{subfigure}%
\begin{subfigure}[t]{0.25\textwidth}
\centering
\begin{tikzpicture}[scale=0.4]
\centering
\node(v) [dot][label={[xshift=5pt]above:\( v \)}][label={[vcolour]above left:1}]{};
\draw (v)--++(-135:1) node[dot]{};
\draw (v)--++(-45:1) node[dot]{} --++(-45:1) node(w')[dot]{}--++(-45:1) node(x')[dot]{}--++(-45:1) node(y')[dot]{}--++(-45:1) node(y'leaf)[dot]{};
\draw (w')--+(-135:1) node[dot]{};
\draw (x')--+(-135:1) node[dot]{};
\draw (y')--+(-135:1) node[dot]{};

\path (v)--coordinate(midOfB2) (y'leaf);
\node at (midOfB2) [draw,ellipse,minimum width=2.25cm,minimum height=1.2cm,rotate=-45,dotted][label=below right:\( B_2 \)]{};
\draw (v)++(0,1.15) node[subgraph](H)[label=\( H \)]{};

\draw (v)+(0,-5.5) node{}; 
\end{tikzpicture}
\caption{\( T\bm{'} \)}
\end{subfigure}%
\begin{subfigure}[t]{0.25\textwidth}
\centering
\begin{tikzpicture}[scale=0.4]
\centering
\node(v) [dot][label={[xshift=5pt]above:\( v \)}]{};
\draw (v)--++(-135:1) node[dot]{};
\draw (v)--++(-45:1) node[dot]{} --++(-45:1) node(w')[dot]{}--++(-45:1) node(x')[dot]{}--++(-45:1) node(y')[dot]{}--++(-45:1) node(y'leaf)[dot]{};
\draw (w')--+(-135:1) node[dot]{};
\draw (x')--+(-135:1) node[dot]{};
\draw (y')--+(-135:1) node[dot]{};

\path (v)--coordinate(midOfB2) (y'leaf);
\node at (midOfB2) [draw,ellipse,minimum width=2.25cm,minimum height=1.2cm,rotate=-45,dotted][label=below right:\( B_2 \)]{};
\draw (v)++(0,1.15) node[subgraph](H)[label=\( H \)]{};

\draw (v)+(0,-5.5) node{}; 
\end{tikzpicture}
\caption{\( T^* \)}
\end{subfigure}%
\begin{subfigure}[t]{0.1\textwidth}
\centering
\begin{tikzpicture}[scale=0.4]
\centering
\node(v) [dot][label={[xshift=-5pt]above:\( v \)}][label={[vcolour]above right:0}]{};
\draw (v)--++(-135:1) node[dot]{};
\draw (v)--++(-45:1) node[dot]{};

\draw (v)++(0,1.15) node[subgraph](H)[label=\( H \)]{};

\draw (v)+(0,-5.5) node{}; 
\end{tikzpicture}
\caption{\( T\bm{''} \)}
\end{subfigure}%
\caption{The branch \( B_1\in \) Class~D. So, \( T \) admits a 3-rs colouring extension \( \iff \) \( T\bm{'} \) admits a 3-rs colouring extension \( \iff \) \( T^* \) admits a 3-rs colouring extension \( f^* \) such that \( f^*(v)=1 \). The branch \( B_2\in \) Class~B. So, \( T^* \) admits a 3-rs colouring extension \( f^* \) such that \( f^*(v)=1 \) \( \iff \) \( T\bm{''} \) admits a 3-rs colouring extension \( f\bm{''} \) such that \( f\bm{''}(v)=1 \) (impossible obviously). So, \( T \) does not admit a 3-rs colouring extension}
\label{fig:eg colour conflict}
\end{figure}

In the case \( b>0 \) and \( c+d>0 \), \( T \) does not admit a 3-rs colouring extension due to a colour conflict at \( v \) (see Figure~\ref{fig:eg colour conflict} for an example; see supplementary material for the proof of \( B_1\in \) Class~D and \( B_2\in \) Class~B). Thus, \( R\in \) Class~I by Lemma~\ref{lem:classes 1 and I}. This proves the entry \( b>0 \), \( c+d>0 \) of Table~\ref{tbl:class of rooted subtree}. Figure~\ref{fig:subtree replacement} shows an overview of the remaining cases. To prove Table~\ref{tbl:class of rooted subtree}, it suffices to show that each rooted subtree replacement operation of Figure~\ref{fig:subtree replacement} (drawn by a thick arrow) preserves 3-rs colouring extension status (proof is available in supplementary material). 


\begin{figure}[hbtp] 
\centering
\begin{subfigure}[b]{\textwidth}
\centering
\begin{tikzpicture}
\node at (0,0) (u)[dot] [label=above:\( v \),label={[vcolour,label distance=4pt]right:0}]{};
\draw (u) --++(-1,-1) node[dot](x1) [label=left:\( x_1 \)]{} --+(0,-1) node[dot](y1) [label=below:\( y_1 \), label={[vcolour]above left:1}]{}
      (u) --++(-0.25,-1) node[dot](xr) [label={[label distance=-2pt]left:\( x_e \)}]{} --+(0,-1) node[dot](yr) [label=below:\( y_e \), label={[vcolour]above right:1}]{}
      (u) --+(0.25,-1) node[dot](l1) [label=below:\( \ell_1 \)]{}
      (u) --+(1.0,-1) node[dot](ls) [label=below:\( \ell_g \)]{};
\path (y1) --node{\tiny \( \dots \)} (yr);
\path (l1) --node{\tiny \( \dots \)} (ls);

\path  (ls) ++(0.75,0) node(from)[inner sep=0pt]{} +(0.5,0) node(to)[inner sep=0pt]{};
\draw [-stealth,draw=black,ultra thick] (from) --(to);

\path (ls) ++(2.25,0.25) node(Col4-3)[inner sep=0pt]{};
\draw (Col4-3) node(u)[dot] [label=above:\( u\)][label={[vcolour,label distance=4pt]right:0}]{};
\path (u) ++(0,-0.75) node[align=center, font=\scriptsize]{\( R\!\in\! \) Class~II};
 
\end{tikzpicture}
\caption{Case 1: \( v \) is coloured 0 in \( R \)}
\label{fig:case 1}
\end{subfigure}

\begin{subfigure}[b]{0.45\textwidth}
\centering
\begin{tikzpicture}[scale=0.8]

\node(Col2Row1)[inner sep=0pt]{};

\node at (Col2Row1) (u)[dot] [label=above:\( v \),label={[vcolour,label distance=4pt]right:1}]{};
\draw (u) --+(-0.25,-1) node[dot](l1) [label=below:\( \ell_1 \)]{}
      (u) --+(0.25,-1) node[dot](ls) [label=below:\( \ell_g \)]{}
      (u) --+(0.75,-1) node[dot](z1) [label=below:\( z_1 \),label={[vcolour,xshift=1pt,font=\tiny]above:0}]{}
      (u) --+(1.70,-1) node[dot](zt) [label=below:\( z_c \),label={[vcolour,font=\tiny]above:0}]{}
      (u) --+(-1.50,-1) node[dot](x1) [label={[label distance=-3pt]below left:\( x_1 \)}]{}
      (u) --+(-0.75,-1) node[dot](xr) [label={[label distance=-3pt]below left:\( x_e \)}]{};

\draw (x1) --+(0,-1) node[dot](y1) [label=below:\( y_1 \), label={[vcolour]above left:1}]{}
      (xr) --+(0,-1) node[dot](yr) [label=below:\( y_e \), label={[vcolour]above right:1}]{};

\path (y1) --node{\tiny \( \dots \)} (yr);
\path (l1) --node{\tiny ...} (ls);
\path (z1) --node{\tiny \( \dots \)} (zt);

\path (Col2Row1) ++(0,-2.75) node(from)[inner sep=0pt]{} ++(0,-0.5) node(to)[inner sep=0pt]{};
\draw[-stealth,ultra thick] (from) -- (to);

\path (Col2Row1) ++(0,-4) node(Col2Row2)[inner sep=0pt]{};

\node at (Col2Row2) (u)[dot] [label=above:\( v \),label={[vcolour,label distance=4pt]right:1}]{};
\draw (u) --+(-0.25,-1) node[dot](l1) [label=below:\( \ell_1 \)]{}
      (u) --+(0.25,-1) node[dot](ls) [label=below:\( \ell_g \)]{}
      (u) --+(0.75,-1) node[dot](z1) [label=below:\( z_1 \),label={[vcolour,xshift=1pt,font=\tiny]above:0}]{}
      (u) --+(1.70,-1) node[dot](zt) [label=below:\( z_c \),label={[vcolour,font=\tiny]above:0}]{}
      (u) --+(-1.60,-1) node[dot](x1) [label={[label distance=-3pt]below left:\( x_1 \)},label={[vcolour,font=\tiny]above:0}]{}
      (u) --+(-0.75,-1) node[dot](xr) [label={[label distance=-3pt]below left:\( x_e \)},label={[vcolour,font=\tiny]above:0}]{};
\path (x1) --node{\tiny \( \dots \)} (xr);
\path (l1) --node{\tiny ...} (ls);
\path (z1) --node{\tiny \( \dots \)} (zt);

\path (Col2Row2) ++(-0.25,-2.0) node(from)[inner sep=0pt]{} ++(0,-1.5) node(to){};
\draw [->,thick] (from) --node[pos=0.7,right,font=\scriptsize,xshift=-1mm]{\( c\!+\!e\!=\!1 \)} (to);
\path (from) +(-1.5,-1.5) node(to1)[inner sep=0pt]{};
\draw [->,thick,shorten <=-1pt] (from) --node[sloped,above,font=\scriptsize,name=case1Label]{\( c+e=0 \)} (to1);
\path (from) +(2,-1.5) node(to3)[inner sep=0pt]{};
\draw [->,thick,shorten <=-1pt] (from) --node[sloped,above,font=\scriptsize]{\( c+e\geq 2 \)} (to3);

\path (Col2Row2) ++(-2,-4.25) node(Col2-1Row3)[inner sep=0pt]{};

\node at (Col2-1Row3) (u)[dot] [label=above:\( v \),label={[vcolour,label distance=4pt]right:1}]{};
\draw (u) --+(-0.25,-1) node[dot](l1) [label=below:\( \ell_1 \)]{}
      (u) --+(0.25,-1) node[dot](ls) [label=below:\( \ell_g \)]{};
\path (l1) --node{\tiny ...} (ls);

\path (Col2Row2) ++(-0.25,-4.25) node(Col2-2Row3)[inner sep=0pt]{};

\node at (Col2-2Row3) (u)[dot] [label=above:\( v \),label={[vcolour,label distance=4pt]right:1}]{};
\draw (u) --+(-0.25,-1) node[dot](l1) [label=below:\( \ell_1 \)]{}
       (u) --+(0.25,-1) node[dot](ls) [label=below:\( \ell_g \)]{}
       (u) --+(0.75,-1) node[dot](x) [label=below:\( w_1 \),label={[vcolour,xshift=1pt]above:0}]{};
\path (l1) --node{\tiny ...} (ls);


\path (Col2Row2) ++(2,-4.25) node(Col2-3Row3)[inner sep=0pt]{};

\node at (Col2-3Row3) (u)[dot] [label=above:\( v \),label={[vcolour,label distance=4pt,yshift=1pt]right:1}]{};
\draw (u) --+(-0.75,-1) node[dot](l1) [label=below:\( \ell_1 \)]{}
      (u) --+(-0.25,-1) node[dot](ls) [label=below:\( \ell_g \)]{};
\draw [ultra thick]
      (u) --+(0.15,-1) node[dot](w1) [label={[xshift=2pt]below:\( w_1 \)},label={[vcolour,yshift=7pt,xshift=-4pt,thin,font=\tiny]right:0}]{}
      (u) --+(0.75,-1) node[dot](w2) [label=below:\( w_2 \),label={[vcolour,yshift=-0.5pt,xshift=2pt,thin,font=\tiny]above:0}]{};
\draw (u) --+(1.70,-1) node[dot](wc+e) [label=below:\( w_{c+e} \),label={[vcolour,thin,font=\tiny]above:0}]{};
\path (l1) --node{\tiny ...} (ls);
\path (w2) --node{\tiny \( \dots \)} (wc+e);
\path (w1) ++(0.2,-1.1) node[align=right,font=\scriptsize]{\( R\!\in\! \) Class~I}; 

\path (Col2-1Row3) ++(0,-2.00) node(from)[inner sep=0pt]{} ++(0,-0.5) node(to)[inner sep=0pt]{};
\draw[-stealth,ultra thick] (from) -- (to);

\path (Col2-2Row3) ++(0.25,-2.00) node(from)[inner sep=0pt]{} ++(0,-0.5) node(to)[inner sep=0pt]{};
\draw[-stealth,ultra thick] (from) -- (to);


\path (Col2-1Row3) ++(0,-3.25) node (Col2-1Row4)[inner sep=0pt]{};

\node at (Col2-1Row4) (u)[dot] [label=above:\( v \),label={[vcolour,label distance=4pt]right:1}]{};

\path (Col2-2Row3) ++(0,-3.25) node (Col2-2Row4)[inner sep=0pt]{};

\node at (Col2-1Row4) (u)[dot] [label=above:\( v \),label={[vcolour,label distance=4pt]right:1}]{};

\path (u) ++(0,-0.75) node[align=right,font=\scriptsize]{\( R\!\in\! \) Class~IV}; 

\path (Col2-3Row3) ++(0,-3.25) node(Col2-3Row4)[inner sep=0pt]{};
\node at (Col2-2Row4) (u)[dot] [label=above:\( v \),label={[vcolour,label distance=4pt]right:1}]{};
\draw (u) --+(-0.25,-1) node[dot](l1) [label=below:\( \ell_1 \)]{}
       (u) --+(0.75,-1) node[dot](x) [label=below:\( w_1 \),label={[vcolour,xshift=1pt]above:0}]{};
\path (l1)--coordinate(ls) (x);
\path (ls) ++(0,-1.1) node[align=right,font=\scriptsize]{\( R\!\in\! \) Class~III}; 



\end{tikzpicture}
\captionsetup{width=0.85\linewidth}
\caption{Case 2: \( v \) is coloured 1 in \( R \)\\ (vertices \( x_1,x_2,\dots,x_e \),\,\( z_1,z_2,\dots,z_c \) are relabelled \( w_1,w_2,\dots,w_{c+e} \))}
\label{fig:case 2}
\end{subfigure}%
\hspace{0.25cm}
\begin{subfigure}[b]{0.45\textwidth}
\centering
\begin{tikzpicture}[scale=0.8]

\path (0,0) node(Col1Row1)[inner sep=0pt]{};

\draw (Col1Row1) node(u)[dot] [label=above:\( v \)]{} --++(-1,-1) node[dot](x1) [label=left:\( x_1 \)]{} --+(0,-1) node[dot](y1) [label=below:\( y_1 \), label={[vcolour]above left:1}]{}
      (u) --++(-0.25,-1) node[dot](xr) [label={[label distance=-2pt]left:\( x_e \)}]{} --+(0,-1) node[dot](yr) [label=below:\( y_e \), label={[vcolour]above right:1}]{}
      (u) --+(0.25,-1) node[dot](l1) [label=below:\( \ell_1 \)]{}
      (u) --+(1.0,-1) node[dot](ls) [label=below:\( \ell_f \)]{};
\path (y1) --node{\tiny \( \dots \)} (yr);
\path (l1) --node{\tiny \( \dots \)} (ls);

\path (Col1Row1) ++(0,-2.75) node(from)[inner sep=0pt]{} ++(0,-1.5) node(to){};
\draw [->,thick] (from) -- node[pos=0.7,right,font=\scriptsize,xshift=-0.5mm]{\( e=1 \)} (to);
\path (from) +(-1.5,-1.5) node(to1)[inner sep=0pt]{};
\draw [->,thick,shorten <=-1pt] (from) --node[sloped,above,font=\scriptsize]{\( e=0 \), \( f\!>\!0 \)} (to1);
\path (from) +(2,-1.5) node(to3)[inner sep=0pt]{};
\draw [->,thick,shorten <=-1pt] (from) --node[sloped,above,font=\scriptsize]{\( e\geq 2 \)} (to3);

\path (Col1Row1) ++(-1.75,-5) node(Col1-1Row2)[inner sep=0pt]{};

\node at (Col1-1Row2) (u)[dot] [label=above:\( v \)]{};
\draw (u) --+(-0.25,-1) node[dot](l1) [label=below:\( \ell_1 \)]{}
      (u) --+(0.25,-1) node[dot](ls) [label=below:\( \ell_f \)]{};
\path (l1) --node{\tiny ...} (ls);

\path (Col1Row1) ++(0,-5) node(Col1-2Row2)[inner sep=0pt]{};

\node at (Col1-2Row2) (u)[dot] [label=above:\( v \)]{};
\draw (u) --+(-0.5,-1) node[dot](x1) [label={[xshift=-2pt,yshift=-2pt]above:\( x_1 \)}]{}
      (u) --+(0,-1) node[dot](l1) [label=below:\( \ell_1 \)]{}
      (u) --+(0.5,-1) node[dot](ls) [label=below:\( \ell_f \)]{};
\draw (x1) --+(0,-1) node[dot](y1) [label=below:\( y_1 \), label={[vcolour]above left:1}]{};

\path (l1) --node{\tiny ...} (ls);

\path (Col1Row1) ++(2.25,-5) node(Col1-3Row2)[inner sep=0pt]{};

\node at (Col1-3Row2) (u)[dot] [label=above:\( v \)]{};
\draw (u) --+(-1.00,-1) node[dot](x1) [label={[xshift=0.5pt,yshift=-2pt]above:\( x_1 \)}]{}
      (u) --+(-0.25,-1) node[dot](x2) [label={[xshift=-3.5pt,yshift=-2pt]above:\( x_2 \)}]{}
      (u) --+(0.5,-1) node[dot](xr) [label={[xshift=2pt,yshift=-2pt]above left:\( x_e \)}]{}
      (u) --+(1.0,-1) node[dot](l1) [label=below:\( \ell_1 \)]{}
      (u) --+(1.75,-1) node[dot](ls) [label=below:\( \ell_f \)]{};

\draw (x1)--+(0,-1) node[dot](y1) [label=below:\( y_1 \), label={[vcolour]above left:1}]{}
      (x2) --+(0,-1) node[dot](y2) [label=below:\( y_2 \), label={[vcolour]above left:1}]{}
      (xr) --+(0,-1) node[dot](yr) [label=below:\( y_e \), label={[vcolour]above right:1}]{};

\path (y2) --node{\tiny \( \dots \)} (yr);
\path (l1) --node{\tiny \( \dots \)} (ls);

\path (Col1-1Row2) ++(0,-2) node(from)[inner sep=0pt]{} ++(0,-0.5) node(to)[inner sep=0pt]{};
\draw[-stealth,ultra thick] (from) -- (to);

\path (Col1-2Row2) ++(0,-2.750) node(from)[inner sep=0pt]{} ++(0,-0.5) node(to)[inner sep=0pt]{};
\draw[-stealth,ultra thick] (from) -- (to);

\path (Col1-3Row2) ++(0,-2.750) node(from)[inner sep=0pt]{} ++(0,-0.5) node(to)[inner sep=0pt]{};
\draw[-stealth,ultra thick] (from) -- (to);

\path (Col1-1Row2) ++(0,-3) node(Col1-1Row3)[inner sep=0pt]{};

\node at (Col1-1Row3) (u)[dot] [label=above:\( v \)]{};
\draw (u) --+(-0.25,-1) node[dot](l1) [label=below:\( \ell_1 \)]{}
      (u) --+(0.25,-1) node[dot](ls) [label=below:\( \ell_2 \)]{};

\path (u) ++(0,-2.0) node[align=center,font=\scriptsize]{\( R\!\in\! \) Class~VI};

\path (Col1-2Row2) ++(0,-4) node(Col1-2Row3)[inner sep=0pt]{};

\node at (Col1-2Row3) (u)[dot][xshift=5pt] [label=above:\( v \)]{};
\draw (u) --+(-0.5,-1) node[dot](x1) [label={[xshift=-2pt,yshift=-2pt]above:\( x_1 \)}]{}
      (u) --+(0,-1) node[dot](l1) [label=below:\( \ell_1 \)]{};
\draw (x1) --+(0,-1) node[dot](y1) [label=below:\( y_1 \), label={[vcolour]above left:1}]{};

\path (u) ++(0,-3) node[align=right,font=\scriptsize]{\( R\!\in\! \) Class~V};

\path (Col1-3Row2) ++(0,-4) node(Col1-3Row3)[inner sep=0pt]{};

\node at (Col1-3Row3) (u)[dot] [label=above:\( v \),label={[vcolour,yshift=1pt]right:0}]{};
\draw (u) --+(-1.00,-1) node[dot](x1) [label={[xshift=0.5pt,yshift=-2pt]above:\( x_1 \)}]{}
      (u) --+(-0.25,-1) node[dot](x2) [label={[xshift=-3.5pt,yshift=-2pt]above:\( x_2 \)}]{}
      (u) --+(0.5,-1) node[dot](xr) [label={[xshift=2pt,yshift=-2pt]above left:\( x_e \)}]{}
      (u) --+(1.0,-1) node[dot](l1) [label=below:\( \ell_1 \)]{}
      (u) --+(1.75,-1) node[dot](ls) [label=below:\( \ell_f \)]{};

\draw (x1)--+(0,-1) node[dot](y1) [label=below:\( y_1 \), label={[vcolour]above left:1}]{}
      (x2) --+(0,-1) node[dot](y2) [label=below:\( y_2 \), label={[vcolour]above left:1}]{}
      (xr) --+(0,-1) node[dot](yr) [label=below:\( y_e \), label={[vcolour]above right:1}]{};

\path (y2) --node{\tiny \( \dots \)} (yr);
\path (l1) --node{\tiny \( \dots \)} (ls);

\path (Col1-3Row3) ++(0,-2.75) node(from)[inner sep=0pt]{} ++(0,-0.5) node(to)[inner sep=0pt]{};
\draw[-stealth,ultra thick] (from) -- (to);

\path (Col1-3Row3) ++(0,-4) node(Col1-3Row4)[inner sep=0pt]{};

\node at (Col1-3Row4) (u)[dot] [label=above:\( v \),label={[vcolour]right:0}]{};

\path (u) ++(0,-0.65) node[align=right,font=\scriptsize]{\( R\!\in\! \) Class~II}; 


\end{tikzpicture}
\caption{Case 3: \( v \) is uncoloured (in \( R \))\\\hspace*{1cm}(provided \( e+f>0 \))}
\label{fig:case 3}
\end{subfigure}%
\caption{Determining the equivalence class of rooted subtree \( R \) in Table~2}
\label{fig:subtree replacement}
\end{figure}
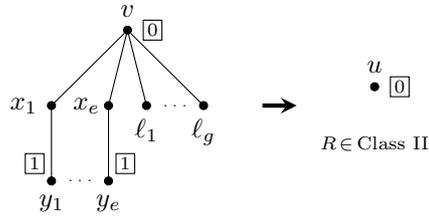
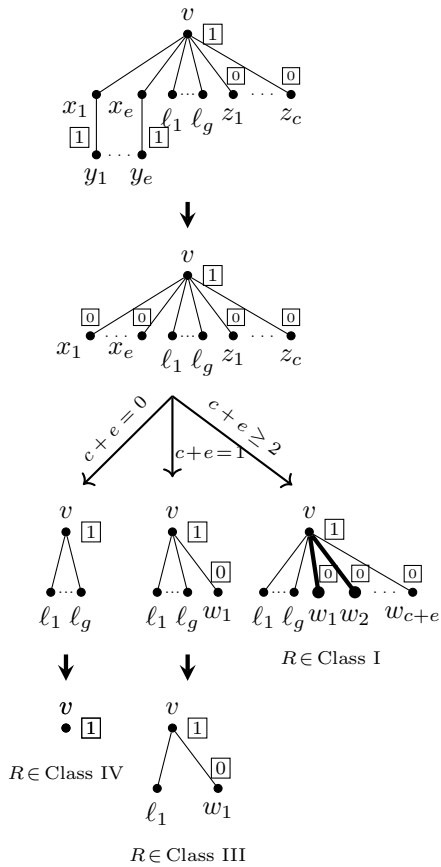
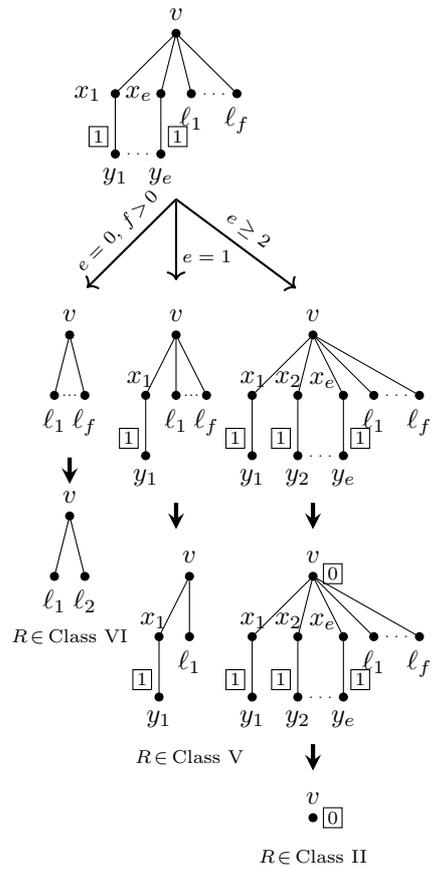

\FloatBarrier

\subsection{Pseudocode}\label{sec:pseudocode}
In this subsection, we provide the pseudocode of 3-rs colourability testing algorithm for trees, and discuss its run-time.

Recall that the input to the algorithm is a rooted tree \( T \) with a 3-plus vertex \( a \) as the root. Assume that \( T \) is stored as an ADT tree. If such a representation is not readily available, it can be obtained by BFS. For simplicity, it is assumed that the number of children of each vertex \( v \) is stored as \texttt{\#children[v]}. The following global variables are used in the pseudocode: \texttt{dist}, \texttt{branch\_class}, \texttt{subtree\_class}, and arrays \texttt{colour[v]}, \texttt{up\_distance[v]}, \texttt{\#ClassC\_branches[v]} and \texttt{\#ClassE\_branches[v]} (where vertex names are used as array indices). The variable \texttt{dist} is used to compute the up-distance of branches. The variables \texttt{branch\_class} and \texttt{subtree\_class} store respectively the equivalence class of the current branch and the current rooted subtree. The variable \texttt{colour[v]} has value 0 or 1 when vertex \( v \) is coloured 0 or 1, and value \( -1 \) when \( v \) is not coloured. Up-distance of a branch of \( T \) at \( v \) composed of a rooted subtree \( T_u \) and the \( u,v \)-path is stored as \texttt{up\_distance[u]} (this is the distance from \( u \) to its immediate ancestor which is a 3-plus vertex). The variables \texttt{\#ClassC\_branches[v]} and \texttt{\#ClassE\_branches[v]} store the number of Class~C branches at \( v \) and the number of Class~E branches at \( v \) respectively. To test 3-rs colourability of \( T \), call \texttt{test3rsColourability(\( T,a \))}.

The subroutine \texttt{subtreeClassify(\( v \))} computes the equivalence class of the rooted subtree \( T_v \), and stores it in the variable \texttt{subtree\_class}. If \( T_v\in\text{Class~I} \), then the subroutine prints ``\( T \) is not 3-rs colourable'' and halts the algorithm. Recall that \( a \) is the root of \( T \). For each vertex \( u\neq a \) of \( T \), the subroutine \texttt{branchClassify(\( u \))} computes the equivalence class of the branch \( B^* \) composed of the rooted subtree \( T_u \) and the \( u,v \)-path where \( v \) is the immediate ancestor of \( u \) which is a 3-plus vertex; the subroutine also stores it in the variable \texttt{branch\_class}. If \( B^*\in \) Class~A, then the subroutine prints ``\( T \) is not 3-rs colourable'' and halts the algorithm. For \( u=a \), the subroutine \texttt{branchClassify(\( u \))} prints ``\( T \) is 3-rs colourable'' and halts the algorithm. For every vertex \( v \) of \( T \) with \( \deg_T(v)\neq 2 \), the subroutine \texttt{traverseSubtree(\( T,v \))} performs a post-order traversal of \( T_v \), and helps to compute the equivalence class of \( T_v \) recursively.\\ 
\vspace{0.25cm}

\begin{algorithm}[H]
  \DontPrintSemicolon
  \Fn{test3rsColourability(T,a)}{
    \tcc{The first argument \( T \) is an ADT tree \( T \), and the second argument \( a \) is the root of \( T \).}
    dist \( \gets -1 \)\tcc*{dist is initialized to -1 to ensure that up\_\,distance[\( a \)] is assigned value 0.}
    traverseSubtree(\( T \),\,\( a \))
    }
\end{algorithm}

\vspace{0.25cm}

\begin{algorithm}[H]
  \DontPrintSemicolon
  \Fn{tryToColour(\( w \),col)}{
    \lIf{\upshape colour[\( w \)]= \( -1 \)}{
      colour[\( w \)] \( \gets \) col 
      }
    \ElseIf{\upshape colour[\( w \)] \( \neq \) col}{
      Print ``\( T \) is not 3-rs colourable'', and halt\;
      \tcc{by colour conflict; see Figure \ref{fig:eg colour conflict}}
      }
    }
\end{algorithm}

\begin{algorithm}[hbt]
  \DontPrintSemicolon
  \Fn{traverseSubtree(T,v)}{
    dist \( \gets \) dist+1 \tcc*{to compute up-distance[\( v \)]}
    \If(\tcc*[f]{ie, v is a 3-plus vertex}){\upshape \#children[\( v \)]\( \geq 2 \)}{
      up\_distance[\( v \)] \( \gets \) dist\;
      colour[\( v \)] \( \gets -1 \) \tcc*{initialization}
      \#ClassC\_branches[\( v \)] \( \gets \) 0\tcc*{initialization}
      \#ClassE\_branches[\( v \)] \( \gets \) 0\tcc*{initialization}
      \ForEach{\upshape child \( w \) of \( v \)}{
        dist \( \gets \) 0\tcc*{initialization}
        traverseSubtree(\( T \),\,\( w \))\;
        \If{\upshape branch\_class=C}{
          \#ClassC\_branches[\( v \)] \( \gets \) \#ClassC\_branches[\( v \)]+1
          }
        \ElseIf{\upshape branch\_class=E}{
          \#ClassE\_branches[\( v \)] \( \gets \) \#ClassE\_branches[\( v \)]+1
          }
        \lIf{\upshape branch\_class=B}{
          tryToColour(\( v \),\,0)
          }
        \lElseIf{\upshape branch\_class=C or D}{
          tryToColour(\( v \),\,1)
          }
        \tcc{(to spot colour conflict; see Figure \ref{fig:eg colour conflict})}
        }
      subtreeClassify(\( v \)) \tcc*{finds class of \( T_v \)}
      branchClassify(\( v \)) \tcc*{(for \( v\neq a \)) finds class of the branch composed of \( T_v \) and the \( v,x \)-path where \( x \) the immediate ancestor of \( v \) which is a 3-plus vertex.}
      }
    \ElseIf{\upshape \#children[\( v \)]=1}{
      \( w \gets \) child of \( v \)\;
      traverseSubtree(\( T \),\,\( w \))\;
      }
    \Else(\tcc*[f]{ie, v is a leaf}){
      up\_distance[\( v \)] \( \gets \) dist\;
      subtree\_class \( \gets \) VII\;
      \tcc{by the case b\,=\,c\,=\,d\,=\,e\,=\,f\,=\,0 of Table \ref{tbl:class of rooted subtree}}
      branchClassify(\( v \))\;
      }
    }
\end{algorithm}

\begin{algorithm}[hbtp]
  \DontPrintSemicolon
  \Fn{subtreeClassify(\( v \))}{
    \tcc{Determines the equivalence class of \( T_v \) assuming \( T_v\not\cong K_1 \). The subroutine traverseSubtree() deals with the case \( T_v\cong K_1 \) (i.e., b\,=\,c\,=\,d\,=\,e\,=\,f\,=\,0). \qquad \qquad \quad 
If \( T_v\in\text{Class I} \), the subroutine prints `\!\!`\( T \) is not 3-rs colourable'\!\!', and halts the algorithm.}
    \If(\tcc*[f]{Case 1 of Table~\ref{tbl:class of rooted subtree}: b\,>\,0}){\upshape colour[\( v \)]=0}{
      subtree\_class \( \gets \) II\;
      \tcc{here, c\,=\,d\,=\,0. The subroutine tryToColour() handles the case b\,>\,0, c+d\,>\,0.}
      }
    \ElseIf(\tcc*[f]{Case 2 of Table~\ref{tbl:class of rooted subtree}:\,b\,=\,0,\,c\,+\,d\,>\,0}){\upshape colour[\( v \)]=1}{
      \If{\upshape \#ClassC\_branches[\( v \)] + \#ClassE\_branches[\( v \)] = 0}{
        \tcc{i.e., Subcase 2.1: c\,+\,e\,=\,0}
        subtree\_class \( \gets \) IV\;
        }
      \ElseIf{\upshape \#ClassC\_branches[\( v \)] + \#ClassE\_branches[\( v \)] = 1}{
        \tcc{i.e., Subcase 2.2: c\,+\,e\,=\,1}
        subtree\_class \( \gets \) III\;
        }
      \Else{
        \tcc{i.e., Subcase 2.3: c\,+\,e\,\( \geq \)\,2}
        subtree\_class \( \gets \) I\;
        Print ``\( T \) is not 3-rs colourable'', and halt \tcc*{by Lemma~\ref{lem:classes 1 and I}}
        }
      }
    \Else(\tcc*[f]{Case 3 of Table~\ref{tbl:class of rooted subtree}: b\,=\,c\,=\,d\,=0}){
      \If{\upshape \#ClassE\_branches[\( v \)]=0}{
        \tcc{i.e., Subcase 3.1: e\,=\,0 and f\,>\,0.\\
        (since \( T_v\not\cong K_1 \), b\,=\,c\,=\,d\,=\,e\,=\,0 \( \implies \) f\,>\,0).}
        subtree\_class \( \gets \) VI
        }
      \ElseIf{\upshape \#ClassE\_branches[\( v \)]=1}{
        \tcc{i.e., Subcase 3.2: e\,=\,1}
        subtree\_class \( \gets \) V
        }
      \Else{
        \tcc{i.e., Subcase 3.3: e\,\( \geq \)\,2}
        subtree\_class \( \gets \) II
        }
      }
    }
\end{algorithm}


\begin{algorithm}[hbt]
  \DontPrintSemicolon
  \Fn{branchClassify(\( u \))}{
  \tcc{For \( u\neq a \), let \( v \) be the immediate ancestor of \( u \) which is a 3-plus vertex, and let \( B^* \) be the branch composed of the rooted subtree \( T_u \) and the \( u,v \)-path. The subroutine determines the equivalence class of \( B^* \) if \( u\neq a \). If \( B^*\in \) Class~A, it prints `\!\!`\( T \) is not 3-rs colourable'\!\!', and halts the algorithm. Since the algorithm halts whenever a Class~I rooted subtree is encountered, \( T_u \) is not in Class~I. For \( u=a \), the subroutine reports that \( T \) is 3-rs colourable, and halts the algorithm.}
  \If{\upshape up\_distance[\( u \)]=0}{
    \tcc{this happens only if u\,=\,a (i.e., the root of T)}
    Print ``\( T \) is 3-rs colourable", and halt
    }
  \Else{
    Using Table~\ref{tbl:class of branch}, find the equivalence class of \( B^* \) based on \texttt{subtree\_class} (i.e., class of \( T_u \)) and \texttt{up\_distance[\( u \)]}\\(i.e., up-distance of \( B^* \)), and store it in \texttt{branch\_class}\;
    \If{\upshape branch\_class=A}{
      Print ``\( T \) is not 3-rs colourable'', and halt\tcc*{by Lemma~\ref{lem:classes 1 and I}}
      }
    }
  }
\end{algorithm}



\FloatBarrier


It is time to see the run-time analysis of the algorithm. The subroutine \texttt{traverseSubtree()} carries out a post-order traversal of the tree, and performs some processing at each vertex visited; therefore, it is easy to analyze the time complexity of the algorithm by comparing it to post-order traversal. It is easy to verify that the subroutine \texttt{branchClassify()} requires at most 13 steps (see supplementary material for detailed pseudocode of \texttt{BranchClassify()}), and the subroutine \texttt{subtreeClassify()} requires at most 6 steps. In the run-time analysis of the algorithm, the cost of subroutine calls, except for subroutine \texttt{traverseSubtree()}, can be ignored because each call for other subroutines can be replaced by respective code. There is at most one call of subroutines \texttt{branchClassify()} and \texttt{subtreeClassify()} per vertex. Also, other lines of the subroutine \texttt{traverse\allowbreak Subtree()} such as \texttt{dist \( \gets \) dist+1} and \texttt{dist \( \gets 0 \)} are executed at most once per vertex. At most 14 such lines are executed per vertex. Hence, the worst-case run-time of the algorithm is \( (33+C)n \) steps where \( C \) is the cost of a subroutine call.

\FloatBarrier

\subsection{Chordal Graphs}\label{sec:chordal}
In this subsection, we show that \textsc{3-RS Colourability} is polynomial-time decidable for the class of chordal graphs. 
Note that a chordal graph has no induced cycle of length greater than three. Hence, a connected chordal graph \( G \) is either a tree, or it contains a triangle. If \( G \) is a tree, then 3-rs colourability of \( G \) can be tested using our 3-rs colourability testing algorithm for trees. So, we assume that \( G \) has a triangle, say \( (u,v,w) \). There are two types of triangles in \( G \) namely (i)~type-I: \( u,v \) and \( w \) are 3-plus vertices in \( G \), and (ii)~type-II: at least one of the vertices \( u,v,w \) is of degree two. Observation~\ref{obs:3-rs 3-plus triangle} deals with a graph (not necessarily chordal) which contains a type-I triangle; and Theorem~\ref{thm:3-rs get rid of triangle} deals with a graph (not necessarily chordal) which contains a type-II triangle.

\begin{observation}
Let \( G \) be a graph, and let \( (u,v,w) \) be a triangle in \( G \). If \( u,v \) and \( w \) are 3-plus vertices in \( G \), then \( G \) is not 3-rs colourable.
\qed
\label{obs:3-rs 3-plus triangle}
\end{observation}

\begin{theorem}
Let \( G \) be a graph, let \( (u,v,w) \) be a triangle in \( G \), and let \( w \) be a vertex of degree two in \( G \). Let \( G\bm{'} \) be the graph obtained from \( G-w \) by attaching two pendant vertices each at \( u \) and \( v \) (see Figure~\ref{fig:3-rs get rid of triangle}). Then, \( G \) is 3-rs colourable if and only if \( G\bm{'} \) is 3-rs colourable.
\label{thm:3-rs get rid of triangle}
\end{theorem}
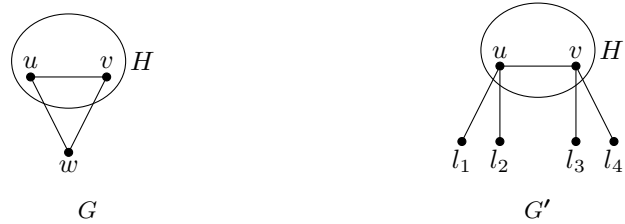
\begin{figure}[hbt]
\centering
\begin{subfigure}[b]{0.33\textwidth}
\centering
\begin{tikzpicture}[label distance=-2pt]
\tikzset{
dot/.style={draw,fill,circle,inner sep = 0pt,minimum size = 3pt},
vcolour/.style={draw,inner sep=1.5pt,font=\scriptsize,label distance=2pt},
subgraph/.style={draw,ellipse,minimum width=1.75cm,minimum height=2cm},
subgraphHoriz/.style={draw,ellipse,minimum width=1.5cm,minimum height=1.25cm},
}
\draw (0,0) node (u)[dot][label=above:\( u \)]{} --++(1,0) node (v)[dot][label=above:\( v \)]{};
\path (u) --node[subgraphHoriz][yshift=6pt,label=right:\( H \)]{} (v);
\draw (u) --++(0.5,-1) node[dot](w)[dot][label=below:\( w \)]{} -- (v);
\end{tikzpicture}
\caption*{\( G \)}
\end{subfigure}%
\begin{subfigure}[b]{0.33\textwidth}
\centering
\begin{tikzpicture}[label distance=-2pt]
\tikzset{
dot/.style={draw,fill,circle,inner sep = 0pt,minimum size = 3pt},
vcolour/.style={draw,inner sep=1.5pt,font=\scriptsize,label distance=2pt},
subgraph/.style={draw,ellipse,minimum width=1.75cm,minimum height=2cm},
subgraphHoriz/.style={draw,ellipse,minimum width=1.5cm,minimum height=1.25cm},
}
\draw (0,0) node (u)[dot][label=above:\( u \)]{} --++(1,0) node (v)[dot][label=above:\( v \)]{};
\path (u) --node[subgraphHoriz][yshift=6pt,label=right:\( H \)]{} (v);
\draw (u) --++(-0.5,-1) node[dot](l1)[dot][label=below:\( l_1 \)]{}
      (u) --++(0,-1) node[dot](l2)[dot][label=below:\( l_2 \)]{};
\draw (v) --++(0,-1) node[dot](l3)[dot][label=below:\( l_3 \)]{}
      (v) --++(0.5,-1) node[dot](l4)[dot][label=below:\( l_4 \)]{};

\end{tikzpicture}
\caption*{\( G\bm{'} \)}
\end{subfigure}%
\caption{Graphs \( G \) and \( G\bm{'} \)}
\label{fig:3-rs get rid of triangle}
\end{figure}

\begin{proof}
Suppose that \( G \) admits a 3-rs colouring. Then, \( G \) admits a 3-rs colouring \( f \) such that \( f(w)=2 \) (for instance, if a 3-rs colouring of \( G \) assigns colour 2 at \( v \), then \( \deg_G(v)=2 \) and hence swapping colours at \( v \) and \( w \) gives a 3-rs colouring of \( G \) that assigns colour 2 at \( w \)). Since \( f(w)=2 \), we have \( f(u),f(v)\in\{0,1\} \). Hence, the restriction of \( f \) to \( V(G-w) \) can be extended into a 3-rs colouring of \( G\bm{'} \) by assigning colour 2 on the newly added pendant vertices. Conversely, suppose that \( G\bm{'} \) admits a 3-rs colouring \( f\bm{'} \). Since \( u \) and \( v \) are 3-plus vertices in \( G\bm{'} \), \( f\bm{'}(u),f\bm{'}(v)\in\{0,1\} \) by Property~P1. Hence, \( f\bm{'} \) restricted to \( V(G-w) \) can be extended into a 3-rs colouring of \( G \) by assigning colour~2 at \( w \).
\end{proof}

\begin{theorem}
\textsc{3-RS Colourability} is in P for the class of chordal graphs.
\end{theorem}
\begin{proof}
We can test whether a given connected chordal graph \( G \) is 3-rs colourable as follows.\\

\noindent Case 1: \( G \) is a tree.\\
Then, 3-rs colourability of \( G \) can be tested using our 3-rs colourability testing algorithm for trees.\\[-5pt]

\noindent Case 2: \( G \) contains a triangle.\\
List all triangles in \( G \) in \( O(n^3) \) time. If \( G \) contains a type-I triangle, then \( G \) is not 3-rs colourable by Observation~\ref{obs:3-rs 3-plus triangle}. If \( G \) contains only type-II triangles, we can get rid of all triangles in \( G \) by repeated application of Theorem~\ref{thm:3-rs get rid of triangle} (see Figure~\ref{fig:T from G} for an example). The resultant graph will be a tree \( T \) on at most \( n+3{n\choose 3}=O(n^3) \) vertices. Thus, we can test 3-rs colourability of \( G \) in \( O(n^3) \) time by testing 3-rs colourability of \( T \).
\end{proof}

\begin{figure}[hbt]
\begin{tikzpicture}[label distance=-2pt,scale=0.65]
\tikzset{
dot/.style={draw,fill,circle,inner sep = 0pt,minimum size = 3pt},
vcolour/.style={draw,inner sep=1.5pt,font=\scriptsize,label distance=2pt},
subgraph/.style={draw,ellipse,minimum width=1.75cm,minimum height=2cm},
subgraphHoriz/.style={draw,ellipse,minimum width=1.5cm,minimum height=1.25cm},
}
\node at (0,0) (Col1)[inner sep=0pt]{};

\draw (Col1) node (d)[dot][label=above:\( x \)]{} --++(1,0) node (u1)[dot][label=above:\( u_1 \)]{} --++(1,0) node (v1)[dot][label=above:\( v_1 \)]{};
\draw (u1) --++(0.5,-1) node[dot](w1)[dot][label=below:\( w_1 \)]{} -- (v1);

\draw (v1) --++(1.5,0) node (u2)[dot][label=above:\( u_2 \)]{} --++(1,0) node (v2)[dot][label=above:\( v_2 \)]{};
\draw (u2) --++(0.5,-1) node[dot](w2)[dot][label=below:\( w_2 \)]{} -- (v2);

\path (Col1) --+(5,-0.6) node(from){};
\draw [-stealth,black,ultra thick] (from)--++(0.75,0);

\path (Col1) --+(6.5,0) node(Col2)[inner sep=0pt]{};

\draw (Col2) node (d)[dot][label=above:\( x \)]{} --++(1,0) node (u1)[dot][label=above:\( u_1 \)]{} --++(1,0) node (v1)[dot][label=above:\( v_1 \)]{};
\draw (u1) --++(-0.5,-1) node[dot](l1)[dot][label=below:\( l_1 \)]{}
      (u1) --++(0,-1) node[dot](l2)[dot][label=below:\( l_2 \)]{};
\draw (v1) --++(0,-1) node[dot](l3)[dot][label=below:\( l_3 \)]{}
      (v1) --++(0.5,-1) node[dot](l4)[dot][label=below:\( l_4 \)]{};

\draw (v1) --++(1.5,0) node (u2)[dot][label=above:\( u_2 \)]{} --++(1,0) node (v2)[dot][label=above:\( v_2 \)]{};
\draw (u2) --++(0.5,-1) node[dot](w2)[dot][label=below:\( w_2 \)]{} -- (v2);

\path (Col2) --+(5,-0.6) node(from){};
\draw [-stealth,black,ultra thick] (from)--++(0.75,0);

\path (Col2) --+(6.5,0) node(Col3)[inner sep=0pt]{};

\draw (Col3) node (d)[dot][label=above:\( x \)]{} --++(1,0) node (u1)[dot][label=above:\( u_1 \)]{} --++(1,0) node (v1)[dot][label=above:\( v_1 \)]{};
\draw (u1) --++(-0.5,-1) node[dot](l1)[dot][label=below:\( l_1 \)]{}
      (u1) --++(0,-1) node[dot](l2)[dot][label=below:\( l_2 \)]{};
\draw (v1) --++(0,-1) node[dot](l3)[dot][label=below:\( l_3 \)]{}
      (v1) --++(0.5,-1) node[dot](l4)[dot][label=below:\( l_4 \)]{};

\draw (v1) --++(1.5,0) node (u2)[dot][label=above:\( u_2 \)]{} --++(1,0) node (v2)[dot][label=above:\( v_2 \)]{};
\draw (u2) --++(-0.5,-1) node[dot](l1)[dot][label=below:\( l_5 \)]{}
      (u2) --++(0,-1) node[dot](l2)[dot][label=below:\( l_6 \)]{};
\draw (v2) --++(0,-1) node[dot](l3)[dot][label=below:\( l_7 \)]{}
      (v2) --++(0.5,-1) node[dot](l4)[dot][label=below:\( l_8 \)]{};

\end{tikzpicture}
\caption{Construction of a tree \( T \) from a chordal graph \( G \) such that \( \chi_{rs}(T)\leq 3 \) if and only if \( \chi_{rs}(G)\leq 3 \)}
\label{fig:T from G}
\end{figure}
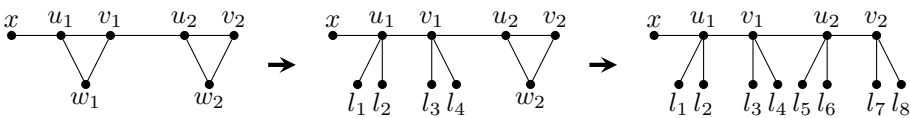 

\section{Related Colouring Variants}\label{sec:related}
Star colouring and ordered colouring are colouring variants closely related to restricted star colouring. A \emph{\( k \)-ordered colouring} of a graph \( G \) is a function \( f:V(G)\to \{0,1,\dots,k-1\} \) such that (i)~\( f(x)\neq f(y) \) for every edge \( xy \) of \( G \), and (ii)~every non-trivial path in \( G \) with the same colour at its endpoints contains a vertex of higher colour (i.e., every path \( P_{q+1} \) with endpoints coloured \( i \) contains a vertex of colour \( j > i \)) \cite{katchalski}. The ordered chromatic number \( \chi_o(G) \) is the least integer \( k \) such that \( G \) admits a \( k \)-ordered colouring. Analogous to the problem \textsc{\( k \)-RS Colourability} (resp.\ \textsc{RS Colourability}), we can define problems \textsc{\( k \)-Star Colourability} and \textsc{\( k \)-Ordered Colourability} (resp.\ \textsc{Star Colourability} and \textsc{Ordered Colourability}). 

For every graph \( G \), \( \chi_s(G)\leq \chi_{rs}(G) \leq \chi_o(G) \) \cite{karpas}. It is known that the three parameters can be arbitrarily apart \cite{karpas}. On the other hand, all three parameters are equal for cographs because \( \chi_s(G)=\text{tw}(G)+1=\chi_o(G) \) \cite{scheffler,lyons} where tw\( (G) \) denotes the treewidth of \( G \). We prove that the star chromatic number equals the ordered chromatic number for co-bipartite graphs as well.
\begin{theorem}  
For a co-bipartite graph \( G \), \( \chi_s(G)=\chi_{rs}(G)=\chi_o(G) \).
\end{theorem}  
\begin{proof}
Let \( G \) be a co-bipartite graph whose vertex set is partitioned into two cliques \( A \) and \( B \). Let \( k=\chi_s(G) \), and let \( f \) be a \( k \)-star colouring of \( G \). To prove the theorem, it suffices to show that \( G \) admits a \( k \)-ordered colouring. Clearly, each colour class under \( f \) contains at most one vertex from \( A \) and at most one vertex from \( B \). 
Let \( \{U_0,U_1,\dots,U_{k-1}\} \) be the set of colour classes under \( f \). Let \( t \) be the number of colour classes with cardinality two. We assume that \( t\neq 0 \) (if \( t=0 \), then \( k=n \), and \( f \) itself is a \( k \)-ordered colouring of \( G \)). W.l.o.g., we assume that \( |U_i|=2 \) for \( 0\leq i\leq t-1 \), and \( |U_i|=1 \) for \( t\leq i\leq k-1 \). Observe that if \( a_i\in U_i\cap A \), \( b_i\in U_i\cap B \), \( a_j\in U_j\cap A \) and \( b_j\in U_j\cap B \), then \( a_ib_j\notin E(G) \) and \( a_jb_i\notin E(G) \)
(if not, the path \( a_j,a_i,b_j,b_i \) or the path \( a_i,a_j,b_i,b_j \) is a \( P_4 \) in \( G \) bicoloured by \( f \)). Therefore, there are no edges in \( G \) between \( A\cap\bigcup_{i=0}^{t-1}U_i \) and \( B\cap\bigcup_{i=0}^{t-1}U_i \). We claim that \( h:V(G)\to \{0,1,\dots,k-1\} \) defined as \( h(v)=i \) for all \( v\!\in\!U_i \) is a \( k \)-ordered colouring of \( G \). Suppose that \( u \) and \( v \) are distinct vertices in \( G \) with \( h(u)=h(v)=j \) (where \( 0\leq j\leq k-1 \)). Then, \( j\leq t-1 \). W.l.o.g., we assume that \( u\in A \) and \( v\in B \). Then, \( u\in A\cap\bigcup_{i=0}^{t-1}U_i \) and \( v\in B\cap\bigcup_{i=0}^{t-1}U_i \). Hence, every \( u,v \)-path \( Q \) in \( G \) must contain a vertex \( w_Q \) from \( U_t\cup\dots\cup U_{k-1} \). Since \( f(w_Q)>f(u) \), and \( u,v \) and \( Q \) are arbitrary, \( h \) is indeed a \( k \)-ordered colouring of \( G \).
\end{proof}
Since \textsc{\( k \)-Ordered Colourability} is in P for every \( k \), and \textsc{Ordered Colourability} is NP-complete for co-bipartite graphs \cite{bodlaender}, we have the following corollary.
\begin{corollary}
For the class of co-bipartite graphs, problems \textsc{\( k \)-RS Colourability} and \textsc{\( k \)-Star Colourability} are in P for all \( k \), whereas problems \textsc{RS Colourability} and \textsc{Star Colourability} are NP-complete.
\end{corollary}
\section{Conclusion}\label{sec:conclusion}
It is known that deciding whether a planar bipartite graph admits a 3-star colouring is NP-complete \cite{albertson}. We prove that deciding whether a subcubic planar bipartite graph of arbitrarily large girth admits a 3-restricted star colouring is NP-complete. In addition, we prove that it is NP-complete to test whether a 3-star colourable graph admits a 3-restricted star colouring (see Theorem~\ref{thm:3-rsc planar bipartite etc}). Karpas \etal.~\cite{karpas} produced an \( O(n^\frac{1}{2}) \) approximation algorithm for the optimization problem of restricted star colouring on a 2-degenerate bipartite graph with the minimum number of colours. We prove that this optimization problem is NP-hard to approximate within \( n^{\frac{1}{3}-\epsilon} \) for all \( \epsilon>0 \). 
For the class of co-bipartite graphs, \textsc{RS Colourability} is NP-complete, but \textsc{\( k \)-RS Colourabality} is in P for all \( k\in\mathbb{N} \). 
We present (i)~an \( O(n) \)-time algorithm to test 3-rs colourability of trees, and (ii)~an \( O(n^3) \)-time algorithm to test 3-rs colourability of chordal graphs. The complexity of \textsc{RS Colourability} in the class of chordal graphs remains open.\\

\section*{Declaration of competing interest}
The authors declare that they have no known competing financial interests or personal relationships that could have appeared to influence the work reported in this paper.

\section*{Acknowledgements}
The suggestions of two anonymous referees improved the presentation of this paper. The first author is supported by SERB (DST), MATRICS scheme MTR/2018/000086.

\section*{Appendix A. Supplementary data}
Supplementary material related to this article can be found online at \url{https://doi.org/10.1016/j.dam.2021.05.015}.
%


\end{document}